\documentclass[a4paper]{amsart}
\usepackage{etex}
\usepackage{enumerate}
\usepackage[english]{babel}

\usepackage{graphicx}
\usepackage{multicol}
\usepackage[bottom]{footmisc}
\usepackage[export]{adjustbox}
\usepackage{functan}
\usepackage[latin1]{inputenc}
\usepackage{cancel}
\usepackage{calligra}
\usepackage{colortbl}
\usepackage{multirow}
\usepackage{enumerate}
\usepackage{varioref}
\usepackage{booktabs}
\usepackage{amscd}
\usepackage{color}
\usepackage{colortbl}
\usepackage{pstricks}
\usepackage{pst-all}
\usepackage{mparhack}
\usepackage{amssymb}
\usepackage{amsmath}
\usepackage{dsfont}
\usepackage{mathrsfs}
\usepackage{amsfonts}
\usepackage{amssymb}
\usepackage[mathcal]{euscript}
\usepackage[all,cmtip]{xy}
\usepackage{amssymb}
\usepackage{amsthm}
\theoremstyle{definition}
\newtheorem{definition}{Definition}[section]
\newtheorem{example}[definition]{Example}
\newtheorem{remark}[definition]{Remark}

\theoremstyle{plain}
\newtheorem*{ThmA}{Theorem A}
\newtheorem*{ThmB}{Theorem B}
\newtheorem*{ThmC}{Theorem C}

\newtheorem{theorem}[definition]{Theorem}
\newtheorem{proposition}[definition]{Proposition}
\newtheorem{lemma}[definition]{Lemma}
\newtheorem{corollary}[definition]{Corollary}

\numberwithin{equation}{section}

\def \alt96 {`}
\def \RN {\mathds{R}^N}
\def \R {\mathds{R}}

\def \loc {\mathrm{loc}}

\def \N {\mathds{N}}

\def \G {\mathbb{G}}

\def \dela {{\delta_\lambda}}

\def \LL {{\mathcal{L}}}

\def \d {\mathrm{d}}
\def \de {\partial}

\def \emptyset {\varnothing}

\def \Lie {\mathrm{Lie}}

\def \longto {\longrightarrow}
\def \txt {\textstyle}

\def \LL {\mathcal{L}}

\newcommand{\meanint}{-\!\!\!\!\!\!\!\int}

\begin{document}
  \author{Stefano Biagi}
 \address{Stefano Biagi: Dipartimento di Matematica,
 Politecnico di Milano, Via Bonardi 9, I-20133 Milano, Italy.}
 \email{stefano.biagi@polimi.it}
  \author{Andrea Bonfiglioli}
 \address{Andrea Bonfiglioli: Dipartimento di Matematica,
 Alma Mater Studiorum - Università di Bologna,
 Piazza Porta San Donato 5, I-40126 Bologna, Italy.}
 \email{andrea.bonfiglioli6@unibo.it}
  \author{Marco Bramanti}
 \address{Marco Bramanti: Dipartimento di Matematica, Politecnico di Milano,
 Via Bonardi 9, I-20133 Milano, Italy.}
 \email{marco.bramanti@polimi.it}

 \title[Global estimates for the fundamental solution]{Global estimates for the fundamental solution\\ of homogeneous H\"ormander operators}

\begin{abstract}
 Let $\mathcal{L}=\sum_{j=1}^{m}X_{j}^{2}$ be a H\"{o}rmander sum of squares
 of vector fields in $\mathbb{R}^{n}$, where any $X_{j}$ is homogeneous of
 degree $1$ with respect to a family of non-isotropic dilations in $\mathbb{R}^{n}$.
 Then $\mathcal{L}$ is known to admit a global fundamental solution $\Gamma (x;y)$,
 that can be represented as the integral of a fundamental
 solution of a sublaplacian operator on a lifting space 
 $\mathbb{R}^{n}\times \mathbb{R}^{p}$, equipped with a Carnot group structure. The aim
 of this paper is to prove global pointwise (upper and lower) estimates of $\Gamma $,
 in terms of the Carnot-Carath\'{e}odory distance induced by 
 $X=\{X_{1},\ldots ,X_{m}\}$ on $\mathbb{R}^{n}$, as well as global pointwise
 (upper) estimates for the $X$-derivatives of any order of $\Gamma $,
 together with suitable integral representations of these derivatives.
 The least dimensional case $n=2$ presents several peculiarities which are also investigated.
 Applications to the potential theory for $\mathcal{L}$ and to singular-integral 
 estimates for the kernel $X_{i}X_{j}\Gamma $ are also provided.
 Finally, most of the results about $\Gamma$ are extended to the case of
 H\"{o}rmander operators with drift $\sum_{j=1}^{m}X_{j}^{2}+X_{0}$, where $X_{0}$ is
 $2$-homogeneous and $X_{1},...,X_{m}$ are $1$-homogeneous.
\end{abstract}
\maketitle

\section{Introduction and main results}\label{sec:introductionMain}
 In this paper we consider a class of linear
 second order partial differential operators
\begin{equation*}
 \mathcal{L}=X_{1}^{2}+\cdots +X_{m}^{2},
\end{equation*}
 where $X=\{X_{1},\ldots ,X_{m}\}$ is a set of H\"{o}rmander vector fields in
 $\mathbb{R}^{n}$ ($n\geq 2$), and any element of $X$ is homogeneous of
 degree $1$ with respect to a family $\{\delta_\lambda\}_{\lambda>0}$ of non-isotropic dilations. (Precise
 definitions will be given below). Our main aim is to prove global pointwise
 (upper and lower) estimates, in terms of the Carnot-Carath\'{e}odory
 distance $d_{X}$ induced by $X$ on $\mathbb{R}^{n}$, of a suitable positive global
 fundamental solution $\Gamma (x;y)$ for $\mathcal{L}$, and global
 pointwise (upper) estimates for the $X$-derivatives of any order of $\Gamma $
 (see Theorem \ref{th.teoremone}). The least dimensional case $n=2$ seems to
 be particularly delicate, as fundamental solutions near the diagonal may
 exhibit different behavior from point to point (logarithmic or power-like).
 Furthermore, applications to the potential theory for $\mathcal{L}$
 (Section \ref{sec:potentialtheory}) and to singular integral estimates for the kernel
 $X_{i}X_{j}\Gamma $ (Section \ref{sec:meanvalue}) are provided.\medskip

 Most profound ideas in the study of the geometrical subelliptic analysis of
 general H\"{o}rmander sums of squares of vector fields $L=\sum_{j=1}^{m}Y_{j}^{2}$
 (of which our $\mathcal{L}$ is a particular case)
 are contained in the seminal papers by H\"{o}rmander \cite{Hormander}, by
 Folland \cite{Fo2}, by Rothschild and Stein \cite{RothschildStein}, by
 Nagel, Stein and Wainger \cite{NSW}, by S\'{a}nchez-Calle \cite{SC}. A
 paramount tool in the analysis of $L$ is Rothschild-Stein's lifting
 technique, which locally approximates $L$ with a sublaplacian operator on
 some higher dimensional free Carnot group.

 Broadly speaking, in the cited papers,
 a large part of the most relevant theory for $L$ (geometric analysis,
 function spaces, subelliptic estimates, etc.) was ultimately settled under
 its \emph{local} form. On the other hand, a \emph{global} theory (to which
 we are interested) is developed by Folland \cite{Fo2} in the special case of
 homogeneous left invariant H\"{o}rmander operators on homogeneous groups,
 but is inevitably missing in the general case of $L$, since one cannot
 expect that $L$ be equipped with a global fundamental solution $\Gamma $
 defined out of the diagonal of $\mathbb{R}^{n}\times \mathbb{R}^{n}$,
 without further assumptions on $L$.
 Analogously, for such general operators $L$, geometrically meaningful
 results involving CC-balls $B_{X}(x,r)$ are mainly available when the radius
 $r$ is sufficiently small and the center $x$ is located in some fixed
 compact set. The locality of Rothschild-Stein's lifting technique is also an
 implicit (hardly avoidable) obstruction to a global theory for $L$.

 Hence, if we aim to give global estimates for a globally defined fundamental
 solution, some further assumptions on the operator must be made. Roughly
 put, the homogeneity of $L$ with respect to a family of dilations as in
 \eqref{intro.dela} can help recovering large $r$'s, whereas the invariance
 of $L$ with respect to a family of Lie-group translations is of aid in
 dealing with arbitrary $x$'s. As is well known, a simultaneous
 homogeneity/translation-invariance boils down to the case when $L$ is a
 sublaplacian operator on a Carnot group $\mathbb{G}$, for which a global
 fundamental solution $\Gamma_{\mathbb{G}}$ is known to exist after
 Folland's paper \cite{Fo2}, and the CC-ball $B_{X}(x,r)$ is just the left
 translation by $x$ of the $\delta_{r}$-dilated of the ball $B_{X}(0,1)$, so
 that the underlying subelliptic geometry is much simpler. One can easily say
 that the worthwhile results of the subelliptic analysis for the
 sublaplacians in the Carnot group case are nowadays well
 established.\medskip

 A convenient framework, more general than the Carnot group setting, is the
 one considered in this paper, that is the case when $\mathcal{L}$ is
 $\delta_{\lambda }$-homogeneous of degree $2$, but \emph{not left-invariant}.
\begin{example}
 (1).\,\,A first instance of
 such operators are the Grushin-type PDO's in $\mathbb{R}^{2}$
\begin{equation*}
 (\partial_{x_{1}})^{2}+(x_{1}^{k}\,\partial_{x_{2}})^{2}\qquad \text{(with $k\in \mathbb{N}$)},
\end{equation*}
 associated with the dilations $(\lambda x_{1},\lambda ^{k+1}x_{2})$. These
 are not left-invariant PDO's on any Lie group on $\mathbb{R}^{2}$ (when $k\geq 1$)
 as $x_{1}^{k}\partial_{x_{2}}$ vanishes when $x_{1}=0$ without being the
 null vector field.\medskip

 (2).\,\,Another class of operators to which our theory applies is given by
  $$ (\partial_{x_1})^2+\Big(x_{1}\partial_{x_{2}}+x_{2}\partial_{x_{3}}+\ldots +x_{n-1}\partial_{x_{n}}\Big)^{2}\quad \text{in $\R^n$,}$$
 which is $\dela$-homogeneous of degree $2$ (but not left invariant on $\R^n$, for the same reasons as in (1)) with respect to the dilations
 $\delta_{\lambda}(x)  =(\lambda x_{1},\lambda^{2}x_{2},\cdots,\lambda^{n}x_{n})$.\medskip

 (3).\,\, A further example is the operator
 \begin{equation*}
  X_1^2+X_2^2=(\de_{x_1})^2+\Big(x_1\,\de_{x_2}+x_1^2\,\de_{x_3}\Big)^2\quad \text{on $\R^3$,}
   \end{equation*}
  which is homogeneous of degree $2$ (but not left invariant on $\R^3$) with respect to
\begin{equation*}
 \delta_\lambda(x) =     (\lambda x_1,\lambda^2 x_2,\lambda^3 x_3).
\end{equation*}
 The Lie algebra generated by $X_1,X_2$ is the Lie algebra of the so-called
 Engel group on $\R^4$. \medskip

 (4).\,\,Finally, the operator
\begin{equation*}
  (\de_{x_1})^2 + \Big(x_1\,\de_{x_2}+x_1^2\,\de_{x_3}+\cdots+x_1^{n-1}\,\de_{x_n}\Big)^2\quad\text{on $\R^n$}
\end{equation*}
 is homogeneous of degree $2$ with respect to the same dilations as in (2), but not left invariant on $\R^n$.
\end{example}
 Let us now precisely fix our assumptions. We assume that $X=\{X_{1},\ldots,X_{m}\}$ fulfils the following conditions (H.1), (H.2), (H.3):
\begin{itemize}
\item[\textbf{(H.1)}]
 there exists a family of (non-isotropic) dilations $\{\delta_{\lambda }\}_{\lambda >0}$ of the form
\begin{equation}\label{intro.dela}
 \delta_{\lambda }:\mathbb{R}^{n}\longrightarrow \mathbb{R}^{n}\qquad \delta_{\lambda }(x)=(\lambda ^{\sigma_{1}}x_{1},\ldots ,\lambda ^{\sigma_{n}}x_{n}),
\end{equation}
 where $1=\sigma_{1}\leq \cdots \leq \sigma_{n}$, such that $X_{1},\ldots ,X_{m}$ are $\delta_{\lambda }$-homogeneous of degree $1$,
 i.e.,
\begin{equation*}
 X_{j}(f\circ \delta_{\lambda })=\lambda \,(X_{j}f)\circ \delta_{\lambda},
 \quad \text{for every $\lambda >0$, $f\in C^{\infty}(\mathbb{R}^{n})$ and $j=1,\ldots,m$.}
\end{equation*}
\end{itemize}
 \noindent In what follows, we denote by
\begin{equation}\label{eq.defq}
 q:=\textstyle\sum_{j=1}^{m}\sigma_{j}
\end{equation}
 the so-called $\delta_{\lambda }$-homogeneous dimension of $(\mathbb{R}^{n},\delta_{\lambda })$.\medskip
\begin{itemize}
 \item[\textbf{(H.2)}] $X_{1},\ldots ,X_{m}$
 are linearly independent\footnote{The linear independence of the $X_i$'s
 is meant with respect to the vector space of the smooth vector fields on $\R^n$;
 this must not be confused with the linear independence of the
 vectors $X_1(x),\ldots,X_m(x)$ in $\R^n$
 (when $x\in\R^n$):
 the latter is sufficient but not necessary to the former linear independence.
 Thus, $X_1=\de_{x_1}$ and $X_2=x_1\,\de_{x_2}$ are linearly independent vector fields, even if $X_1(0,x_2)\equiv (1,0)$ and $X_2(0,x_2)\equiv(0,0)$
 are dependent vectors of $\R^2$.}
 and satisfy H\"{o}rmander's rank condition at $0$, i.e.,
\begin{equation*}
 \dim \big\{Y(0):Y\in \mathrm{Lie}(X)\big\}=n.
\end{equation*}
 Here $\mathrm{Lie}(X)$ stands for the smallest Lie subalgebra of $\mathcal{X}(\mathbb{R}^{n})$
 containing $X$, where $\mathcal{X}(\mathbb{R}^{n})$ is
 the Lie algebra of all the smooth vector fields on $\mathbb{R}^{n}.$
\end{itemize}
 \noindent Finally, we make the following dimensional assumptions:
\begin{itemize}
\item[\textbf{(H.3)}] we require that $q>2$ and
\begin{equation}\label{ipotesidimensinale}
 N:=\mathrm{dim}(\Lie\{X\})>n\geq 2.
\end{equation}
\end{itemize}
 Throughout the paper we let $p:=N-n\geq 1$ and we denote the points of $\mathbb{R}^{N}\equiv \mathbb{R}^{n}\times \mathbb{R}^{p}$ by
 $$(x,\xi ),\quad \text{with $x\in \R^n$ and $\xi\in\R^p$.}$$
\begin{remark}[Some consequences of (H.1)-to-(H.3)] \label{rem.assumptionsH}
 Assumptions (H.1) and (H.2) together imply that the $\sigma_i$'s in \eqref{intro.dela} are integers, and that $\Lie\{X\}$
 is nilpotent of step $\sigma_n$ (see e.g., \cite{BB}).

 We observe that, by (H.1), the validity of H\"{o}rmander's rank condition at
 $0$ (condition (H2)) implies its validity at any other point $x\in \mathbb{R}^{n}$
 (this is proved in Remark \ref{hormandereverywhere}).
 Thus, the H\"{o}rmander operator $\mathcal{L}=\sum_{j=1}^m X_{j}^{2}$
 is $C^{\infty }$-hypoelliptic on every open subset of $\mathbb{R}^{n}$.

 The assumption $q>2$ in (H3) is harmless since the case $q=2$ only happens
 when $\mathcal{L}$ is a strictly-elliptic constant-coefficient operator in $\mathbb{R}^{2}$
  (which is also left invariant on $(\R^2,+)$), a well-known situation we are not interested in. The
 assumption $N>n$ in (H3) is rather harmless as well, for the following reason. Since $X$ is
 a H\"{o}rmander set, $N$ defined in \eqref{ipotesidimensinale} cannot be $<n$;
 on the other hand, when $N=n$, it follows from a general result contained
 in \cite{BBCCM} (and exploiting the $\delta_{\lambda }$-homogeneity of the
 elements of $X$) that $\mathcal{L}$ is necessarily a sublaplacian on a homogeneous
 Carnot group on $\mathbb{R}^{n}$, a well-studied situation in which the
 results of this paper are already known (see \cite{Fo2}).
\end{remark}

 Very recently, under assumption \eqref{ipotesidimensinale}, the existence of
 a global fundamental solution $\Gamma $ for $\delta_{\lambda }$-homogenous
 $\mathcal{L}$'s has been obtained in \cite{BB} via a lifting procedure due to
 Folland \cite{Folland}, a global simplified version of Rothschild-Stein's
 lifting. Folland's technique consists in lifting $\mathcal{L}$ directly to a
 sublaplacian $\mathcal{L}_{\mathbb{G}}$ on a (strictly higher dimensional)
 Carnot group $(\mathbb{G},\ast)$ (which is not necessarily free). After an appropriate change of variable (performed in \cite{BB}),
 one can suppose that the manifold of $\mathbb{G}$ takes
 the product form $\mathbb{G}=\mathbb{R}_{x}^{n}\times \mathbb{R}_{\xi}^{p}$,
 with $p=N-n$. Under assumption \eqref{ipotesidimensinale}, this $p$ is at
 least $1$. We are now going to review this result, which also gives an integral
 representation for $\Gamma $; this representation will be used throughout
 the paper.

 In what follows, we refer to \cite[\S 1.4]{BLUlibro} for the
 notions of sublaplacian and of homogeneous Carnot group, with the sole difference that we do not
 require the exponents of the associated dilations $D_\lambda$ to be increasingly ordered; this is because
 we have already performed a change of variable on $\R^N\equiv \R^n\times \R^p$, which separates the unlifted variables $x$ from the
 lifting variables $\xi$.
 We also implicitly
 invoke Folland's result \cite{Fo2} on the existence of a global fundamental
 solution for any sublaplacian on any Carnot group.
\begin{ThmA}[{{\protect\cite[Theorems 3.2 and 4.4]{BB}}}] \label{theoA}
 Assume that $X=\{X_{1},\ldots ,X_{m}\}$ satisfies \emph{(H.1)}-to-\emph{(H.3)},
 of which we inherit the notation. Then the following facts hold:\medskip

\emph{(1).} There exist a homogeneous Carnot group $\mathbb{G}=
 (\mathbb{R}^{N},\ast ,D_{\lambda })$ of homogeneous dimension $Q>q$ and a system
 $\{\widetilde{X}_{1},\ldots ,\widetilde{X}_{m}\}$ of Lie-generators of
 $\mathrm{Lie}(\mathbb{G})$ such that $\widetilde{X}_{i}$ is a lifting of $X_{i}$ for
 every $i=1,\ldots ,m$; by this we mean that
\begin{equation}\label{lifting}
 \widetilde{X}_{i}(x,\xi )=X_{i}(x)+R_{i}(x,\xi ),
\end{equation}
 where $R_{i}(x,\xi )$ is a smooth vector field operating only in the
 variable $\xi \in \mathbb{R}^{p}$, with coefficients possibly depending on
 $(x,\xi )$. In particular, the $\widetilde{X}_{i}$'s are
 $D_{\lambda}$-homogeneous of degree $1$. \medskip

\emph{(2).} If $\widetilde{\Gamma }$ is the \emph{(}unique\emph{)}
 smooth fundamental solution of $\sum_{i=1}^{m}\widetilde{X}_{i}^{2}$
 vanishing at infinity con\-struc\-ted
 in \cite{Fo2},
 then $\mathcal{L}$ admits a global fundamental solution
 $\Gamma (x;y)$ under the form
\begin{equation}\label{sec.one:mainThm_defGamma}
 \Gamma (x;y):=\int_{\mathbb{R}^{p}}\widetilde{\Gamma}
 \big((x,0);(y,\eta )\big)\, \d\eta \qquad (\text{for $x\neq y$
 in $\mathbb{R}^{n}$}).
\end{equation}
 By saying that $\Gamma $ is a global fundamental solution of $\mathcal{L}$
 we mean that\footnote{Note that $\mathcal{L}$ is formally selfadjoint on test functions, due to
 simple arguments based on (H.1).} the map $y\mapsto \Gamma (x;y)$
 is locally integrable on $\mathbb{R}^{n}$ and that
\begin{equation*}
 \int_{\mathbb{R}^{n}}\Gamma (x;y)\,\mathcal{L}\varphi (y)
 \,\d y=-\varphi (x)\qquad \text{for every $\varphi \in
 C_{0}^{\infty }(\mathbb{R}^{n})$ and every $x\in \mathbb{R}^{n}$.}
\end{equation*}
 Furthermore, setting $\Gamma_\G(\cdot):=\widetilde{\Gamma }(0;\cdot)$, the
 integrand in \eqref{sec.one:mainThm_defGamma} takes the convolution form
\begin{equation}\label{sec.one:mainThm_defGamma2}
 \widetilde{\Gamma }\big((x,0);(y,\eta )\big)=\Gamma_\G
 \Big((x,0)^{-1}\ast (y,\eta )\Big),
 \end{equation}
 valid for any $x\neq y$, so that \eqref{sec.one:mainThm_defGamma}
 becomes
\begin{equation}\label{sec.one:mainThm_defGamma22222}
 \Gamma (x;y)=\int_{\mathbb{R}^{p}}\Gamma_\G\Big(
 (x,0)^{-1}\ast (y,\eta )\Big)\,\d\eta \qquad (\text{for $x\neq
 y$ in $\mathbb{R}^{n}$}).
\end{equation}

\emph{(3).} $\Gamma $ enjoys further properties: it is smooth out of the
 diagonal; it is symmetric in $x,y$; it is strictly positive; it is
 locally integrable on $\mathbb{R}^{n}\times \mathbb{R}^{n}$; it vanishes
 when $x$ or $y$ go to infinity; it is jointly homogeneous of
 degree $2-q<0$, i.e.,
\begin{equation}\label{sec.one:mainThm_defGamma3}
 \Gamma \big(\delta_{\lambda }(x);\delta_{\lambda }(y)\big)
 =\lambda ^{2-q}\,\Gamma (x,y),\qquad x\neq y,\,\,\lambda >0.
\end{equation}
\end{ThmA}
\bigskip

 Once we have uniquely defined the global fundamental solution $\Gamma$ for $\LL$ we are interested in estimating, 
 our main results are contained in the following Theorem \ref{th.teoremone}, collecting the
 content of Theorems \ref{thm.estimatesDERIVgamma}, \ref{th.stimebasso}, \ref{thm.derivn2}, \ref{thm.estimateGamman2}, \ref{thm.estimateGamman2bis}
 and Lemma \ref{Thm 0} of the paper. Here and throughout, by `structural constant' we mean a constant only depending on 
 the objects introduced in the axioms (H.1)-to-(H.3) (like $X$, the $\sigma_i$'s, $q,n,N$, etc.)
 or other fixed parameters (usually explicitly declared).

\begin{theorem}\label{th.teoremone}
 Let $\mathcal{L}=\sum_{j=1}^{m}X_{j}^{2}$ satisfy
 assumptions \emph{(H1)-(H2)-(H3)}, and let $\Gamma$ and $\Gamma_\G$ be as in
 Theorem A. Then the following facts hold.\medskip

\emph{(I).}
 For any $s,t\geq 1$, and any choice of $i_{1},\ldots,i_{s},j_{1},\ldots ,j_{t}\in \{1,\ldots ,m\}$, we have the following
 representation formulas for the $X$-derivatives of $\Gamma $
 \emph{(}holding true for $x\neq y$ in $\mathbb{R}^{n}$\emph{)}:
\begin{align*}
 & X_{i_{1}}^{y}\cdots X_{i_{s}}^{y}\big(\Gamma (x;\cdot )\big)(y)=
 \int_{\mathbb{R}^{p}}\Big(\widetilde{X}_{i_{1}}\cdots \widetilde{X}_{i_{s}}\Gamma_{\mathbb{G}}\Big)
 \Big((x,0)^{-1}\ast (y,\eta )\Big)\,\d\eta \,; \\[0.2cm]
 & X_{j_{1}}^{x}\cdots X_{j_{t}}^{x}\big(\Gamma (\cdot ;y)\big)(x)=
 \int_{\mathbb{R}^{p}}\Big(\widetilde{X}_{j_{1}}\cdots \widetilde{X}_{j_{t}}
 \Gamma_{\mathbb{G}}\Big)\Big((y,0)^{-1}\ast (x,\eta )\Big)\,\d\eta \,; \\[0.2cm]
 & X_{j_{1}}^{x}\cdots X_{j_{t}}^{x}X_{i_{1}}^{y}\cdots X_{i_{s}}^{y}\Gamma(x;y) \\
 & \qquad \qquad =\int_{\mathbb{R}^{p}}\bigg(\widetilde{X}_{j_{1}}\cdots
 \widetilde{X}_{j_{t}}\Big(\big(\widetilde{X}_{i_{1}}\cdots
 \widetilde{X}_{i_{s}}\Gamma_{\mathbb{G}}\big)\circ \iota \Big)\bigg)\Big((y,0)^{-1}\ast
 (x,\eta )\Big)\,\d\eta\,.
\end{align*}
 Here $\iota $ denotes the inversion map of the Lie group $\mathbb{G}$.\medskip

\emph{(II).}
 For any integer $r\geq 1$ there exists $C_{r}>0$ such that
\begin{equation*}
 \Big\vert Z_{1}\cdots Z_{r}\Gamma (x;y)\Big\vert\leq C_{r}\,\frac{
 d_{X}(x,y)^{2-r}}{\big\vert B_{X}(x,d_{X}(x,y))\big\vert},
\end{equation*}
 for any $x,y\in \mathbb{R}^{n}$ \emph{(}with $x\neq y$\emph{)} and any
 choice of $Z_{1},\ldots ,Z_{r}\in \big\{X_{1}^{x},\ldots,
 X_{m}^{x},X_{1}^{y},\ldots ,X_{m}^{y}\big\}$. In particular,
 for every fixed $x\in\R^n$ we have
 $$
  \lim_{|y|\to\infty}Z_{1}\cdots Z_{r}\Gamma (x;y) = 0.
 $$

\emph{(III).}
 Suppose that $n>2$. Then one has
\begin{equation*}
 C^{-1}\frac{d_{X}(x,y)^{2}}{\big|B_{X}\big(x,d_{X}(x,y)\big)\big|}\leq
 \Gamma (x;y)\leq C\,\frac{d_{X}(x,y)^{2}}{\big\vert B_{X}(x,d_{X}(x,y))\big\vert},
\end{equation*}
 for any $x,y\in \mathbb{R}^{n}$ \emph{(}with $x\neq y$\emph{)}. Here $C\geq 1 $ is a structural constant.\medskip

\emph{(IV).}
 Suppose that $n=2$. For every compact set $K\subseteq \mathbb{R}^{n}$
 there exist structural constants $c_{1},c_{2}>0$ and real numbers
 $R_{1},R_{2}>0$ \emph{(}all depending on $K$\emph{)} 
 such that
\begin{equation*}
 c_{1}\,\log \Big(\frac{R_{1}}{d_{X}(x,y)}\Big)\leq \Gamma (x;y)\leq c_{2}\,
 \frac{d_{X}(x,y)^{2}}{\big|B_{X}(x,d_{X}(x,y))\big|}\cdot \log \Big(\frac{R_{2}}{d_{X}(x,y)}\Big),
\end{equation*}
 uniformly for $x\neq y$ in $K$. Moreover,
 for every fixed pole $x\in \mathbb{R}^{n}$,
 there exist constants $\gamma_{1}(x),\gamma_{2}(x)>0$ and $0<\varepsilon(x)<1$ such that
\begin{equation*}
 \gamma_{1}(x)\,F(x,y)\leq \Gamma (x;y)\leq \gamma_{2}(x)\,F(x,y),
\end{equation*}
 for any $y$ such that $0<d_{X}(x,y)<\varepsilon (x)$, where
\begin{equation*}
 F(x,y)=
\begin{cases}
 \log \left( \dfrac{1}{d_{X}(x,y)}\right) & \text{if $f_{2}(x)>0$,} \\[0.4cm]
 \dfrac{d_{X}(x,y)^{2}}{\big|B_{X}(x,d_{X}(x,y))\big|} & \text{if $f_{2}(x)=0$.}
\end{cases}
\end{equation*}
 Here $f_{2}$ is the nonnegative function which will be defined in Theorem B.
 In the case $f_{2}(x)=0$, the estimate of $\Gamma (x;y)$
 holds true with $\varepsilon (x)=1/2$ and $\gamma_{1}(x)$ independent of $x$;
 in this case, $F(x,y)$ diverges like $d_{X}(x,y)^{2-k}$, for some $k\in \{3,\ldots ,q\}$ which depends on $x$.\medskip

\emph{(V).}
 In particular, for any $n\geq 2$, $\Gamma (x;\cdot )$ has a pole at $x\in \mathbb{R}^{n}$, i.e.,
\begin{equation*}
 \lim_{y\rightarrow x}\Gamma (x;y)=\infty.
\end{equation*}
\end{theorem}
\begin{remark}
\label{rem.drift}
As a matter of facts, most of the results in Theorem
\ref{th.teoremone} can be naturally
extended to homogeneous H\"{o}rmander operators of the kind
$$
\mathcal{L}=\sum_{i=1}^{m}X_{i}^{2}+X_{0},
$$
where $X_{1},...,X_{m}$ are $\delta_{\lambda}$-homogeneous of degree $1$ and
$X_{0}$ (the drift) is $\delta_{\lambda}$-homogeneous of degree $2$
(see Theorem \ref{th.teoremone_drift} for the precise statement).
In order to keep more
readable our presentation, we have only briefly sketched
in Section \ref{sec.stimeDrift} the adjustments of the theory 
necessary to cover this more general
setting, while the core of the paper is written for sum of squares.
\end{remark}

 Let us now say a few words about the techniques used in the proofs of our
 results, thus seizing the opportunity to put our paper in the context of the existing literature. Our
 first step is to combine the integral representation of $\Gamma $ given in
 \eqref{sec.one:mainThm_defGamma}, and similar representation formulas which
 will be established for the derivatives of $\Gamma $ (those in (I) of Theorem \ref{th.teoremone}), with the
 global growth estimates satisfied, for homogeneity reasons, by $\Gamma_\G$
 and its derivatives: this combination gives
\begin{equation*}
 \Big\vert Z_{1}\cdots Z_{r}\Gamma (x;y)\Big\vert
 \leq c_{r}\int_{\mathbb{R}^{p}}
 {d}_{\widetilde{X}}^{2-Q-r}\left( \left( x,0\right)^{-1}\ast \left( y,\eta \right) \right) \d\eta,
\end{equation*}
 for every $x,y\in \mathbb{R}^{n},x\neq y$, where ${d}_{\widetilde{X}}$
 is the CC-distance induced in the Carnot group $\mathbb{R}^{N}$ by the
 lifted vector fields $\widetilde{X}_{1},\ldots,\widetilde{X}_{m}$.
 Once this is accomplished, we shall bound the above
 integral by means of two deep results related to the geometry of H\"{o}rmander
 vector fields and established in the papers \cite{NSW} and \cite{SC},
 here suitably extended to a global version (thanks to the underlying $\delta_{\lambda }$-homogenous structure): see Theorems B and C in Section \ref{sec:notations.review}.

 One could object to our procedure the fact that local estimates for
 $\Gamma $ and its derivatives are also contained in \cite{NSW, SC}. So one
 could think to derive global estimates from the existing local estimates, just by
 dilation arguments. However, what remains a bit unclear in those papers is \emph{which}
 object referred to as $\Gamma $ is actually being estimated. For instance, in \cite[Thm.\,5]{NSW}
 a conditional statement is proved, saying that if, in the space
 of the lifted variables, a kernel $\widetilde{\Gamma }$ satisfies an
 estimate of the kind
\begin{equation*}
 \widetilde{\Gamma }( ( x,\xi ) ;( y,\eta ))
 \leq c\,{d}_{\widetilde{X}}^{2-Q}( ( x,\xi ),( y,\eta ) ),
\end{equation*}
 then the kernel $\Gamma $ that we get by locally saturating the lifted variables similarly to \eqref{sec.one:mainThm_defGamma}
  satisfies local estimates in $\mathbb{R}^{n}$
 of the kind
\begin{equation*}
 \Gamma ( x;y) \leq c\frac{d_X^{2}( x,y) }{\vert B_X( x,d_X( x,y) ) \vert }
\end{equation*}
 (with analogous statements about the derivatives of $\widetilde{\Gamma }$
 and $\Gamma $). In \cite{NSW} the alluded kernel $\widetilde{\Gamma }$ is
 the \emph{parametrix} for the lifted operator constructed in the paper
 \cite{RothschildStein} (no \emph{fundamental solution} is built in \cite{RothschildStein}).
 However, the kernel $\Gamma $ obtained by this procedure
 is hopefully a \emph{local parametrix} for $\LL$, but not necessarily a
 \emph{fundamental solution}. Also, it is a function defined \emph{only locally},
 and in a \emph{non-unique way}. Actually, to produce a true local
 fundamental solution saturating a parametrix in a lifted space, a hard
 extra-work is needed (see e.g., \cite{BBMP}). 
 In contrast with this, the function
 $\Gamma $ that we consider is a uniquely defined, global, fundamental
 solution for $\LL$; for this object, and its derivatives, global estimates are
 proved, together with representation formulas which also contain some
 additional information, not limited to the \emph{size} of these functions.
 An example of the relevance of this last statement will be given in Section \ref{sec:meanvalue},
 see Theorem \ref{Thm prop sing kern}-(iii) and Remark \ref{Remark final}.

 We also note that the estimates that we shall prove for the derivatives of
 $\Gamma \left( x;y\right) $ apply to derivatives of any order, with respect
 to both variables $x$ and $y$. As we shall see, the case of mixed derivatives
 requires a more delicate proof (see Lemma \ref{Thm 0}). On
 the other hand, this bound has interesting consequences, as we shall see in
 Section \ref{sec:meanvalue} (Theorem \ref{Ch1-Thm Lagrange}).
 Incidentally, in \cite[p.\,141]{NSW} a proof is written only for the
 basic estimate of $\widetilde{\Gamma}$, while the proof of the derivative
 estimates is left to the reader.\bigskip

 Once the results in Theorem \ref{th.teoremone} are established, we consider
 some possible applications. Firstly, in Section \ref{sec:potentialtheory} we
 deal with potential-theoretic properties of $\mathcal{L}$. Indeed, the
 estimates of $\Gamma $ and the presence of a blowing-up pole (see (V) in
 Theorem \ref{th.teoremone}) allow us to verify, for our operators $\mathcal{L}$,
 all the axioms of potential theory required for the analysis contained in
 the series of papers \cite{AbbondanzaBonfiglioli, BattagliaBonfiglioli, BattBonf, BLJems}.
 Secondly, in Section \ref{sec:meanvalue} we shall show that
 the kernel
\begin{equation*}
 k\left( x,y\right) =X_{i}^{x}X_{j}^{x}\Gamma \left( x;y\right)
\end{equation*}
 satisfies, globally in $\mathbb{R}^{n}$, the so-called standard estimates of
 singular integrals, together with a suitable cancelation property, with
 respect to both variables. These facts will be proved as a consequence of
 the estimates on second and third order (pure or mixed) derivatives of
 $\Gamma $, together with the explicit integral representation formula of
 $k\left( x,y\right) $ in terms of the homogeneous fundamental solution
 $\widetilde{\Gamma }$ on the Carnot group $\mathbb{R}^{N}$. These properties
 of $k\left( x,y\right) $ could be a starting point to prove global Sobolev
 estimates for solutions to $Lu=f$, both for our operator $\LL$ and
 for more general classes of non-variational operators
 $\sum_{i,j}a_{ij}(x) X_{i}X_{j}$ modeled on our vector fields (with low regular $a_{i,j}$'s and $
 A=(a_{i,j}(x))$ in some class of ellipticity). This theory will be
 developed elsewhere.
%
%

\section{Notations and a review of known results}\label{sec:notations.review}
 In what follows, we denote by $d_{X}$ the Carnot-Carath\'{e}odory (CC, shortly) distance associated with the set of H\"{o}rmander vector fields
 $X=\{X_1,\ldots,X_m\}$, that is,
\begin{equation} \label{eq.defdCC}
 d_{X}(x,y):=\inf \bigg\{r>0:\,\text{there exists $\gamma \in C(r)$ with $\gamma (0)=x$ and $\gamma (1)=y$}\bigg\},
\end{equation}
 where $C(r)$ is the set of the absolutely continuous maps $\gamma:[0,1]\rightarrow \mathbb{R}^{n}$ satisfying (a.e.\thinspace on $[0,1]$)
\begin{equation*}
 \gamma ^{\prime }(t)=\sum_{j=1}^{m}a_{j}(t)\,X_{j}(\gamma (t)),\qquad \text{with $|a_{j}(t)|\leq r$ for all $j=1,\ldots ,m$}.
\end{equation*}
 Given any $x\in \mathbb{R}^{n}$ and $r>0$, we denote by $B_{X}(x,r)$ the
 $d_{X}$-ball of center $x$ and radius $r$. Without risk of confusion, $|\cdot|$ will denote
 the Lebesgue measure in $\mathbb{R}^{n}$ (whatever the $n$).

Thanks to (H.1), $d_{X}$ and $B_{X}$ enjoy the homogeneity properties:
 \begin{gather}\label{propertiesdCC}
  \begin{split}
  & d_{X}(\delta_{\lambda}(x),\delta_{\lambda}(y))=\lambda\, d_{X}(x,y),\\
  & y\in B_{X}(x,r) \Longleftrightarrow \delta_{\lambda}(y) \in B_{X} (\delta_{\lambda}(x),
  \lambda\,r),\\
  & | B_{X}(\delta_{\lambda}(x),\lambda\,r) | =\lambda^{q}\,
  | B_{X}(x,r)|,
 \end{split}
 \end{gather}
 for any $x,y\in \mathbb{R}^{n}$ and any $\lambda ,r>0$. Moreover, it is
 well known from \cite{NSW} that, for $x$ in some compact set and for $r$ small enough, one
 has the doubling property
\begin{equation}\label{globaldoubling}
 |B_{X}(x,2\,r)|\leq C_{d}\,|B_{X}(x,r)|.
\end{equation}
 As a matter of fact, by a homogeneity argument based on \eqref{propertiesdCC},
 the above doubling property holds true for every $x\in \mathbb{R}^{n}$ and
 every $r>0$ (see, e.g., \cite{BattBonf}, or see Section \ref{sec:localtoglobal}).

 In the next sections, $d_{\widetilde{X}}$ will stand for the CC-distance
 associated with the system of vector fields $\widetilde{X}=\{\widetilde{X}_{1},\ldots ,\widetilde{X}_{m}\}$ introduced in point (1) of Theorem A.
 Since the $\widetilde{X}_{j}$'s lift the $X_{j}$'s
 (and all these vector fields are homogeneous with respect
 to appropriate dilations)
 it is known that (here, $\pi_{n}$ is the projection of $\mathbb{R}^{N}=\mathbb{R}^{n}\times \mathbb{R}^{p}$ onto $\mathbb{R}^{n}$)
 \begin{gather} \label{eq.proiezionipalled}
 \begin{split}
 & d_X(x,y)\leq d_{\widetilde{X}}\big((x,\xi),(y,\eta)\big),
 \quad \text{for any $x,y\in\R^n$ and $\xi,\eta\in\R^p$}, \\
 & \pi_n\big(B_{\widetilde{X}}\big((x,\xi),r\big)\big) =
 B_X(x,r), \quad \text{for any $x\in\R^n$, $\xi\in\R^p$
 and $r > 0$}.
 \end{split}
 \end{gather}
Furthermore, since the $\widetilde{X}_{j}$'s are left-invariant on the group
$\mathbb{G}=(\mathbb{R}^{N},\ast )$, one has
\begin{equation*}
 d_{\widetilde{X}}(z,z^{\prime })=d_{\widetilde{X}}\big(0,z^{-1}\ast z^{\prime }\big),\quad \text{for every $z,z^{\prime }\in \mathbb{G}$}.
\end{equation*}
 By an abuse of notation, we shall systematically denote $d_{\widetilde{X}}(0,\cdot )$ by $d_{\widetilde{X}}$. Simple arguments on ho\-mo\-ge\-neous
 groups also show that
\begin{equation}\label{eq.misurapallehom}
 \big|B_{\widetilde{X}}(z,r)\big|=\omega_{Q}\,r^{Q},\quad \text{where $\omega_{Q}=\big|B_{\widetilde{X}}(0,1)\big|$}.
\end{equation}
 It can also be proved that $\mathbb{G}$ is a homogeneous group with dilations
\begin{equation}
 D_{\lambda }:\mathbb{R}^{N}=\mathbb{R}^{n}\times \mathbb{R}^{p}\longrightarrow \mathbb{R}^{N},
 \qquad D_{\lambda }(x,\xi )=(\delta_{\lambda }(x),E_{\lambda }(\xi )),
\end{equation}
 where $E_{\lambda }(\xi )=(\lambda ^{\tau_{1}}\xi ,\ldots ,\lambda ^{\tau_{p}}\xi_{p})$
 for suitable integers $1\leq \tau_{1}\leq \cdots \leq \tau_{p}$.
 We point out that the exponents of $D_{\lambda }$ are not
 increasingly ordered (but this is true of the $\sigma_{i}$'s and of $\tau_{j}$'s separately).
 This will cause some subtleties in handling the group
 structure of $\mathbb{G}$ (see, e.g., Lemma \ref{lem.tecnico}).\medskip

 In this paper we shall make use of the following two theorems, concerning
 the volume of the $d_{X}$-metric balls (Theorem B) and concerning the relation
 between the volumes of lifted and unlifted balls (Theorem C). We shall
 derive them in Section \ref{sec:localtoglobal} via a local-to-global
 homogeneity-argument, starting from their local counterparts. These local
 counterparts will be obtained from the deep investigations on subelliptic
 distances carried out by Nagel, Stein, Wainger \cite[Theorem 1]{NSW}, and by
 S\'{a}nchez-Calle \cite[Theorem 4]{SC} (see also \cite[Lemma 3.2]{NSW} and
 Jerison \cite{J}).

\begin{ThmB}
 Let $q$ be as in \eqref{eq.defq}. For any $k\in \{n,\ldots ,q\}$ there
 exists a function $f_{k}:\mathbb{R}^{n}\rightarrow \mathbb{R}$ which is
 continuous, nonnegative and $\delta_{\lambda }$-homogeneous of degree $q-k$,
 and there exist structural constants $\gamma_{1},\gamma_{2}>0$ such that
\begin{equation}\label{eq.NSWmodificata}
 \gamma_{1}\,\sum_{k=n}^{q}f_{k}(x)\,\rho ^{k}\leq {\big|B_{X}(x,\rho )\big|}
 \leq \gamma_{2}\,\sum_{k=n}^{q}f_{k}(x)\,\rho ^{k},
\end{equation}
 for every $x\in \mathbb{R}^{n}$ and every $\rho >0$. Moreover, $f_{q}(x)$ is constant in $x$, and
 strictly positive.
\end{ThmB}
\begin{remark}[Explicit form of the $f_{k}$'s]\label{Remark explicit f_k}
 In the following, we shall occasionally need the
 explicit form of the functions $f_{k}$. To explain their definition
 (according to \cite{NSW}), we need to introduce some more notation, that will
 also be useful for other reasons. For a multi-index
\begin{equation*}
 I=(i_{1},\ldots ,i_{k}),\quad \text{with $i_{1},\ldots ,i_{k}\in\{1,2,\ldots ,m\}$,}
\end{equation*}
 let us define (when $k=1$) $X_{[I]}=X_{i_{1}}$, and (when $k>1$)
\begin{equation*}
 X_{[I]}:= \left[  \left[  \left[ X_{i_{1}},X_{i_{2}} \right] ,X_{i_{3}} \right],\ldots ,X_{i_{k}} \right] ;
\end{equation*}
 we also define the weight (or length) of $I$ as $|I|=k$. At any fixed point $x\in \mathbb{R}^{n}$, let us consider a basis
 of $\mathbb{R}^{n}$ consisting in a set  $\{X_{[I]}(x):\,I\in \mathcal{A}\}$ (where $\mathcal{A}$ is a set of $n$ multi-indices)
 of $n$ commutators evaluated at $x$, and let us arrange these $n$
vectors in an $n\times n$ matrix which we denote by:
\begin{equation*}
 \mathcal{B}(x)=\Big( X_{ [ I ] }(x)\Big)_{I\in \mathcal{A}}.
\end{equation*}
 The weight of this basis will be, by definition,
\begin{equation*}
 |\mathcal{B}(x) | =\sum_{I\in \mathcal{A}} |I | .
\end{equation*}
 Then, according to \cite{NSW}, the functions $f_{k}$ appearing in Theorem B
 are equal to
\begin{equation*}
 f_{k} ( x ) =\sum_{ | \mathcal{B}(x) | =k}
 \big|\det \big(\mathcal{B}(x) \big)\big|,
\end{equation*}
 where the sum is taken over all the possible bases of $\mathbb{R}^{n}$ having weight $k$.
\end{remark}
\begin{ThmC}
There exist constants $\kappa \in (0,1)$ and $c_{1},c_{2}>0$ such that, for
every $x\in \mathbb{R}^{n}$, $\xi \in \mathbb{R}^{p}$ and $r>0$ one has the
following estimates:
\begin{align}
 \Big|\{\eta \in \mathbb{R}^{p}:(y,\eta )\in B_{\widetilde{X}}((x,\xi ),r)\}\Big|
 & \leq c_{1}\frac{\big|B_{\widetilde{X}}((x,\xi ),r)\big|}{\big|B_{X}(x,r)\big|},\quad
 \text{for all $y\in \mathbb{R}^{n}$},  \label{eq.SanchezI} \\[0.3cm]
 \Big|\{\eta \in \mathbb{R}^{p}:(y,\eta )\in B_{\widetilde{X}}((x,\xi ),r)\}\Big|
 & \geq c_{2}\frac{\big|B_{\widetilde{X}}((x,\xi ),r)\big|}{\big|B_{X}(x,r)\big|},\quad
 \text{for all $y\in B_{X}(x,\kappa \,r)$}.  \label{eq.SanchezII}
\end{align}
\end{ThmC}
\begin{remark} \label{rem.dacontrollare}
 As a matter of fact, the local version of Theorem
 C is proved in \cite{NSW, SC} (see also \cite{J}) in a slightly different
 framework: namely \cite{NSW,SC} assume that the vector fields in $\widetilde{X}$
 lift those in $X$ in the sense of the lifting by Rothschild-Stein \cite{RothschildStein}.
 Nevertheless, scrutinizing the proof of Theorem C in its
 local version (as written in detail, e.g., in \cite[Chap.\,10]{BBbook}), one
 can see that the estimates in \eqref{eq.SanchezI}-\eqref{eq.SanchezII} hold
 true for our vector fields as well, where $\widetilde{X}$ lift $X$ \emph{in
 the sense of Folland} \cite{Folland}.

 Namely, most of the proof of Theorem C
 relies on the deep properties established by Nagel, Stein, Wainger \cite{NSW}
 for \emph{every} system of H\"{o}rmander vector fields, while the only
 special
 properties of the vector fields $\widetilde{X}_{i}$ which are exploited are
 the following:
\begin{enumerate}
 \item the vector field $\widetilde{X}_{i}$ projects onto $X_{i}$ for any $i$;

 \item if a family of commutators of $\big\{\widetilde{X}_{[I]}\big\}_{I\in
 \mathcal{A}}$ is a basis of $\mathbb{R}^{N}$ at some point of $\mathbb{R}^{N}$,
 then the same is true at every other point;

\item if any two families of commutators
 $\big\{\widetilde{X}_{[I]}\big\}_{I\in \mathcal{A}}$ and $\big\{\widetilde{X}_{[I']}\big\}_{I'\in \mathcal{A}'}$
 are bases of $\mathbb{R}^{N}$ at some point of $\mathbb{R}^{N}$, then it holds that
\begin{equation*}
\sum_{I\in \mathcal{A}}|I|=\sum_{I'\in \mathcal{A}'}|I'|.
\end{equation*}
\end{enumerate}
 In our setting, property (1) holds true by our very definition of lifting
 \eqref{lifting}; properties (2) and (3) (fulfilled by the Rothschild-Stein's
 lifted vector fields as they are \emph{free} vector fields) hold true in our
 case since the $\widetilde{X}_{i}$'s are the Lie-generators of a \emph{stratified Carnot group $\mathbb{G}$} (see Theorem A).

 More precisely, property (2) is a consequence of the left invariance (see
 \cite[Prop.\,C.5]{BiagiBonfBook}), while property (3) is a consequence of
 the stratification of the Lie algebra. Namely, assume that
\begin{equation*}
 \mathcal{X}=\big\{\widetilde{X}_{[I]}\big\}_{I\in \mathcal{A}}
\end{equation*}
 is a basis of $\mathbb{R}^{N}$ at some point in space: then $\mathcal{X}$ is
 also a basis of $\mathrm{Lie}(\mathbb{G})$ (see again \cite[Prop.\,C.5]{BiagiBonfBook});
 moreover, by grouping together the elements of $\mathcal{X} $
 with the same $D_{\lambda }$-homogeneity, one gets bases of the layers of
 the stratification of $\mathrm{Lie}(\mathbb{G})$. It then suffices to apply
 \cite[Prop.\,2.2.8]{BLUlibro}, ensuring that $\sum_{I\in \mathcal{A}}|I|$ is
 nothing but the so-called $D_{\lambda }$-homogeneous dimension of $\mathbb{G}$,
 which is independent of $\mathcal{X}$. \medskip

 A last remark is in order. Our statement of Theorem C involves lifted balls
 $B_{\widetilde{X}}((x,\xi ),r)$, while the analogous (local) statement in
 \cite{NSW} deals with lifted balls centred at points of the form $(x,0)$.
 However, this latter choice is immaterial, motivated (as appears by a close
 inspection of the proof in \cite{NSW}) by notational convenience.
\end{remark}
\section{Local to global via homogeneity}\label{sec:localtoglobal}
 The following fact is a very simple consequence of homogeneity; we shall use it
 so many times that we provide it in details for the sake of reference convenience.
\begin{remark}\label{rem.usehomog}
 Let $m\in \N$; let $\{M_\lambda\}_{\lambda>0}$ be the family of dilations of $\R^m$ defined by
  $$M_\lambda(w_1,\ldots,w_m)=\big(\lambda^{\mu_1}w_1,\ldots,\lambda^{\mu_m}w_m\big), $$
  where $\mu_1,\ldots,\mu_m$ are fixed positive real numbers. Let $\Omega\subseteq\R^m$ be closed under $\{M_\lambda\}_\lambda$, i.e.,
\begin{equation}\label{chiusoMlambdaOmega}
  \text{$M_\lambda(w)\in \Omega$ for every $w\in\Omega$ and every $\lambda>0$.}
\end{equation}
  Suppose that $F,G:\Omega\to \R$ are two $M_\lambda$-homogeneous functions of the same degree, say $\alpha$, i.e.,
  $$\begin{cases}
       F(M_\lambda(w))=\lambda^\alpha\,F(w)\\
       G(M_\lambda(w))=\lambda^\alpha\,G(w),
 \end{cases}
 \quad \text{for every $w\in\Omega$ and every $\lambda>0$.}$$
 Finally, suppose that there exists a neighborhood $O$ of $0\in\R^m$ such that
 $O\cap \Omega\neq \emptyset$ and $F\leq G$   on $O\cap \Omega$; then $F\leq G$ on $\Omega$. Indeed, let $w\in \Omega$ be arbitrary;
 then there exists a small $\lambda>0$ such that $M_\lambda(w)\in O\cap \Omega$
 (this follows from \eqref{chiusoMlambdaOmega} and since $M_\lambda(w)\to 0\in\R^m$ as $\lambda\to 0^+$). As $F\leq G$ on  $O\cap \Omega$
 we infer that $F(M_\lambda(w))\leq G(M_\lambda(w))$; due to the $M_\lambda$-homogeneity of $F$ and $G$, this is equivalent to
 $\lambda^\alpha\,F(w)\leq \lambda^\alpha\,G(w)$. Canceling out $\lambda^\alpha>0$, this gives $F(w)\leq G(w)$.

 A completely analogous result holds true if we replace ``$F\leq G$'' with ``$F= G$'' or ``$F<G$.''
\end{remark}
 As a first application of Remark \ref{rem.usehomog}, we prove the global doubling inequality \eqref{globaldoubling}:
 indeed, by classical  results in \cite{NSW}, one knows that there exist a constant $C_d>0$, a neighborhood $U_0$ of $0\in\R^n$
 and some $r_0>0$ such that
 $|B_X(x,2\,r)|\leq C_d\,|B_X(x,r)|$ for every $x\in U_0$ and every $r\in (0,r_0)$. Then we apply
 Remark \ref{rem.usehomog} with the choices $m=n+1$, $\Omega=\R^n\times (0,\infty)$,
 $$M_\lambda(x,r)=(\dela(x),\lambda\,r)\quad \text{for $x\in \R^n$, $r>0$ and $\lambda>0$,} $$
 and with the functions
 $ F(x,r)=|B_X(x,2\,r)|$ and $G(x,r)=C_d\,|B_X(x,r)|$.
 These choices satisfy the assumptions in Remark \ref{rem.usehomog}, since
 $\Omega$ is invariant under $\{M_\lambda\}_\lambda$,
 $F\leq G$ on $U_0\times (0,r_0)$ (which is of the form $O\cap \Omega$ for some
 neighborhood $O$ of $0\in\R^m$) and since $F$ and $G$ are both $M_\lambda$-homogeneous
 of degree $q$, due to \eqref{propertiesdCC}.

 Another application of Remark \ref{rem.usehomog} is the following:
\begin{remark}\label{hormandereverywhere}
 We prove that, due to assumptions (H.1) and (H.2),
 the validity of H\"{o}rmander's rank condition at $0$ implies its validity at any other point
 $x\in\mathbb{R}^{n}$.
 Indeed, it is easy to check that, by (H.1), the vector field $X_{[I]}$ is
 $\dela$-homogeneous of degree $|I|$, which is equivalent to
\begin{equation}\label{equidellahomogvf}
 X_{[I]}(\dela(x))=\lambda^{-|I|}\dela(X_{[I]}(x)),\quad \forall\,\,\lambda>0,\,\,x\in\R^n.
\end{equation}
 Next, we observe that the iterated (left nested) brackets $X_{[I]}$ span $\mathrm{Lie}(X)$.
 Hence, by (H.2), we can find a family $X_{[I_1]},\ldots,X_{[I_n]}$
 such that  $X_{[I_1]}(0),\ldots,X_{[I_n]}(0)$ is a basis of $\R^n$.
 Thus, the function
 $$x\mapsto F(x):=\det \big(X_{[I_1]}(x)\cdots X_{[I_n]}(x)\big)$$
 is non-null on a neighborhood $O$ of $0\in\R^n$. If we show that $F$ is $\dela$-homogeneous, then
 Remark \ref{rem.usehomog} will prove that $F(x)\neq 0$ for every $x\in \Omega:=\R^n$.
 The $\dela$-homogeneity of $F$ can be proved as follows:
 \begin{align*}
  F(\delta_\lambda(x))&\stackrel{\eqref{equidellahomogvf}}{=}
    \det\Big(\lambda^{-|I_1|}\,\delta_\lambda\big(X_{[I_1]}(x)\big)\cdots
    \lambda^{-|I_n|}\,\delta_\lambda\big(X_{[I_n]}(x)\big)\Big) \\
    &\,\,\,=\,\, \lambda^{-|I_1|-\cdots-|I_n|}\,
    \det\Big(\delta_\lambda\big(X_{[I_1]}(x)\big)\cdots
    \delta_\lambda\big(X_{[I_n]}(x)\big)\Big)\\
    &\,\,\,=\,\, \lambda^{q-|I_1|-\cdots-|I_n|}\,
    \det\Big(X_{[I_1]}(x)\cdots X_{[I_n]}(x)\Big)=
     \lambda^{q-|I_1|-\cdots-|I_n|}\,F(x).
   \end{align*}
\end{remark}
 \begin{remark} \label{rem.confrontoBramanti}
   As a matter of fact, the CC-distance $d_X$ in \eqref{eq.defdCC} is not the unique distance one can attach to
   the H\"ormander family $\mathcal{X}$: for instance, following the notation
  in Remark \ref{rem.dacontrollare}, one can deal with the so-called subelliptic distance
  \begin{equation*}
  \rho(x,y) := \inf\bigg\{r > 0:\,\text{there exists $\gamma\in C'(r)$
  with $\gamma(0) = x$ and $\gamma(1) = y$}\bigg\},
  \end{equation*}
 where $C'(r)$ is the set of absolutely continuous maps $\gamma:[0,1]\to\R^n$ satisfying (a.e.\,on $[0,1]$)
 $$\gamma'(t) = \sum_{|I|\leq\sigma_n}a_I(t)\,X_{[I]}(\gamma(t)),
 \qquad \text{with $|a_I(t)|\leq r^{|I|}$ for all $I$}.$$
 We remind that $\sigma_n$ is the largest length of some non-vanishing commutator $X_{[I]}$
 (see Remark \ref{rem.assumptionsH}). On account of
 the results in \cite[Section 1.4]{NSW}, $\rho$ and $d_X$ are locally equivalent:
 \begin{equation} \label{eq.localequivdrho}
  c_1\,d_X(x,y)\leq \rho(x,y)\leq c_2\,d_X(x,y),
 \end{equation}
 for every $x,y$ in some neighborhood $U_0$ of the origin, and some constants $c_1, c_2 > 0$. Moreover,
 it is not difficult to check that
 $\gamma\in C'(r)$ if and only if $\dela\circ\gamma\in C'(\lambda\,r)$;
 as a consequence,
 \begin{equation} \label{eq.rhohomog}
 \text{$\rho(\dela(x),\dela(y)) = \lambda\,\rho(x,y)$ for every $x,y\in\R^n$ and every $\lambda > 0$}.
 \end{equation}
 Starting from
 \eqref{eq.localequivdrho}-\eqref{eq.rhohomog},
 and using Remark \ref{rem.usehomog}, one can prove that \eqref{eq.localequivdrho} holds for every
 $x,y\in\R^n$. The global version of \eqref{eq.localequivdrho} implies the following inclusions:
 \begin{equation} \label{eq.inclusionballs}
  B_\rho(x,c_1\,r)\subseteq B_X(x,r)\subseteq B_\rho(x,c_2\,r),\quad \text{for all $x\in\R^n$ and $r > 0$}.
\end{equation}
 \end{remark}
 We shall now use Remarks \ref{rem.usehomog}-\ref{rem.confrontoBramanti}
 in deriving Theorems B and C from their local counterparts.
\begin{proof}[Proof of Theorem B]
 We start from the notable estimate proved in \cite[Theorem 1]{NSW}: there exist a neighborhood $U_0$ of the origin
 in $\R^n$, a real $r'_0 > 0$ and two constants $\gamma_1',\gamma_2' > 0$ such that
 \begin{equation} \label{eq.estimNSWlocal}
   \gamma_1'\,\Lambda(x,r) \leq \big|B_\rho(x,r)\big| \leq \gamma_2'\,\Lambda(x,r), \quad \text{for every $x\in U_0$ and every $0<r\leq r'_0$},
\end{equation}
 where (following the notation in Remark \ref{rem.dacontrollare} for $I_j$, $X_{[I_j]}$ and $|I_j|$)
 \begin{equation} \label{eq.defLambdaxr}
 \Lambda(x,r) := \sum_{B = (I_1,\ldots,I_n)}\Big|\mathrm{det}\big(X_{[I_1]}(x)\cdots X_{[I_n]}(x)\big)\Big|\,r^{|I_1|+\cdots+|I_n|},
\end{equation}
 and the sum runs over all the $n$-tuples $B$ of multi-indexes $I_j$'s with $|I_j|\leq \sigma_n$. We now observe that, by definition,
 $|I_1|+\cdots+|I_n|\geq n$; on the other hand, from Remark \ref{rem.confrontoBramanti}
 we know that
 $$x\mapsto \mathrm{det}\big(X_{[I_1]}(x)\cdots X_{[I_n]}(x)\big)$$
 is $\dela$-homogeneous of degree $q-(|I_1|+\cdots+|I_n|)$. Hence, the latter function (which is smooth on $\R^n$) must vanish whenever
 $|I_1|+\cdots+|I_n| > q$ and it is constant when the equality holds.

 Gathering together all these facts, we can reorder the sum
 in the right-hand side of \eqref{eq.defLambdaxr} with respect to $|I_1|+\cdots+|I_n|$, obtaining
 the representation
 $$\Lambda(x,r) = \sum_{k = n}^q f_k(x)\,r^k,$$
 where $f_k$ is a nonnegative continuous function on $\R^n$ which is $\dela$-homogeneous of degree $q-k$. In particular, $f_q$ is constant and
 \begin{equation} \label{eq.Lambdaxrhom}
  \Lambda(\dela(x),\lambda r) = \lambda^q\,\Lambda(x,r), \quad \text{for every $x\in\R^n$, every $r,\lambda > 0$}.
 \end{equation}
 On account of \eqref{eq.inclusionballs} and \eqref{eq.estimNSWlocal}, we get the estimate
\begin{equation} \label{eq.touseinsiemeremark}
   \gamma_1\,\sum_{k = n}^q f_k(x)\,r^k \leq \big|B_X(x,r)\big|\leq \gamma_2\,\sum_{k = n}^q f_k(x)\,r^k, \quad \text{for $x\in U_0$ and $0<r\leq r_0$},
\end{equation}
 for some structural positive constants $\gamma_1,\gamma_2$ and $r_0$. One can pass from the local \eqref{eq.touseinsiemeremark} to
 the global \eqref{eq.NSWmodificata} by using Remark \ref{rem.usehomog} and the fact that
 all members of \eqref{eq.touseinsiemeremark} are homogeneous of degree $q$ with respect to the dilations
 $(x,r)\mapsto (\dela(x),\lambda r)$. Finally, taking $x = 0$ and $r = 1$ in \eqref{eq.NSWmodificata}, we infer that
 the constant $f_q$ is strictly positive, since
 $f_k(0) = 0$ for all $n\leq k\leq q-1$ (which is a consequence of the $\dela$-homogeneity of $f_k$ of degree $q-k$).
\end{proof}
 Finally we provide the
\begin{proof}[Proof of Theorem C]
 The deep result proved in \cite[Theorem 4]{SC} provides a local version
 of \eqref{eq.SanchezI}-\eqref{eq.SanchezII}. As for \eqref{eq.SanchezI}, one can
 pass from local to global via an
 application of Remark \ref{rem.usehomog}, once noticed that the members in
 \eqref{eq.SanchezI} are homogeneous of degree $Q-q$ with respect to
 (see also \eqref{eq.dilatazioniG})
 $$(x,\xi,y,r)\mapsto (\dela(x), E_\lambda(\xi),\dela(y), \lambda\,r).$$
 In this argument it may help observing that
 $$\Big\{\eta\in\mathbb{R}^{p}: (\dela(y),\eta)  \in
 B_{\widetilde{X}} ((\dela(x),E_\lambda(\xi)), \lambda\,r)\Big\}
 = E_\lambda\bigg(\Big\{\eta'\in\mathbb{R}^{p}: (y,\eta')  \in
 B_{\widetilde{X}} ((x,\xi), r)\Big\}\bigg).$$
 As for \eqref{eq.SanchezII}, one chooses $\lambda > 0$ so small that
 $$\Big|\Big\{\eta\in\mathbb{R}^{p}: (\dela(y),\eta)  \in
 B_{\widetilde{X}} ((\dela(x),E_\lambda(\xi)), \lambda\,r)\Big\}\Big|
 \geq c_{2}\frac{|B_{\widetilde{X}}((\dela(x),E_\lambda(\xi)), \lambda\,r)|}{|B_{X}(\dela(x), \lambda\,r)|},
 $$
 holds true when $\dela(y)\in B_X(\dela(x),\kappa\,\lambda\,r)$ (here, $\kappa > 0$ is the same as in
 \cite[Theorem 4]{SC}). From this, by applying \eqref{propertiesdCC} and by arguing
 as above, one gets the desired \eqref{eq.SanchezII}.

\end{proof}

\section{Global upper estimates of $\Gamma$ and its derivatives}\label{sec:pointwise.estimates}
 In this section we tacitly inherit the notations in Section \ref{sec:notations.review}, and the
 assumptions (H.1)-to-(H.3);
 in particular $q$ is always assumed to be larger than $2$. This assumption is not restrictive, as the
 case $q=2$ boils down to the well-known case when
 $\LL$ is a (strictly) elliptic operator in $\R^2$ with constant coefficients. Indeed
 $q=2$ implies that $n=2$ and $\delta_\lambda(x_1,x_2)=(\lambda\,x_1,\lambda\,x_2)$;
 as a consequence, since $X_1,\ldots,X_m$ are linearly independent H\"ormander vector fields, $\dela$-homogeneous of degree $1$,
 one necessarily has $m=2$ and
  $$X_1=a\,\de_{x_1}+b\,\de_{x_2},\quad X_2=\alpha\,\de_{x_1}+\beta\,\de_{x_2}\quad \text{with $a\beta-b\alpha\neq 0$}.$$

 The aim of the present section is to prove the following main result:
\begin{theorem}[Global estimates of $\Gamma$ and
 its derivatives]\label{thm.estimatesDERIVgamma}
 Let $\Gamma$ be the fundamental solution of $\LL$ in
  \eqref{sec.one:mainThm_defGamma}; we also assume that $n>2$.
 Then, for any integer $r\geq 0$, there exists $C>0$
 \emph{(}depending on $r$ and on the set $X=\{X_1,\ldots,X_m\}$\emph{)} such that
\begin{equation}\label{goal}
 \Big\vert Z_1\cdots Z_r\Gamma(x;y) \Big\vert
   \leq C\,
   \frac{d_{X}(x,y)^{2-r}}
   {\big\vert B_{X}(x,d_{X}(x,y)) \big\vert },
\end{equation}
 holding true for any $x,y\in\mathbb{R}^{n}$ \emph{(}with $x\neq y$\emph{)} and any
 choice of
\begin{equation}\label{goalZZZZZ}
 Z_1,\ldots,Z_r\in\Big\{X^x_1,\ldots,X^x_m,X^y_1,\ldots,X^y_m\Big\},
\end{equation}
 where superscripts denote the variable with respect to which differentiation is performed.
 In particular, for every fixed $x\in\R^n$ we have
 \begin{equation} \label{eq.derGammavanish}
  \lim_{|y|\to\infty}Z_1\cdots Z_r\Gamma(x;y) = 0.
 \end{equation}
\end{theorem}
 The estimate \eqref{goal} results from an integral representation of
 $Z_1\cdots Z_r\Gamma(x;y)$ (see precisely Lemma \ref{Thm 0}), which seems to have
 an interest in its own.

 The case $n = 2$, which is not comprised in Theorem \ref{thm.estimatesDERIVgamma},
 will be investigated in Section \ref{sec:casen2}.
 \begin{remark} \label{rem.simmetriafrac}
  The function $H(x,y)$ in the right-hand side of
  \eqref{goal} is not symmetric (in $x,y$) as it stands; however, one can recognize that
  it is equivalent (up to a structural constant) to the function $H(y,x)$. Indeed, one has
  the following computation based on the doubling inequality \eqref{globaldoubling}
  (and on the trivial inclusion $B_X(y,d_X(x,y))\subseteq B_X(x,2\,d_X(x,y))$):
  \begin{align*}
    H(x,y) & = \frac{d_{X}(x,y)^{2-r}}
   {\big\vert B_{X}(x,d_{X}(x,y)) \big\vert } = \frac{d_{X}(y,x)^{2-r}}
   {\big\vert B_{X}(x,d_{X}(y,x)) \big\vert } \leq C_d\,\frac{d_{X}(y,x)^{2-r}}
   {\big\vert B_{X}(y,d_{X}(y,x)) \big\vert } = C_d\,H(y,x).
  \end{align*}
 \end{remark}
 The proof of Theorem  \ref{goal} is long and requires several preliminary results, the first
 of which is the following lemma, where the role of mixed derivatives is unexpectedly delicate.
\begin{lemma}\label{Thm 0}
  For any $s,t\geq 1$, and any
  choice of $i_1,\ldots,i_s,j_1,\ldots,j_t\in \{1,\ldots,m\}$, we have the following
  representation formulas \emph{(}holding true for $x\neq y$ in $\R^n$\emph{)}:
\begin{align}
  & X^y_{i_1}\cdots X^y_{i_s} \big(\Gamma(x;\cdot)\big)(y)=
   \int_{\R^p}
   \Big(\widetilde{X}_{i_1}\cdots \widetilde{X}_{i_s}\Gamma_\G \Big)\Big((x,0)^{-1}
   *(y,\eta)\Big) \,\d\eta\,;\label{deriyyyyy}\\[0.2cm]
   & X^x_{j_1}\cdots X^x_{j_t} \big(\Gamma(\cdot;y)\big)(x)=
   \int_{\R^p}
   \Big(\widetilde{X}_{j_1}\cdots
   \widetilde{X}_{j_t}\Gamma_\G \Big)
   \Big((y,0)^{-1}*(x,\eta)\Big) \,\d\eta\,;\label{derixxxx0}\\[0.2cm]
   & X^x_{j_1}\cdots X^x_{j_t}X^y_{i_1}\cdots X^y_{i_s} \Gamma(x;y)\label{derixxxxyyyy0}  \\
   &\qquad \qquad=
   \int_{\R^p}
   \bigg(\widetilde{X}_{j_1}\cdots \widetilde{X}_{j_t}\Big(\big(\widetilde{X}_{i_1}
   \cdots \widetilde{X}_{i_s}\Gamma_\G \big)\circ\iota\Big)\bigg)
   \Big((y,0)^{-1}*(x,\eta)\Big) \,\d\eta\,.\nonumber
\end{align}
 Here $\iota$ denotes the inversion map on the Lie group $\G$; moreover,
 $\widetilde{X}_1,\ldots,\widetilde{X}_m$ are lifting vector fields of
 ${X}_1,\ldots,{X}_m$ as in Theorem A.
\end{lemma}
 Whereas \eqref{derixxxx0} follows from \eqref{deriyyyyy} and from the symmetry of $\Gamma$,
 the representation \eqref{derixxxxyyyy0} of the \emph{mixed}
 derivatives is more delicate and it requires a suitable change of variable
 argument. This extra work is motivated by the investigation on singular integrals carried out in Section \ref{sec:meanvalue}.
 \begin{proof}
 We split the proof in four parts: (I) contains a general argument in order to
 pass one vector field $Z$ (indifferently operating in $x$ or $y$) under $\int_{\R^p} g(x,y,
 \eta)\,\d \eta $,
 for a suitable homogeneous $g$; next (II)-to-(IV) contain the proofs of
 \eqref{deriyyyyy}-to-\eqref{derixxxxyyyy0}.\medskip

 (I) Let us consider the following families of dilations:
\begin{gather} \label{eq.dilatazioniEFG}
\begin{split}
 & \text{$D_\lambda(x,\xi)=(\delta_\lambda(x),E_\lambda(\xi))$
 \quad on $\R^n_x\times\R^p_\xi$\,\,\,as in \eqref{eq.dilatazioniG}};\\
 & F_\lambda(x,y,\eta):= \Big(\dela(x),
 \delta_\lambda(y),E_\lambda(\eta)\Big)\quad
 \text{on $\R^n_x\times \R^n_y\times \R^p_\eta$};\\
 & G_\lambda(x,y):=\big(\dela(x),\dela(y)\big)\quad \text{on $\R^n_x\times \R^n_y$}.
 \end{split}
\end{gather}
 Let $\Omega:=\{(x,y,\eta)\in \R^{n}\times \R^{n}\times \R^p\,:\,(x,0)\neq (y,\eta)\}$, and
 suppose $g\in C^\infty(\Omega)$ is homogeneous of degree $\alpha<q-Q$ with respect to $F_\lambda$.
 Let $Z$ be any smooth vector field in the $(x,y)$-variables, homogeneous of degree $m>0$
 with respect to $G_\lambda$.
 Then, the following facts hold:
 \begin{enumerate}
   \item[(i)] for any fixed $(x,y)\in \R^{n}\times \R^{n}$ with $x\neq y$, the map $\eta\mapsto g(x,y,\eta)$
   belongs to $L^1(\R^p)$;
   \item[(ii)] $Z$ can pass under the integral sign as follows
 \begin{equation} \label{eq.ZpassIntegral}
  Z\bigg\{(x,y)\mapsto \int_{\R^p} g(x,y,\eta)\,\d\eta\bigg\}=
  \int_{\R^p} Z\Big\{(x,y)\mapsto g(x,y,\eta)\Big\}\,\d\eta,\qquad \text{for $x\neq y$.}
  \end{equation}
 \end{enumerate}
 We prove (i). Let us fix $x_0,y_0\in\R^{n}$ such that $x_0\neq y_0$ and let
 $S,\,N$ be the homogeneous norms (with respect to $\delta_\lambda$ and $E_\lambda$, respectively)
 \begin{equation} \label{eqSedNdausare}
  \textstyle S(x):=\sum_{j=1}^n |x_j|^{1/\sigma_j}\quad \text{and} \quad
 N(\eta):=\sum_{j=1}^p |\eta_j|^{1/\tau_j}.
 \end{equation}
 Since, obviously,
 $\eta\mapsto g(x_0,y_0,\eta)$ belongs to $L^1_{\loc}(\R^p)$, assertion (i) will follow if we prove that
 $$\int_{\{N > 1\}}g(x_0,y_0,\eta)\,\d\eta < \infty.$$
 To this end, we first choose $\rho_0 > 0$ in such a way that $x_0,y_0\in\{S(x)\leq \rho_0\}$ and we observe that,
 since the set $K:= \{x:S(x)\leq \rho_0\}\times\{y:S(y)\leq \rho_0\} \times\{\eta:N(\eta) = 1\}$ is compact and contained in $\Omega$,
 there exists $C > 0$ such that
 \begin{equation} \label{eq.boundgcompact}
  |g(x,y,\eta)|\leq C \qquad\text{for every $(x,y,\eta)\in K$}.
  \end{equation}
 On the other hand, if $\eta\in\R^p$ is such that $N(\eta) > 1$ and if we set $\lambda := 1/N(\eta)\in (0,1)$,
 it is readily seen that $F_\lambda(x_0,y_0,\eta)\in K$;
 thus, by \eqref{eq.boundgcompact} and the $F_\lambda$-homogeneity of $g$, we get
 $$|g(x_0,y_0,\eta)| \leq C\,N(\eta)^{\alpha} \quad \text{for every $\eta\in\R^p$ with $N(\eta) > 1$}.$$
 Since $\alpha < q-Q$, we conclude that $\eta\mapsto g(x_0,y_0,\eta)$
 is integrable on $\{N > 1\}$, as a simple homogeneity argument shows (see e.g., \cite[eq.\,(5.14)]{BB}); here one
 also exploits the fact that the $E_\lambda$-homogeneous dimension is $\sum_{j=1}^p\tau_j=Q-q$.

 We prove (ii). We show that, if $Z$ is as above, then
 $\Phi(x,y,\eta):= Z\{(x,y)\mapsto g(x,y,\eta)\}$ is $\eta$-integrable in $\R^p$
 (for any $x\neq y$ in $\R^{n}$).
 To this end we observe that, if we think of $Z$ as a vector field
 defined on $\R^{n}_x\times \R^{n}_y\times \R^p_\eta$ but acting only in the $(x,y)$ variables
 (and not on $\eta$), then
 $Z$ is $F_\lambda$-homogeneous of degree $m$; as a consequence,
 $\Phi$ is $F_\lambda$-homogeneous of degree $\alpha - m$.
 Since, by assumption, $m > 0$ and $\alpha < q-Q$, we derive from statement (i) that
 $\Phi(x,y,\cdot)$  belongs to $L^1(\R^p)$ for every $x\neq y$ in $\R^{n}$.
   We now prove \eqref{eq.ZpassIntegral} with a dominated-convergence argument.

   To this aim, we write
   $$\int_{\R^p}|\Phi(x,y,\eta)|\,\d\eta =
   \int_{\{N(\eta)\leq 1\}}|\Phi(x,y,\eta)|\,\d\eta +
   \int_{\{N(\eta) > 1\}}|\Phi(x,y,\eta)|\,\d\eta.$$
   We fix $x_0\neq y_0$ in $\R^{n}$ and we provide integrable dominant
   functions for both the above integrals,
   independent of
   $(x,y)$ near $(x_0,y_0)$. As for the first integral,
   we choose $r > 0$ in such a way that
   $\overline{B(x_0,r)}\cap\overline{B(y_0,r)}=\varnothing$ and we set $K :=
    \overline{B(x_0,r)}\times\overline{B(y_0,r)}\times\{N\leq 1\}$. By the choice of $r$,
    we see that $K$ is a compact subset of $\Omega$; thus,
    there exists $C > 0$ such that
    $$|\Phi(x,y,\eta)|\leq C \quad
    \text{for every $(x,y,\eta)\in K$}.$$
    As for the second integral, we argue as in the proof of the previous statement (i):
    if $\rho_0 > 0$ is such that $x_0,y_0\in\{S(x)\leq \rho_0\}$,
    from the $F_\lambda$-homogeneity
    of $\Phi$ we infer the existence of $C' > 0$ such that
    $$|\Phi(x,y,\eta)| \leq C'\,N(\eta)^{\alpha-m} \quad
     \text{for every $x,y\in\{S\leq \rho_0\}$ and every $\eta\in \{N > 1\}$};$$
    since
    $\alpha-m <\alpha < q-Q$, the function $N^{\alpha-m}$ is integrable
    on $\{N > 1\}$.\medskip

  (II) We prove \eqref{deriyyyyy}: for any $s\geq1$, a repeated application of (I) shows that
\begin{equation} \label{sec.one:mainThm_defGammakkkkkkk}
 \begin{split}
   & X^y_{i_1}\cdots X^y_{i_s} \big(\Gamma(x;\cdot )\big)(y)=
   \int_{\R^p}
    X^y_{i_1}\cdots X^y_{i_s}
 \bigg\{y\mapsto
 \Gamma_\G \Big((x,0)^{-1}*(y,\eta)\Big)
 \bigg\}
    \,\d\eta \qquad (\text{for $x\neq y$}).
 \end{split}
\end{equation}
  We claim that in the above right-hand side it is legitimate to replace the vector fields $X^y_i$
  with their lifted $\widetilde{X}^{(y,\eta)}_i$. Indeed, this will follow upon an inductive argument based
  on the next fact: if $h$ is smooth on $\RN\setminus\{0\}$, and it is $D_\lambda$-homogeneous of degree $<q-Q$, then
\begin{equation} \label{sec.one:mainThm_defGammakkkkkkkeq1}
 \begin{split}
   & \int_{\R^p}
 X^y_{i}
   \bigg\{y\mapsto
   h\Big((x,0)^{-1}*(y,\eta)\Big) \bigg\}
    \,\d\eta
   =
   \int_{\R^p}
    \big(\widetilde{X}_{i}h\big)\Big((x,0)^{-1}*(y,\eta)\Big)
    \,\d\eta.
 \end{split}
\end{equation}
 This will inductively prove \eqref{sec.one:mainThm_defGammakkkkkkk}, since
 $\Gamma_\G , \widetilde{X}_{i_1}\Gamma_\G ,\widetilde{X}_{i_1}\widetilde{X}_{i_2}\Gamma_\G ,\ldots$
 are smooth out of $0$ and $D_\lambda$-ho\-mo\-ge\-neous of degrees $2-Q, 1-Q, -Q,\ldots$, respectively. Thus, we turn to show
 \eqref{sec.one:mainThm_defGammakkkkkkkeq1}. To this end, by (I), both integrals in \eqref{sec.one:mainThm_defGammakkkkkkkeq1} are finite.
 Moreover, we remind that
\begin{equation*} 
 \widetilde{X}^{(y,\eta)}_i=X^y_i+R_i,\qquad \text{with $R_i=\sum_{j=1}^p \alpha_{i,j}(y,\eta)\,\frac{\de}{\de \eta_j}$,}
\end{equation*}
  where $\alpha_{i,j}$ is smooth and $D_\lambda$-homogeneous of degree $\tau_j-1$; in particular $\alpha_{i,j}$ does not depend on $\eta_j$.
  Now, by exploiting the left-invariance of $\widetilde{X}_{i}$ on $\G = (\RN,*)$ we have
\begin{align*}
   &\int_{\R^p}
 X_i^y
   \bigg\{y\mapsto
   h\Big((x,0)^{-1}*(y,\eta)\Big) \bigg\}
    \,\d\eta=
    \int_{\R^p}
 \Big(\widetilde{X}^{(y,\eta)}_i-R_i\Big)
   \bigg\{(y,\eta)\mapsto
   h\Big((x,0)^{-1}*(y,\eta)\Big) \bigg\}
    \,\d\eta\\
    &=
    \int_{\R^p}
 \widetilde{X}^{(y,\eta)}_i
   \bigg\{(y,\eta)\mapsto
   h\Big((x,0)^{-1}*(y,\eta)\Big) \bigg\}
    \,\d\eta
    -
    \int_{\R^p}
 R_i
   \bigg\{\eta\mapsto
   h\Big((x,0)^{-1}*(y,\eta)\Big) \bigg\}
    \,\d\eta \\
    & = \int_{\R^p}
    \big(\widetilde{X}_{i}h\big)\Big((x,0)^{-1}*(y,\eta)\Big)
    \,\d\eta -
    \int_{\R^p}
 R_i
   \bigg\{\eta\mapsto
   h\Big((x,0)^{-1}*(y,\eta)\Big) \bigg\}
    \,\d\eta.
        \end{align*}
  This will imply \eqref{sec.one:mainThm_defGammakkkkkkkeq1} once we prove that the second integral
  in the far right-hand side is null. The last statement is a consequence of the following
  computation:
  \begin{align*}
   & \int_{\R^p}
 R_i
   \bigg\{\eta\mapsto h\Big((x,0)^{-1}*(y,\eta)\Big) \bigg\}
    \,\d\eta = \sum_{j=1}^p
    \int_{\R^p}\alpha_{i,j}(y,\eta)\,\frac{\de}{\de \eta_j}
    \bigg\{h\Big((x,0)^{-1}*(y,\eta)\Big) \bigg\}\,\d\eta \\
    & = \sum_{j=1}^p
    \int_{\R^p}\frac{\de}{\de \eta_j}\bigg\{\alpha_{i,j}(y,\eta)\,
    h\Big((x,0)^{-1}*(y,\eta)\Big) \bigg\}\,\d\eta = 0;
  \end{align*}
  the second equality is a consequence of the fact that $\alpha_{i,j}$ does not depend on $
  \eta_j$, while the
  last equality derives from
  $$\lim_{\eta_j\to\pm\infty}\alpha_{i,j}(y,\eta)\,
    h\Big((x,0)^{-1}*(y,\eta)\Big) =
    \alpha_{i,j}(y,\eta)\,\lim_{\eta_j\to\pm\infty}h\Big((x,0)^{-1}*(y,\eta)\Big) = 0.$$
    The latter fact follows from $\lim_{|z|\to\infty}h(z) = 0$ (a consequence of
    the $D_\lambda$-homogeneity of $h$ of negative degree), and since
    $(x,0)^{-1}*(y,\eta)$ goes to infinity as $|\eta_j|$ diverges.
\medskip

  (III) We prove \eqref{derixxxx0}. First, since $\Gamma$ is symmetric (see Theorem A) we have
  $$\Gamma(x,y) = \int_{\R^p}\Gamma_\G \Big((x,0)^{-1}*(y,\eta)\Big)\,\d \eta
  = \int_{\R^p}\Gamma_\G \Big((y,0)^{-1}*(x,\eta)\Big)\,\d \eta.$$
  By virtue of this last representation of $\Gamma$, we can argue exactly
  as in Step II: for every choice of indexes
  $j_1,\ldots,j_t\in\{1,\ldots,m\}$ one has
  \begin{align*}
   X^x_{j_1}\cdots X^x_{j_t} \big(\Gamma(\cdot;y)\big)(x)=
   \int_{\R^p}
   \big(\widetilde{X}_{j_1}\cdots \widetilde{X}_{j_t}\Gamma_\G \big)
   \Big((y,0)^{-1}*(x,\eta)\Big)\,\d \eta.
  \end{align*}

 (IV) We finally prove \eqref{derixxxxyyyy0}. By Step II,
 for any $i_1,\ldots,i_s\in\{1,\ldots,m\}$
 one gets
 $$X^y_{i_1}\cdots X^y_{i_s}\big\{y\mapsto\Gamma(x;y)\big\}
 = \int_{\R^p}
   \Big(\widetilde{X}_{i_1}\cdots \widetilde{X}_{i_s}\Gamma_\G \Big)\Big((x,0)^{-1}
   *(y,\eta)\Big) \,\d\eta;$$
  then, we apply the change of variable
  $\eta = \Phi_{x,y}(\zeta)$ in Lemma \ref{lem.tecnico}, obtaining
  (by \eqref{eq.tousechangeofvariables})
  \begin{align*}
   X^y_{i_1}\cdots X^y_{i_s}\big\{y\mapsto\Gamma(x;y)\big\}
  & = \int_{\R^p}
   \Big(\widetilde{X}_{i_1}\cdots \widetilde{X}_{i_s}\Gamma_\G \Big)\Big((x,\zeta)^{-1}
   *(y,0)\Big)\,\d \zeta \\[0.15cm]
  & = \int_{\R^p}
   \Big(\big(\widetilde{X}_{i_1}\cdots \widetilde{X}_{i_s}\Gamma_\G \big)
   \circ\iota\Big)\Big((y,0)^{-1}
   *(x,\zeta)\Big) \,\d \zeta.
   \end{align*}
   On account of
   an analogous $x$-derivative formulation
   of \eqref{sec.one:mainThm_defGammakkkkkkkeq1}
   (with $h = (\widetilde{X}_{i_1}\cdots \widetilde{X}_{i_s}\Gamma_\G )
   \circ\iota$),
   for every choice of $j_1,\ldots, j_t\in\{1,\ldots,m\}$ we obtain
   $$X^x_{j_1}\cdots X^x_{j_t}X^y_{i_1}\cdots X^y_{i_s} \Gamma(x;y)
   =
   \int_{\R^p}
   \bigg(\widetilde{X}_{j_1}\cdots \widetilde{X}_{j_t}\Big(\big(\widetilde{X}_{i_1}
   \cdots \widetilde{X}_{i_s}\Gamma_\G \big)\circ\iota\Big)\bigg)
   \Big((y,0)^{-1}*(x,\zeta)\Big) \,\d \zeta.$$
   This is precisely \eqref{derixxxxyyyy0}, and the proof is complete.
\end{proof}
 In the previous proof we used the technical Lemma \ref{lem.tecnico} concerning
 the operation $*$. First we need to closely scrutinize the construction of the lifting group in Theorem A:
\begin{remark} \label{rem.costruzioneG}
 In the sequel, we shall need to invoke the explicit construction
 of the group $\G$ in point (1) of Theorem A, which is now in order;
 for all the details, see \cite{BB}.

 First of all, since $X_1,\ldots,X_m$ are $\dela$-homogeneous
 of degree $1$, then $\mathfrak{a} :=
 \mathrm{Lie}(X)$ is nilpotent; as a consequence, if
 $\diamond$ denotes the Baker-Campbell-Hausdorff series
 on $\mathfrak{a}$ (boiling down to a finite sum), it is well known that
 $(\mathfrak{a},\diamond)$ is a Lie group (whose inversion
 map is $X\mapsto -X$). Moreover, as
 $X_1,\ldots,X_m$ are linearly independent in $\mathfrak{a}$,
 we can choose a basis for $\mathfrak{a}$ as
 $$\mathcal{E} = \{X_1,\ldots,X_m,X_{m+1},\ldots,X_N\}.$$
 We can equip $\RN$ with a Lie group structure $(\RN,\bullet)$
 by reading $\diamond$ in $\mathcal{E}$-coordinates, i.e.,
 \begin{equation} \label{eq.defiBullet}
 \txt\sum_{j = 1}^N(a\bullet b)_j X_j = \Big(\sum_{j = 1}^N
 a_jX_j\Big)\diamond\Big(\sum_{j = 1}^N
 b_jX_j\Big), \quad \text{for every $a,b\in\RN$}.
 \end{equation}
 For a future use, we set $a\cdot X := \sum_{j = 1}^Na_jX_j$
 for any $a\in\RN$. Folland \cite{Folland} proved that
 the map
 \begin{equation} \label{eq.defiPiFolland}
  \Pi:\RN\longto\R^n, \qquad
 \Pi(a) := \Phi^{a\cdot X}_1(0)
 \end{equation}
 is surjective; here, given a vector field $Y\in\mathfrak{a}$,
 we denote by $\Phi^Y_t(z)$ the flow of $Y$ at time $t$ starting
 from $z$ at $t = 0$. In \cite{BB} it is proved that there always
 exist indexes $j_1,\ldots,j_p\in\{1,\ldots,N\}$ such that
 \begin{equation} \label{eq.definizioneT}
  T:\RN\longto\RN, \qquad T(a) := \big(\Pi(a),a_{j_1},\ldots,a_{j_p}\big)
  \end{equation}
 is a smooth diffeomorphism. Finally, the group $\G = (\RN,*)$ is
 obtained as follows:
 \begin{equation*}
  z*z' := T\Big(T^{-1}(z)\bullet T^{-1}(z')\Big), \qquad\text{for every
 $z,z'\in\RN$}.
 \end{equation*}
 It can also be proved that $\G$ is a homogeneous group with
 dilations
 \begin{equation} \label{eq.dilatazioniG}
  D_\lambda:\RN = \R^n\times\R^p
 \longto \RN, \qquad
 D_\lambda(x,\xi)=(\delta_\lambda(x),E_\lambda(\xi)),
 \end{equation}
 where $E_\lambda(\xi)=(\lambda^{\tau_1}\xi,\ldots,\lambda^{\tau_p}\xi_p)$
 for suitable integers $1\leq \tau_1\leq \cdots \leq \tau_p$.
 We point out that the exponents of $D_\lambda$ are not increasingly ordered (but this is true of the
 $\sigma_i$'s and of $\tau_j$'s separately). This will cause some subtleties
 in handling the group structure of $\G$ (see, e.g., Lemma \ref{lem.tecnico}).
\end{remark}
 The proof of the next lemma is quite delicate, due to the
 lack of ordering of the exponents of the dilation $D_\lambda$ of $\G$.
 \begin{lemma} \label{lem.tecnico}
  For every $x,y\in\R^n$ and every $\eta\in\R^p$, we set
  $$\theta(x,y,\eta) := (x,0)*(x,\eta)^{-1}*(y,0).$$
  Denoting by $\pi$ the projection of $\R^n\times\R^p$ onto $\R^p$, the map
  $$\eta\mapsto \Phi_{x,y}(\eta) := \pi(\theta(x,y,\eta))$$
  is a smooth diffeomorphism of $\R^p$ \emph{(}for any fixed $x,y\in\R^n$\emph{)}
  whose Jacobian determinant is $\pm 1$. Furthermore,
  the following identity holds:
  \begin{equation} \label{stefanoid}
    (x,0)^{-1}*(y,\Phi_{x,y}(\zeta )) = (x,\zeta )^{-1}*(y,0),
    \quad \text{for every $x,y\in\R^n$ and every $\zeta \in\R^p$}.
  \end{equation}
  In particular,
  $y_1,\ldots,y_n$ are the first
  $n$ components of $\theta(x,y,\eta)$.
  \end{lemma}
  In other words, the change of variable
  $$\eta = \Phi_{x,y}(\zeta)$$
  satisfies $\d \eta = \d \zeta$ and
  \begin{equation} \label{eq.tousechangeofvariables}
   (x,0)^{-1}*(y,\eta) = (x,\zeta)^{-1}*(y,0).
  \end{equation}
 \begin{proof}
  We follow the notation in Remark \ref{rem.costruzioneG}.
  As observed at the end of that remark, the last $p$ exponents
  of the dilation $D_\lambda$ of $\G$ are increasingly ordered; as a consequence,
  arguing as in \cite[Theorem 1.3.15]{BLUlibro},
  it is not difficult to check that (for every $k = 1,\ldots,p$)
  \begin{equation} \label{eq.thetanpiuk}
  \theta_{n+k}(x,y,\eta) = \sum_{j\,:\,\tau_j = \tau_k}c_j\eta_j + Q_k(x,y,\eta),
  \end{equation}
  where the $c_j$'s are real numbers and $Q_k$ is a polynomial only depending
  the $\eta_j$'s such that $\tau_j < \tau_k$. Since $\theta(0,0,\eta) = (0,\eta)^{-1}$
  (whose components are $E_\lambda$-homogeneous functions), one can reproduce
  the arguments in \cite[Corollary 1.3.16]{BLUlibro} and infer that
  \begin{equation} \label{eq.thetanpiuk2}
   \theta_k(0,0,\eta) = -\eta_k + q_k(\eta),
  \end{equation}
  where $q_k$ is a polynomial only depending
  the $\eta_j$'s such that $\tau_j < \tau_k$. By gathering together
  \eqref{eq.thetanpiuk} and \eqref{eq.thetanpiuk2} we obtain
  $$\theta_{n+k}(x,y,\eta) = -\eta_k + Q_k(x,y,\eta),$$
  which readily proves that $\Phi_{x,y}$ is a polynomial diffeomorphism whose Jacobian determinant
  is $\pm 1$.

  We now prove \eqref{stefanoid}, which is equivalent to
  $$(y,\Phi_{x,y}(\zeta )) = (x,0)*(x,\zeta )^{-1}*(y,0) = \theta(x,y,\zeta ).$$
  On account of the very definition of $\Phi_{x,y}$, this last identity is
  will follow if we prove that
  $y_1,\ldots, y_n$ are the first $n$ components of $\theta(x,y,\zeta )$.
  We now invoke the explicit construction of the operation $*$
  in Remark \ref{rem.costruzioneG}, whose notation we fully inherit.
  In this notation, we have
  $$\theta(x,y,\zeta ) = T\Big(T^{-1}(x,0)\bullet T^{-1}\big((x,\zeta )^{-1}\big)\bullet
  T^{-1}(y,0)\Big).$$
  As a consequence, we need to prove that
  (see the notation in \eqref{eq.definizioneT})
  $$\Pi\Big(T^{-1}(x,0)\bullet T^{-1}\big((x,\zeta )^{-1}\big)\bullet
  T^{-1}(y,0)\Big) = y.$$
  Let then $a,b,c\in\RN$
  be such that
  $T(a) = (x,0)$, $T(b) = (x,\zeta )$ and $T(c) = (y,0)$. In particular, we have
  $\Pi(a) = \Pi(b) = x$ and $\Pi(c) = y$. Since, by definition,
  $T$ is a homomorphism of $(\RN,\bullet)$ onto $(\RN,*)$,
  and since the inversion map of $(\RN,\bullet)$ is $z\mapsto -z$
  (see Remark \ref{rem.costruzioneG}), we can write
  $$T^{-1}(x,0)\bullet T^{-1}\big((x,\zeta )^{-1}\big)\bullet
  T^{-1}(y,0)
  = a\bullet(-b)\bullet c;$$
  as a consequence, by the very definitions
  of $\bullet$ and $\Pi$ (see \eqref{eq.defiBullet}
  and \eqref{eq.defiPiFolland}), we have
  $$\Pi\big(a\bullet(-b)\bullet c\big) =
  \Phi_1^{(a\bullet(-b)\bullet c)\cdot X}(0)
  = \Phi_1^{(a\cdot X)\diamond (-b\cdot X)\diamond (c\cdot X)}(0).$$
  By exploiting
  the Baker-Campbell-Hausdorff
  Theorem for ODEs
  (see \cite[Chapter 13]{BiagiBonfBook}) we obtain
  \begin{align*}
   \Phi_1^{(a\cdot X)\diamond (-b\cdot X)\diamond (c\cdot X)}(0)
   = \Phi_1^{c\cdot X}
   \Big(\Phi_1^{-b\cdot X}\big(\Phi_1^{a\cdot X}(0)\big)\Big)
   = \Phi_1^{c\cdot X}\big(\Phi_1^{-b\cdot X}(x)\big)
   = \Phi_1^{c\cdot X}(0) = y.
  \end{align*}
  The second equality follows from
  $\Pi(a) = x$; the third equality
  is a consequence of $\Phi_1^{b\cdot X}(0) = x$ together with
  the semigroup property $(\Phi_t^Y)^{-1} = \Phi_t^{-Y}$; the last equality
  follows from $\Pi(c) = y$.
 \end{proof}
 From now on, we adopt our abused notation
 $d_{\widetilde{X}}(z,z') = d_{\widetilde{X}}(z^{-1}*z')$.
\begin{proposition}\label{Thm 00}
  For every integer $r\geq0$ there exists $c>0$ \emph{(}depending on $r$ and on
  the set $X$\emph{)} such that, for every $x,y\in\mathbb{R}^{n}$
 \emph{(}with $x\neq y$\emph{)}, one has
\begin{equation}\label{stimadel_THEO0}
 \Big\vert Z_1\cdots Z_r\Gamma(x;y) \Big\vert
  \leq c
   \int_{\mathbb{R}^{p}}
    d_{\widetilde{X}}^{2-Q-r}\Big(  (x,0)^{-1}*(y,\eta)  \Big)\,\d\eta,
\end{equation}
 for any choice of $Z_1,\ldots,Z_r$ as in \eqref{goalZZZZZ}.
 Here, we remind that $Q$ is the homogeneous dimension of the Carnot group $\G=(\RN,*,D_\lambda)$.
 In particular, the map
$$
\eta\mapsto
 d_{\widetilde{X}}^{2-Q-r}\Big(  (x,0)^{-1}*(y,\eta)  \Big)
$$
 belongs to $L^{1}(\mathbb{R}^{p})$ for every $x\neq y\in \mathbb{R}^{n}$ and every $r\geq0$.
\end{proposition}
\begin{proof}
 When $r=0$ (so that no derivatives apply), \eqref{stimadel_THEO0} is a simple consequence of
  \eqref{sec.one:mainThm_defGamma22222} and
\begin{equation}\label{equivalaaaanza}
 c^{-1}\,d_{\widetilde{X}}(0,z)^{2-Q}\leq \Gamma_\G (z)\leq c\,d_{\widetilde{X}}(0,z)^{2-Q},\qquad \forall\,\,z\in \R^{N}\setminus\{0\},
\end{equation}
  where $c\geq 1$ is a constant only depending on the group $\G$ and the set $\widetilde{X}=\{\widetilde{X}_1,\ldots,\widetilde{X}_m\}$; the latter estimate
  trivially follows from the fact that
   $\Gamma_\G (z)^{1/(2-Q)}$ and $d_{\widetilde{X}}(0,z)$
   are homogeneous norms on $\G$, and the
  equivalence of all homogeneous norms on $\G$ (see e.g., \cite[Prop.\,5.1.4]{BLUlibro}).

  When $r \geq 1$,
  a repeated use of the map $\iota$ in \eqref{derixxxx0}-\eqref{derixxxxyyyy0} will allow us
  to express the therein integrands as functions
  of $(x,0)^{-1}*(y,\eta)$: this will lead to a unitary proof
  of \eqref{goal}.
  Indeed, we claim that \eqref{derixxxx0}-\eqref{derixxxxyyyy0} can be rewritten as follows:
  \begin{align}
   & X^x_{j_1}\cdots X^x_{j_t} \big(\Gamma(\cdot;y)\big)(x)=
   \int_{\R^p}
   \Big(\big(\widetilde{X}_{j_1}\cdots
   \widetilde{X}_{j_t}\Gamma_\G \big)\circ\iota\Big)
   \Big((x,0)^{-1}*(y,\eta)\Big) \,\d\eta\,;\label{derixxxx}\\[0.2cm]
   & X^x_{j_1}\cdots X^x_{j_t}X^y_{i_1}\cdots X^y_{i_s} \Gamma(x;y)\label{derixxxxyyyy}  \\
   &\qquad \qquad=
   \int_{\R^p}
   \bigg\{\bigg(\widetilde{X}_{j_1}\cdots \widetilde{X}_{j_t}\Big(\big(\widetilde{X}_{i_1}
   \cdots \widetilde{X}_{i_s}\Gamma_\G \big)\circ\iota\Big)\bigg)\circ \iota \bigg\}
   \Big((x,0)^{-1}*(y,\eta)\Big) \,\d\eta\,.\nonumber
\end{align}
  As for \eqref{derixxxx}, it suffices to apply to \eqref{derixxxx0}
  the change of variable $\eta = \Phi_{y,x}(\zeta )$ in Lemma \ref{lem.tecnico}, together with
  \eqref{eq.tousechangeofvariables} (with $x,y$ interchanged).
  As for \eqref{derixxxxyyyy}, we start from
  \eqref{derixxxxyyyy0} and we argue as above (plus another insertion of the $\iota$ map). \medskip

  We now exploit \eqref{deriyyyyy} (with $s=r$), \eqref{derixxxx} (with $t=r$) and
  \eqref{derixxxxyyyy} (with $s=r_1$, $t=r_2$ and $r_1+r_2=r$):
  in all three cases, the functions
  $$
   \widetilde{X}_{i_1}\cdots \widetilde{X}_{i_s}\Gamma_\G ,\qquad
   \big(\widetilde{X}_{j_1}\cdots \widetilde{X}_{j_t}\Gamma_\G \big)\circ\iota,\qquad
   \Big(\widetilde{X}_{j_1}\cdots \widetilde{X}_{j_{r_1}}\Big(\big(\widetilde{X}_{i_1}\cdots
   \widetilde{X}_{i_{r_2}}\Gamma_\G \big)\circ\iota\Big)\Big)\circ \iota
     $$
     are smooth out of the origin of $\R^N$, and $D_\lambda$-homogeneous of degree
     $2-Q-r$; thus, for simple homogeneity arguments,
     their absolute values are bounded from above by
     $c\,d_{\widetilde{X}}^{2-Q-r}(0,\cdot)$,
     where $c>0$ is a constant only depending on $r$ (and the system $X$).
     This readily gives \eqref{stimadel_THEO0}.

     As for the last statement of the proposition, it follows from (I)
     in the proof of Lemma \ref{Thm 0}, taking into account that
     $2-Q-r<q-Q$ (as $q>2$).
\end{proof}
 On account of Proposition \ref{Thm 00},
 the next step towards the proof of Theorem \ref{thm.estimatesDERIVgamma} is to
 estimate the integral in the right-hand side of \eqref{stimadel_THEO0}.
 This is accomplished as follows:
 \begin{itemize}
   \item[-] the estimate of the integral for arbitrary $x,y$ will follow by a
    homogeneity argument as in Remark \ref{rem.usehomog}, once it is proved for $x,y$ in some fixed compact neighborhood $K$ of $0\in\R^n$;
   \item[-] for $x\neq y\in K$, the right-hand side of \eqref{stimadel_THEO0} will be estimated by splitting the
   integral in the two parts $\{|\eta|\geq\delta\}$ and $\{|\eta|<\delta\}$, where $\delta > 0$ is some fixed positive number.
 \end{itemize}
    The two needed estimates are provided in the following
   Propositions \ref{Prop 1} and \ref{Prop 2}, which we now state.
 \begin{proposition} \label{Prop 1}
 For every compact neighborhood of the origin, say $K\subset\mathbb{R}^{n}$,
 for every integer $r\geq0$ and every $\delta > 0$, there exists $C_1 = C_1(K,r,\delta)>0$
 such that
\begin{equation} \label{eq.mainestimstepI}
 \int_{|\eta| \geq \delta}d_{\widetilde{X}}^{2-Q-r}\left((x,0)
 ^{-1}*(y,\eta)\right)\,\d\eta\leq C_1\,\frac{d_{X}(x,y)^{2-r}}{\big|B_{X}(x,d_{X}(x,y))\big|},
\end{equation}
 for every $x,y\in K$ with $x\neq y$.
\end{proposition}
\begin{proposition} \label{Prop 2}
 For every compact neighborhood of the origin, say $K\subset\mathbb{R}^{n}$,
 for every integer $r\geq0$ and every $\delta > 0$, there exists $C_2 = C_2(r)>0$
 such that
\begin{equation} \label{eq.mainestimstepII}
 \int_{|\eta| <\delta}d_{\widetilde{X}}^{2-Q-r}
 \left((x,0)^{-1}*(y,\eta)\right)\d\eta\leq C_2\,\frac{d_{X}(x,y)^{2-r}}{\left\vert B_{X}\left(x,d_{X}(x,y)\right)\right\vert},
\end{equation}
for every $x,y\in K$ with $x\neq y$.
\end{proposition}
 With the above propositions at hand (whose proofs will be
 shortly provided), we can give the
 \begin{proof}[Proof of Theorem \ref{thm.estimatesDERIVgamma}]
 Let the notation of Theorem \ref{thm.estimatesDERIVgamma} be understood;
 we arbitrarily fix a compact neighborhood $K$ of the origin in $\R^n$ and a number
 $\delta > 0$.
 Fixing $x\neq y \in K$, we have the following estimate, based on
 Propositions \ref{Thm 00}, \ref{Prop 1} and \ref{Prop 2}:
 \begin{align}
 & \Big\vert Z_1\cdots Z_r\Gamma(x;y) \Big\vert
   \stackrel{\eqref{stimadel_THEO0}}{\leq} c
   \int_{\mathbb{R}^{p}}
    d_{\widetilde{X}}^{2-Q-r}\Big(  (x,0)^{-1}*(y,\eta)  \Big)\,\d\eta \notag\\
    & = c\,\int_{|\eta| \geq \delta}
    d_{\widetilde{X}}^{2-Q-r}\Big(  (x,0)^{-1}*(y,\eta)  \Big)\,\d\eta
    + c\,\int_{|\eta| < \delta}
    d_{\widetilde{X}}^{2-Q-r}\Big(  (x,0)^{-1}*(y,\eta)  \Big)\,\d\eta \notag\\[0.15cm]
    & \qquad \text{(we use \eqref{eq.mainestimstepI} and \eqref{eq.mainestimstepII}, and we set $C := c(r)\,\max\{C_1(K,r,\delta), C_2(r)\}$)} \notag\\
    & \leq C\,\frac{d_{X}(x,y)^{2-r}}{\big|B_{X}(x,d_{X}(x,y))\big|}. \label{eq.estimdausare}
\end{align}
 We remove the condition $x,y\in K$ by a homogeneity argument.
 It suffices to apply Remark \ref{rem.usehomog} to
 $$F(x,y):= \Big\vert Z_1\cdots Z_r\Gamma(x;y) \Big\vert
 \quad \text{and} \quad
 G(x,y):=C\,\frac{d_{X}(x,y)^{2-r}}{\big|B_{X}(x,d_{X}(x,y))\big|},$$
 with the choices (see the notation in Remark \ref{rem.usehomog})
 $m=2\,n$, $\Omega=\{(x,y)\in\R^{2n}:\,x\neq y\}$ and the dilation $M_\lambda$
 given by $G_\lambda$ in \eqref{eq.dilatazioniEFG}. We can apply the cited lemma, as
 $F$ and $G$ are both $M_\lambda$-homogeneous of degree $2-q-r$
 (a simple consequence of \eqref{sec.one:mainThm_defGamma3},
 \eqref{propertiesdCC} and the $\dela$-homogeneity of $Z_1,\ldots,Z_r$), and since
 $F\leq G$ is valid on $(K\times K)\cap\Omega$ (see \eqref{eq.estimdausare}).

 Finally, we prove the vanishing property in \eqref{eq.derGammavanish}.
 To this end, it is sufficient to prove that, for any fixed $x\in \R^n$,
 the right-hand side of \eqref{goal} goes to $0$ as $|y|\to \infty$.
 This is a simple consequence of \eqref{eq.NSWmodificata} which indeed gives
 $$\frac{d_{X}(x,y)^{2-r}}
   {\big\vert B_{X}(x,d_{X}(x,y)) \big\vert }\leq
   \frac{d_{X}(x,y)^{2-r}}
   {\gamma_{1}\,\sum_{k=n}^{q}f_{k}(x)\,d_X(x,y)^{k}}
   \leq \frac{1}{\gamma_1\,f_q\,d_X(x,y)^{q-2+r}}\longto 0\quad \text{as $|y|\to \infty$.}
  $$
 The latter follows from $f_q(x)=f_q>0$ and  $q-2+r\geq q-2>0$. This ends the proof.
 \end{proof}
 We now give the proofs of Propositions \ref{Prop 1} and \ref{Prop 2}.
\begin{proof} [Proof of Proposition \ref{Prop 1}]
  Let $K, r,\delta$ be as in the statement of the proposition. We consider the open set
  $\Omega:=\{(x,y,\eta)\in \R^{n}\times \R^{n}\times \R^p\,:\,(x,0)\neq (y,\eta)\}$, and we set
 $$g(x,y,\eta) := d_{\widetilde{X}}^{2-Q-r}\left(\left(x,0\right)^{-1}*\left(  y,\eta\right)  \right).$$
 If $S$ and $N$ are as in \eqref{eqSedNdausare}, we choose $\varepsilon = \varepsilon(\delta) > 0$ so small that
 $\{|\eta|\geq\delta\}\subseteq\{N(\eta)\geq \varepsilon\}$ and we choose
 $R = R(K) > 0$ so large that $K\subseteq \{S(x)\leq R\}$.
 Since
 $$T := \{x:S(x)\leq R\}\times\{y:S(y)\leq R\}\times\{\eta:N(\eta) = \varepsilon\}$$
 is compact and contained in $\Omega$, there exists $c = c(\varepsilon,R)> 0$ such that
 $g(x',y',\eta')\leq c$ for every
 $(x',y',\eta')\in T$. Now, if $x,y\in K$ and $N(\eta)\geq\varepsilon$, choosing
 $\lambda = \varepsilon/N(\eta)\leq 1$ we clearly have
 $$(x',y',\eta') := F_\lambda(x,y,\eta)\in T,$$
 where $F_\lambda$ is as in \eqref{eq.dilatazioniEFG}. As a consequence, since $g$ is $F_\lambda$-homogeneous of
 degree $2-Q-r$ we get
 $$g(x,y,\eta) = g\big(F_{1/\lambda}(x',y',\eta')\big) = \frac{1}{\lambda^{2-Q-r}}\,g(x',y',\eta') \leq
 \frac{c}{\varepsilon^{2-Q-r}}\,N(\eta)^{2-Q-r}.$$
 Summing up, we have
 \begin{align*}
  & \int_{|\eta| \geq \delta}d_{\widetilde{X}}^{2-Q-r}\left(\left(x,0\right)
^{-1}*\left(  y,\eta\right)  \right)\,\d\eta
\leq \int_{N(\eta)\geq\varepsilon}g(x,y,\eta)\,\d\eta \\
& \leq \frac{c}{\varepsilon^{2-Q-r}}\int_{\{N(\eta)\geq\varepsilon\}}N(\eta)^{2-Q-r}\,\d\eta
= c(\varepsilon,R)\,\varepsilon^{Q-q}\int_{N(\eta)\geq 1}N(\eta)^{2-Q-r}\,\d\eta =: C(K,r,\delta).
 \end{align*}
 Note that $C(K,r,\delta)$ is finite, since the function $N^{2-Q-r}$ is integrable on $\{N\geq 1\}$ (as one can readily deduce from
 $2-Q-r < q-Q$). The above estimate will give \eqref{eq.mainestimstepI} once we prove that
 \begin{equation} \label{eq.NSWserve}
\inf_{\begin{subarray}{c}
x,y\in K \\
x\neq y
\end{subarray}}\frac{d_{X}(x,y)^{2-r}}{\big|B_{X}(x,d_{X}(x,y))\big| } > 0.
 \end{equation}
 In order to prove \eqref{eq.NSWserve} we make use of Theorem B in the introduction: thanks to \eqref{eq.NSWmodificata}, we have
\begin{align} \label{eq.tojustifyn2}
  \sup_{\begin{subarray}{c}
x,y\in K \\
x\neq y
\end{subarray}}
 \frac{\big|B_{X}(x,d_X(x,y))\big|}{d_X(x,y)^{2-r}} \leq \gamma_2\sup_{\begin{subarray}{c}
x,y\in K \\
x\neq y
\end{subarray}}\bigg(\sum_{k=n}^{q}f_{k}(x)\,d_X(x,y)^{k+r-2}\bigg) =: M(K,r) < \infty,
\end{align}
since the functions $f_k$'s are continuous and, by assumption, $n + r- 2 \geq 0$. This implies at once
\eqref{eq.mainestimstepI} with the choice $C_1(K,r,\delta) := C(K,r,\delta)\,M(K,r)$, and the proof is complete.
\end{proof}
 \begin{proof} [Proof of Proposition \ref{Prop 2}]
   It
   follows by combining the next Lemmas \ref{Lemma NSW} and
   \ref{Lemma nonintegrale}, with the choice
   $\beta = 2-r$ (note that, since $r\geq 0$, we have
   $\beta\leq 2 < n+1$), and $C_2(r) := C_3(2-r)\,C_4(2-r)$.
 \end{proof}
\begin{lemma} \label{Lemma NSW}
 Let $K\subset\mathbb{R}^{n}$ be any compact neighborhood of the origin,
 let $\delta$ be any positive real number, and finally let $\beta\leq n+1$.

 Then there exist positive numbers $R_0=R_0(K,\delta)$ and $C_3 = C_3(\beta)$ such that
\begin{equation} \label{eq.lemmaSanchezCalleeNSW}
 \int_{|\eta| <\delta}d_{\widetilde{X}}^{\beta-Q}
 \left(\left(x,0\right)^{-1}*\left(  y,\eta\right)  \right)\d\eta\leq C_3
 \int_{d_X(x,y)}^{R_0}\frac{\rho^{\beta-1}}{\big|B_{X}(x,\rho)\big|}\,\d\rho
\end{equation}
 for every $x,y\in K$ with $x\neq y$. Estimate \eqref{eq.lemmaSanchezCalleeNSW} is meaningful, since
 $R_0$ is chosen in such a way that $R_0 > 2d_X(x,y)$ for all $x,y\in K$ \emph{(}and the integral in the right-hand side
 does not vanish\emph{)}.
\end{lemma}
\begin{proof}
 The following argument is adapted from \cite[Theorem 5]{NSW}.
 We shall use in a crucial way the estimate \eqref{eq.SanchezI} (with $\xi = 0$):
 there exists a structural constant $c_1 > 0$ such that
\begin{equation} \label{eq.Sancheztousedopo}
\Big|\{\eta\in\mathbb{R}^{p}: (y,\eta)  \in
 B_{\widetilde{X}} ((x,0), \rho)\}\Big|
  \leq c_{1}\frac{|B_{\widetilde{X}}((x,0), \rho)|}{|B_{X}(x,\rho)|},
\end{equation}
 for every $x,y\in \R^n$ and every $\rho>0 $.

 Let now $K,\delta,\beta$ be as in the statement of the lemma, and let $R_0=R_0(K,\delta)>0$ be such that
 \begin{equation} \label{eq.estimdXdXtildeusare}
  d_{\widetilde{X}}
  \left(\left(x,0\right)^{-1}*\left(y,\eta\right)\right) <
  \frac{R_0}{2}, \qquad\text{for any $x,y\in K$
  and any $|\eta|\leq \delta$}.
 \end{equation}
Henceforth, we fix $x\neq y$ in $K$. Since we have
$$
d_{X}(x,y)\stackrel{\eqref{eq.proiezionipalled}}{\leq} d_{\widetilde{X}}\left(\left(x,0\right)^{-1}*\left(y,\eta\right)\right)
 \stackrel{\eqref{eq.estimdXdXtildeusare}}{<} \frac{R_0}{2},
$$
 we choose the unique integer $k\geq1$ such that
\begin{equation} \label{eq.perfinire}
\frac{R_0}{2^{k+1}} \leq d_{X}(x,y)
< \frac{R_0}{2^k}.
\end{equation}
On account of \eqref{eq.estimdXdXtildeusare}, this last
estimate implies that
$$
\frac{R_0}{2^{k+1}}
\leq d_{\widetilde{X}}\left(\left(x,0\right)^{-1}*\left(y,\eta\right)\right)
< \frac{R_0}{2},
$$
for every $\eta\in\R^p$ with $|\eta| \leq \delta$.
Thus we have the following computation:
\begin{align*}
&  \int_{|\eta| <\delta}d_{\widetilde{X}}^{\beta-Q}
\left(\left(x,0\right)^{-1}*\left(  y,\eta\right)  \right)\d\eta
\\[0.2cm]
& \quad \leq
\int_{\big\{\eta\in\R^p:\,
\frac{R_0}{2^{k+1}}
\leq d_{\widetilde{X}}((x,0)^{-1}*(y,\eta))
< \frac{R_0}{2}\big\}}d_{\widetilde{X}}^{\beta-Q}
\left(\left(x,0\right)^{-1}*\left(  y,\eta\right)\right)\d\eta
\\[0.2cm]
&  \quad
 = \sum_{j=1}^{k}\int_{\R^p}
 d_{\widetilde{X}}^{\beta-Q}
 \left(\left(x,0\right)^{-1}*\left(  y,\eta\right)\right){\text{\Large{$\chi$}}}_{A_j}(\eta)\,\d\eta
 =: (\star),
\end{align*}
 where we have introduced the notation
 $$A_j := \bigg\{\eta\in\R^p:\,
 \frac{R_0}{2^{j+1}}
 \leq d_{\widetilde{X}}\left((x,0)^{-1}*(y,\eta)\right)
 < \frac{R_0}{2^j}\bigg\}.$$
 We now observe that:
\begin{itemize}
 \item for any $\eta\in A_j$ we have (see also \eqref{eq.misurapallehom})
 $$d_{\widetilde{X}}^Q\Big((x,0)^{-1}*(y,\eta)\Big)\geq
 \Big(\frac{R_0}{2^{j+1}}\Big)^Q = \omega_Q\,\Big|
 B_{\widetilde{X}}\Big((x,0),\frac{R_0}{2^{j+1}}\Big)\Big|;$$
 \item for any $\eta\in A_j$ one has
 $$d_{\widetilde{X}}^\beta\Big((x,0)^{-1}*(y,\eta)\Big)
 \leq c_\beta\,\Big(\frac{R_0}{2^{j}}\Big)^\beta, \qquad\text{where\quad}
 c_\beta = \begin{cases}
 1, & \text{if $\beta\geq 0$}, \\
 2^{-\beta}, & \text{if $\beta < 0$}.
 \end{cases}$$
\end{itemize}
Gathering together these facts, one has
 \begin{align*}
   (\star) & \,\,\,\,\leq\,\,\frac{c_\beta}{\omega_Q}\,\sum_{j=1}^{k}\frac{(2^{-j}R_0)^\beta}
  {\big|B_{\widetilde{X}}((x,0), 2^{-j-1}R_0)\big|}\,
 \int_{\mathbb{R}^{p}}\chi_{A_j}(\eta)\,\d\eta \\
 & \,\,\,\,\leq\,\, \frac{c_\beta}{\omega_Q}\,\sum_{j=1}^{k}\frac{(2^{-j}R_0)^\beta}
  {\big|B_{\widetilde{X}}((x,0), 2^{-j-1}R_0)\big|}\cdot\Big|\big\{\eta\in\mathbb{R}^{p}: (y,\eta)  \in
 B_{\widetilde{X}} ((x,0), 2^{-j}R_0)\big\}\Big| \\
 & \stackrel{\eqref{eq.Sancheztousedopo}}{\leq}
 \frac{c_1 c_\beta}{\omega_Q}\,\sum_{j=1}^{k}\frac{(2^{-j}R_0)^\beta}
  {\big|B_{\widetilde{X}}((x,0), 2^{-j-1}R_0)\big|}\cdot
  \frac{|B_{\widetilde{X}}((x,0), 2^{-j}R_0)|}{|B_{X}(x,2^{-j}R_0)|} \\
  & \stackrel{\eqref{eq.misurapallehom}}{=}
  C(\beta)\,\sum_{j=1}^{k}\frac{(2^{-j}R_0)^\beta}
  {|B_{X}(x,2^{-j}R_0)|} =: (2\star), \qquad \text{with $C(\beta) := \frac{2^Q c_1 c_\beta}{\omega_Q}$}.
\end{align*}
 We now claim that, for a fixed $x$, the (continuous) function
 $$
 (0,\infty)\ni \rho\mapsto\frac{\rho^{\beta-1}}{\big|B_{X}(x,\rho)\big|}
 $$
 is comparable to a (continuous) monotone decreasing function, say $g_x(\rho)$. Indeed, by
 \eqref{eq.NSWmodificata} we have
 \begin{equation} \label{eq.NSWdariferirebeta}
\frac{\big|B_{X}(x,\rho)\big|}{\rho^{\beta-1}}\approx \sum_{h=n}^{q}f_{h}(x)\,\rho^{h+1-\beta}, \qquad\text{for $\rho>0$},
 \end{equation}
 for suitable non-negative functions $f_h$'s. As a consequence, observing that
 the exponents of $\rho$ in \eqref{eq.NSWdariferirebeta} are all non-negative (as $\beta\leq n+1$ by assumption),
 the right-hand side of \eqref{eq.NSWdariferirebeta} is monotone increasing in $\rho$; this proves that
 $$g_x(\rho) := \Bigg(\sum\limits_{h=n}^{q}f_{h}(x)\,\rho^{h+1-\beta}\Bigg)^{-1},$$
 is monotone decreasing, fulfilling our claim. This gives the next chain of inequalities:
 \begin{align*}
   (2\star) & =
   C(\beta)\,\sum_{j=1}^{k}\frac{(2^{-j}R_0)^{\beta-1}}
  {|B_{X}(x,2^{-j}R_0)|}\cdot 2^{-j}R_0 \\
  & \leq \frac{C(\beta)}{\gamma_1}\,\sum_{j = 1}^{k}g_x(2^{-j}R_0)\cdot 2^{-j}R_0
   = \frac{2\,C(\beta)}{\gamma_1}\,\sum_{j = 1}^{k}g_x(2^{-j}R_0)\cdot2^{-j-1}R_0 \\
   & \leq \frac{2\,C(\beta)}{\gamma_1}\,\sum_{j = 1}^{k}
   \int_{2^{-j-1}R_0}^{2^{-j}R_0}g_x(\rho)\,\d \rho =
   \frac{2\,C(\beta)}{\gamma_1}\,\int_{2^{-k-1}R_0}^{2^{-1}R_0}g_x(\rho)\,\d \rho \\
   & \leq \frac{2\,\gamma_2C(\beta)}{\gamma_1}\,\int_{2^{-k-1}R_0}^{2^{-1}R_0}\frac{\rho^{\beta-1}}{\big|B_{X}(x,\rho)\big|}\,\d \rho
   = \frac{\gamma_2C(\beta)}{2^{\beta-1}\,\gamma_1}\,\,\int_{2^{-k}R_0}^{R_0}\frac{t^{\beta-1}}{\big|B_{X}(x,t/2)\big|}\,\d t \\
   & \leq C_3(\beta)\,\int_{d_X(x,y)}^{R_0}\frac{t^{\beta-1}}{\big|B_{X}(x,t)\big|}\,\d t, \qquad \text{with $C_3(\beta) := \frac{C_d\,\gamma_2C(\beta)}{2^{\beta-1}\,\gamma_1}$}.
 \end{align*}
 In the last inequality we used the positivity of $g_x(\rho)$ and \eqref{eq.perfinire},
 jointly with the doubling inequality \eqref{globaldoubling}. This ends the proof.
\end{proof}
%
%
\noindent Finally, the following lemma is proved in \cite[Lemma 3.1]{BBMP}, and we provide
the proof for completeness.
\begin{lemma} \label{Lemma nonintegrale}
 For every $\beta\in\R$ there exists a constant $C_4 = C_4(\beta) > 0$ such that, for every $0<a<b$ and every $x\in \R^n$, the following inequalities hold:
\begin{equation} \label{eq.estimNSWfine}
 \int_{a}^{b}\frac{\rho^{\beta-1}}{\vert B_{X}(x,\rho)\vert}\,\d \rho
 \leq
\begin{cases}
  \displaystyle C_4(\beta)\frac{a^{\beta}}{\vert B_{X}(x,a)\vert } & \text{for $\beta<n$}, \\[0.4cm]
  \displaystyle C_4(\beta)\frac{a^{n}}{\vert B_{X}(x,a)\vert}\,\log\left(\frac{b}{a}\right) & \text{for $\beta=n$}, \\[0.4cm]
  \displaystyle C_4(\beta)\frac{a^{n}}{\vert B_{X}(x,a)\vert}\,b^{\beta-n} & \text{for $\beta>n$}.
\end{cases}
\end{equation}
\end{lemma}
 \noindent
 We shall apply this lemma with the choice $a=d_X(x,y)$ (with $x\neq y$) and $b=R_0>d_X(x,y)$.
 Moreover, in order to prove Theorem \ref{thm.estimatesDERIVgamma}, we shall take $\beta=2-r$ (with $r\geq 0$) which falls in the first estimate in
 \eqref{eq.estimNSWfine}, since $2-r\leq2 < n$ in our case.
 \begin{proof}
 Starting from \eqref{eq.NSWmodificata}, one can obtain the following estimate
 (see also \cite[Remark 2.12]{BBMP}):
$$
 \vert B_{X}(x,\rho)\vert \geq \frac{\gamma_1}{\gamma_2}\,\big\vert B_{X}(x,a)\big\vert
 \left(\frac{\rho}{a}\right)^{n} \quad \text{for $\rho\geq a$}.
$$
As a consequence, if $\beta<n$ we have
\begin{align*}
 &\int_{a}^{b}\frac{\rho^{\beta-1}}{\vert B_{X}(x,\rho)\vert}\,\d\rho   \leq \frac{\gamma_2}{\gamma_1}\,\frac{a^{n}}
 {\big\vert B_{X}(x,a)\big\vert}\,
 \int_{a}^{b}\frac{\rho^{\beta-1}}{\rho^{n}}\,\d\rho  =\frac{\gamma_2}{\gamma_1}\,\frac{a^{n}}
 {\big\vert B_{X}(x,a)\big\vert}\,
 \left(\frac{a^{\beta-n}-b^{\beta-n}}{n-\beta}\right) \\
&\qquad  \leq \frac{\gamma_2}{(n-\beta)\gamma_1}\,\frac{a^{n}}
 {\big\vert B_{X}(x,a)\big\vert}\,a^{\beta-n}  = C_4(\beta)\,\frac{a^{\beta}}
 {\big\vert B_{X}(x,a)\big\vert}, \quad \text{where $C_4(\beta) := \frac{\gamma_2}{(n-\beta)\gamma_1}$}.
\end{align*}
 The case $\beta > n$ is completely analogous. Finally, if $\beta=n$, the above computation has to be modified
 according to
 $\displaystyle
 \int_{a}^{b}\frac{\rho^{\beta-1}}{\rho^{n}}\,\d\rho
 = \log\left(\frac{b}{a}\right)$.
 This ends the proof.
\end{proof}
 In due course of the arguments of this section, we have incidentally proved the following:
\begin{corollary}\label{corollatastimaamam}
 Let $n>2$. For every integer $r\geq0$, there exists $C_3 = C_3(r)>0$ such that
\begin{equation} \label{eq.mainestimstepII33333333}
 \int_{\R^p}d_{\widetilde{X}}^{2-Q-r}
 \left((x,0)^{-1}*(y,\eta)\right)\d\eta\leq C_3\,\frac{d_{X}(x,y)^{2-r}}{\left\vert B_{X}\left(x,d_{X}(x,y)\right)\right\vert},
\end{equation}
for every $x,y\in \R^n$ with $x\neq y$.
\end{corollary}

\section{Global lower estimates of $\Gamma$}\label{sec:loweresimates}
 The aim of this section is to prove  the following
\begin{theorem}\label{th.stimebasso}
 Let $\Gamma$ be the fundamental solution of $\LL$ as in \eqref{sec.one:mainThm_defGamma}; let
  us suppose that $n>2$.
 Then, there exists a \emph{(}structural\emph{)} constant $C>0$ such that
\begin{equation}\label{eq.1gammabasso}
 \Gamma(x;y) \geq C\frac{d_{X}(x,y)^{2}}{\big|B_{X} \big(x,d_{X} (x,y)\big)\big|}\quad \text{for every $x\neq y$
 in $\R^n$.}
\end{equation}
In particular, for every fixed $x\in\R^n$ we have
\begin{equation}\label{vanishingGammanuguale2}
   \lim_{y\to x}\Gamma(x;y) = \infty.
\end{equation}
\end{theorem}
 \noindent We observe that the function in the right-hand side of
 \eqref{eq.1gammabasso} is not symmetric, as is instead true of $\Gamma(x;y)$; on the other hand,
 an inequality analogous to \eqref{eq.1gammabasso} holds true with interchanged $x$ and $y$, as
 we already pointed out in Remark \ref{rem.simmetriafrac}. \medskip

 The proof of Theorem \ref{th.stimebasso} relies on the next two propositions.
\begin{proposition}\label{prop1.stimabasso}
  Let $K\subset\mathbb{R}^{n}$ be any compact neighborhood of the origin.
 Then there exist a positive number $R=R(K)$ and a \emph{(}structural\emph{)} constant
 $C_1 > 0$ such that
\begin{equation}\label{eq.11propostim1}
 \Gamma(x;y) \geq C_1\int_{d_X(x,y)}^{R}\frac{\rho}{|B_X(x,\rho)|}\,\d \rho,
\end{equation}
  for every $x,y\in K$ with $x\neq y$.
 More precisely, $R$ can be chosen in such a way that $d_X(x,y)<R/2$ for all $x,y\in K$.
\end{proposition}
\begin{proposition}\label{prop2.stimabasso}
 There exists a \emph{(}structural\emph{)} constant
  $C_2>0$ such that, for all $x\in\R^n$, one has
\begin{equation}\label{eq.11propostim2}
 \int_{a}^{b}\frac{\rho}{|B_X(x,\rho)|}\,\d \rho\geq C_2\,
 \frac{a^{2}}{\big|B_{X} \big(x,a\big)\big|}, \qquad\text{for any $0<a<b/2$}.
\end{equation}
\end{proposition}
  Before giving the proofs of Propositions  \ref{prop1.stimabasso}
 and \ref{prop2.stimabasso}, we show how they together provide the
\begin{proof}[Proof of Theorem \ref{th.stimebasso}]
 The proof of \eqref{eq.1gammabasso} is a local-to-global argument, via homogeneity.
 Gathering together \eqref{eq.11propostim1} and
 \eqref{eq.11propostim2} (this is legitimate, since $d_X(x,y)<R/2$ for all $x,y\in K$), one gets
\begin{equation}\label{stimabassoserveinettermed1}
 \Gamma(x;y) \geq C\,\frac{d_{X}(x,y)^{2}}{\big|B_{X} \big(x,d_{X} (x,y)\big)\big|}\quad \text{for every $x\neq y$
 in $K$,}
\end{equation}
 where $C=C_1\,C_2$ and $K$ is some fixed compact neighborhood of the origin in $\R^n$.

 Next we apply Remark \ref{rem.usehomog} with the choices of $m$, $\Omega$, $M_\lambda$
 as in the proof of Theorem \ref{thm.estimatesDERIVgamma}, and with the functions
 $G,F$ given by the two members of the inequality \eqref{stimabassoserveinettermed1}
 (valid on $(K\times K)\cap\Omega$), as these functions are both
 $M_\lambda$-homogeneous of degree $2-q$ (a consequence of
 \eqref{sec.one:mainThm_defGamma3} and \eqref{propertiesdCC}).

  Finally, we prove the blow-up property in \eqref{vanishingGammanuguale2}.
  Owing to Theorem B, we have
\begin{align*}
 \liminf_{y\to x}
 \Gamma(x;y) & \stackrel{\eqref{eq.1gammabasso}}{\geq}
 C\liminf_{y\to x} \frac{d_{X}(x,y)^{2}}{\big|B_{X} \big(x,d_{X} (x,y)\big)\big|}
 \stackrel{\eqref{eq.NSWmodificata}}{\geq} \frac{C}{\gamma_2}\liminf_{y\to x}\frac{d_X(x,y)^2}{\sum\limits_{k=n}^{q}f_{k}(x)\,d_X(x,y)^k}.
\end{align*}
 The latter $\liminf$ is $\infty$, due to the assumption $n > 2$ and the fact that $f_k\geq 0$ (and $f_q>0$).
\end{proof}
 We then give the proofs of
 Propositions  \ref{prop1.stimabasso}
 and \ref{prop2.stimabasso}.
\begin{proof}[Proof of Proposition \ref{prop1.stimabasso}]
  In this proof we make use
  of \eqref{eq.SanchezII} in Theorem C (with the choice $y = x$): there exists a structural
  constant $c_2 > 0$ such that
\begin{equation}\label{theremCequsata}
 \Big|\Big\{\eta\in\mathbb{R}^{p}: d_{\widetilde{X}}\big((y,\eta),(y,\xi)\big)<r\Big\}\Big|
   \geq c_{2}\frac{|B_{\widetilde{X}}((y,\xi), r)|}{|B_{X}(y, r)|},\qquad \forall\,\,y\in \R^n,\,\,\xi\in \R^p,\,\,r>0.
\end{equation}
 Let $K$ be any fixed compact neighborhood of $0\in\R^n$, and let us
 choose $R  = R(K) > 0$ such that
\begin{equation*}
 d_X(x,y)<\frac{R}{2} \quad \text{for any $x,y\in K$.}
\end{equation*}
  Next, we fix different points $x,y$ in $K$ and we let $k\geq 1$ be the unique integer such that
\begin{equation}\label{sceltaKrpo22}
 \frac{R}{2^{k+1}}\leq d_X(x,y)< \frac{R}{2^k}\,.
\end{equation}
 The latter and \eqref{eq.proiezionipalled}  ensure that
\begin{equation}\label{sceltaKrpo223}
  \frac{R}{2^{k+1}}\leq  d_{\widetilde{X}}\Big((x,0)^{-1}*(y,\eta)\Big), \qquad \text{for every $\eta\in\R^p$}.
\end{equation}
 By the representation \eqref{sec.one:mainThm_defGamma22222} of $\Gamma$ and owing to
 \eqref{equivalaaaanza}, we have
\begin{align*}
    \Gamma(x;y) &\geq c^{-1}\int_{\R^p}d_{\widetilde{X}}^{2-Q}\Big((x,0)^{-1}*(y,\eta)\Big)\,\d\eta\\
     &\geq c^{-1}\int_{\big\{\eta:\,\, d_{\widetilde{X}}((x,0),(y,\eta))< 2R\big\}}
              d_{\widetilde{X}}^{2-Q}\Big((x,0)^{-1}*(y,\eta)\Big)\,\d\eta\\
              &=c^{-1}\sum_{j=0}^{k+1}
              \int_{\Omega_j}
              d_{\widetilde{X}}^{2-Q}\Big((x,0)^{-1}*(y,\eta)\Big)\,\d\eta=:(\star),
              \end{align*}
 where we have used \eqref{sceltaKrpo223}, together with the notation
 $$\Omega_j :=
 \bigg\{\eta\in\R^p:\,\, \frac{R}{2^{j}}\leq d_{\widetilde{X}}((x,0),(y,\eta))<\frac{R}{2^{j-1}}\bigg\}.$$
 From the very definition of $\Omega_j$ we then have
\begin{align*}
 (\star)&\geq
  c^{-1}\sum_{j=0}^{k+1}
  \left(\frac{R}{2^{j-1}}\right)^{2-Q}
  \big|\Omega_j \big|
  \geq c^{-1}\sum_{j=0}^{k}
  \left(\frac{R}{2^{j-1}}\right)^{2-Q}
  \big|\Omega_j \big| =:(2\star).
  \end{align*}
 We now claim the following assertion: for any $j\in\{0,\ldots,k\}$, there exists
 $\eta_j\in \R^p$ such that
\begin{equation}\label{jfkhshaew}
 \bigg\{\eta\in\R^p:\,\, d_{\widetilde{X}}((y,\eta),(y,\eta_j))<\frac{R}{2^{j+1}}\bigg\}\subseteq \Omega_j.
\end{equation}
 Indeed,  we first observe that
 (see \eqref{eq.proiezionipalled})
 the projection of $B_{\widetilde{X}}\big((x,0),R/2^{k}\big)\big)$ onto $\R^n$
 is precisely $B_X(x,R/2^k)$; thus, being $y\in B_X(x,R/2^k)$ (see \eqref{sceltaKrpo22}),
 there exists $\overline{\eta}\in \R^p\setminus\{0\}$ such that
\begin{equation}\label{gtchenonviene}
 d_{\widetilde{X}}((x,0),(y,\overline{\eta}))<R/2^k.
\end{equation}
 We consider the function $g:[1,\infty)\to (0,\infty)$ defined as follows
 $$g(t):= d_{\widetilde{X}}\big((x,0),(y,t\overline{\eta})\big).$$
 Fixing $j\in \{0,\ldots,k\}$, from \eqref{gtchenonviene} we get $g(1)<R/2^k\leq R/2^j$;
 on the other hand, for $t$ large enough, the point $(y,t\overline{\eta})$ cannot lie
 in $\overline{B}_{\widetilde{X}}((x,0), R/2^{j-1})$, i.e., $g(t)> R/2^{j-1}$. As a consequence (by continuity), there exists $t_j\geq1 $
 such that $$g(t_j)=\frac{1}{2}\left(\frac{R}{2^{j}}+\frac{R}{2^{j-1}}\right)=\frac{3}{2}\,\frac{R}{2^j}.$$
 Setting $\eta_j:=t_j\overline{\eta}$, the above equality means that
 $$d_{\widetilde{X}}\big((x,0),(y,\eta_j)\big)=\frac{3}{2}\,\frac{R}{2^j}. $$
 As simple argument based on the triangle inequality proves that
 $$B_{\widetilde{X}}\big((y,\eta_j),R/2^{j+1}\big)\subseteq
  B_{\widetilde{X}}\big((x,0),R/2^{j-1}\big)\setminus B_{\widetilde{X}}\big((x,0),R/2^{j}\big),$$
  which readily implies \eqref{jfkhshaew}. Owing to \eqref{theremCequsata}, the latter gives
\begin{equation*}
 |\Omega_j|\geq \left|\bigg\{\eta\in\R^p:\,\, d_{\widetilde{X}}((y,\eta),(y,\eta_j))<\frac{R}{2^{j+1}}\bigg\}\right|
 \geq c_{2}\frac{\big|B_{\widetilde{X}}\big((y,\eta_j), {R}/{2^{j+1}}\big)\big|}{\big|B_{X}\big(y, {R}/{2^{j+1}}\big)\big|}.
\end{equation*}
 If we insert this in $(2\star)$, we get
\begin{align*}
 (2\star)&\,\,\,\geq\,\,
  \frac{c_2}{c}\sum_{j=0}^{k}
    \left(\frac{R}{2^{j-1}}\right)^{2-Q}
 \cdot
  \frac{\big|B_{\widetilde{X}}\big((y,\eta_j), {R}/{2^{j+1}}\big)\big|}{\big|B_{X}\big(y, {R}/{2^{j+1}}\big)\big|}\\
  &\stackrel{\eqref{eq.misurapallehom}}{=}
  \frac{c_2\,\omega_Q}{c\,4^Q}\sum_{j=0}^{k}
   \frac{({R}/{2^{j-1}})^2}{\big|B_{X}\big(y, {R}/{2^{j+1}}\big)\big|}
   \geq
     C'_1\sum_{j=0}^{k}
   \frac{({R}/{2^{j}})^2}{\big|B_{X}\big(y, {R}/{2^{j}}\big)\big|}=:(3\star),
 \end{align*}
 where $C_1'=\frac{4c_2\,\omega_Q}{c\,4^Q}$. We now argue as in the proof of Lemma \ref{Lemma NSW}, thus getting
\begin{align*}
 (3\star)&\geq
 \frac{\gamma_1C_1'}{\gamma_2}\int_{R/2^k}^{2R}\frac{\rho}{|B_X(y,\rho)|}\,\d \rho
 \stackrel{\eqref{sceltaKrpo22}}{\geq}
 \frac{\gamma_1C_1'}{\gamma_2}\int_{2\,d_X(x,y)}^{2R}\frac{\rho}{|B_X(y,\rho)|}\,\d \rho\\
 &=
 \frac{4\gamma_1C_1'}{\gamma_2}\int_{d_X(x,y)}^{R}\frac{t}{|B_X(y,2t)|}\,\d t
 \stackrel{\eqref{globaldoubling}}{\geq}
  \frac{4\gamma_1C_1'}{\gamma_2C_d}\int_{d_X(x,y)}^{R}\frac{t}{|B_X(y,t)|}\,\d t.
 \end{align*}
 Summing up, we have proved that (setting $C_1=\frac{4\gamma_1C_1'}{\gamma_2C_d}$)
\begin{equation*}
 \Gamma(x;y) \geq C_1\int_{d_X(x,y)}^{R}\frac{t}{|B_X(y,t)|}\,\d t,\qquad \text{for every $x\neq y$ in $K$};
\end{equation*}
 finally, by interchanging $x$ and $y$ and bearing in mind that $\Gamma$ is symmetric, we get \eqref{eq.11propostim1}.
\end{proof}
 We are left with the
\begin{proof}[Proof of Proposition \ref{prop2.stimabasso}]
 We fix any $x\in\R^n$ and
 we consider the integrand function in
 \eqref{eq.11propostim2}:
 $$(0,\infty)\ni\rho\mapsto \frac{\rho}{|B_X(x,\rho)|}.$$
 In the proof of
 Lemma \ref{Lemma NSW} we showed that, for every $\rho > 0$, one has
\begin{equation} \label{eq.NSWmodificataTER}
 \frac{1}{\gamma_2}\,\phi(\rho) \leq \frac{\rho}{|B_X(x,\rho)|}\leq \frac{1}{\gamma_1}\,\phi(\rho), \qquad \text{where}\quad
 \phi(\rho):=\left(\sum_{h=n}^{q}f_{h}(x)\,\rho^{h-1}\right)^{-1}.
\end{equation}
 Moreover, we observe that $\phi$ enjoys the
 following properties:
\begin{enumerate}
  \item it is non-negative and monotone decreasing;
  \item it is reverse doubling, i.e., for every $r > 0$ one has
  $$\phi(2\,r)\geq \alpha\,\phi(r)\quad \text{with $\alpha := 2^{1-q}$.}$$
\end{enumerate}
 In order to get \eqref{eq.11propostim2}, it is sufficient to prove this claim:
 \emph{any function $\phi$ with the above properties \emph{(1)} and \emph{(2)} also satisfies the following estimate}
\begin{equation} \label{eq.phiintegralosa}
 \int_a^b \phi(\rho)\,\d\rho\geq \alpha^2\,a\,\phi (a),\quad
 \text{\emph{for any $0<a<b/2$.}}
\end{equation}
  This claim will prove \eqref{eq.11propostim2} as follows:
\begin{align*}
 \int_{a}^{b}\frac{\rho}{|B_X(x,\rho)|}\,\d \rho
 &\stackrel{\eqref{eq.NSWmodificataTER}}{\geq }
 \frac{1}{\gamma_2}
 \int_{a}^{b}\phi(\rho)\,\d \rho
 \stackrel{\eqref{eq.phiintegralosa}}{\geq }
 \frac{\alpha^2}{\gamma_2}\,
 \,a\,\phi(a)\\
  &\stackrel{\eqref{eq.NSWmodificataTER}}{\geq }
  C_2\,\frac{a^{2}}{\big|B_{X}(x,a)\big|} \qquad \text{with $C_2=\frac{\alpha^2\gamma_1}{\gamma_2}$.}
\end{align*}
 We are left to prove the above claim. Let $a\in(0,b/2)$ be fixed, and let $k\in\N$ be such that
\begin{equation}\label{RRRRRmezzi}
  \frac{b}{2^{k+1}}< a \leq \frac{b}{2^k}\,.
\end{equation}
 This gives the following computation
 \begin{align*}
    \int_a^b \phi(\rho)\,\d\rho &\geq
    \int_{b/2^k}^b \phi(\rho)\,\d\rho=
    \sum_{j=1}^k \int_{b/2^j}^{b/2^{j-1}} \phi(\rho)\,\d\rho
    \geq \sum_{j=1}^k \phi(b/2^{j-1})\,\frac{b}{2^j}\\
    &\geq \phi(b/2^{k-1})\,\frac{b}{2^k} \geq \alpha^2 \phi(b/2^{k+1})\,\frac{b}{2^k}
    \geq \alpha^2\,a\,\phi(a).
 \end{align*}
 Here, we repeatedly used (1) and (2), and \eqref{RRRRRmezzi}.
\end{proof}
\begin{example}
Let us consider the two vector fields on $\mathbb{R}^{n}$ ($n\geq3$)
 $$
 X_{1}=\partial_{x_{1}},\quad X_{2}=x_{1}\,\partial_{x_{2}}+x_{2}\partial_{x_{3}}+\ldots +x_{n-1}\partial_{x_{n}},
 $$
 and the corresponding PDO $\LL=X_1^2+X_2^2$ on $\mathbb{R}^{n}$.
 The vector fields $X_{1}$ and $X_{2}$ are homogeneous of degree $1$ with respect to the dilations
$$
 \delta_{\lambda}(x)=(\lambda x_{1},\lambda^{2}x_{2} ,\lambda^{3}x_{3},\ldots ,\lambda^{n}x_{n}),
$$
 which gives $q=n (  n+1 )  /2\geq 6$.
 Observe that
$$
Y_{k}:= [[[X_1,\underbrace{X_2],X_2]\cdots X_2}_{\text{$k$ times}}]= \de_{x_{k+1}},\quad \text{for $k=1,\ldots, n-1$.}
$$
 Thus, assumptions (H1)-to-(H3)
 of Section \ref{sec:introductionMain} are satisfied (here $N=n+1$).
 Following the construction in
 Remark \ref{Remark explicit f_k}, the possible choices of a basis of $\mathbb{R}^{n}$ are:
\begin{itemize}
  \item
  $X_{1},Y_{1},\ldots ,Y_{n-1}$, which has weight $1+2+\cdots+n=n (n+1 )/2$ and $f_{q}=1$;

  \item
 $X_{1},X_{2}$ and $n-2$ out of the $n-1$
 commutators $Y_{j}$. Denoting by $\mathcal{B}_{j}$ ($j=1,\ldots ,n-1$) the choice
 containing all the $Y_{k}$'s except for $k=j$, we have
\begin{align*}
 | \mathcal{B}_{j} |  =1+1+ (  2+3+\cdots +n )-(j+1)=q-j,\quad \text{and}\quad f_{q-j}(x)=| x_{j} | .
\end{align*}
\end{itemize}
 Hence
$$
 \Lambda (  x,\rho )  =\rho^{q}+\sum_{j=1}^{n-1} |x_{j}|\,\rho^{q-j}.
$$
 Thus, owing to our Theorem \ref{th.teoremone}-(III), we have the global estimates
$$
\Gamma (x;y) \approx \Big(d_{X}(x,y)^{q-2} +\sum_{j=1}^{n-1} |x_{j}|\,d_{X}(x,y)^{q-j-2}\Big)^{-1}.
$$

\end{example}

\section{The case $n = 2$} \label{sec:casen2}
  In this section we investigate the case $n = 2$, not covered in
  the previous sections. The obstructions imposed by taking $n = 2$
  are not only technical: jointly with the trivial example of
  the logarithmic fundamental solution of Laplace's operator in $\R^2$, we also have
  less trivial examples of fundamental solutions which (near the diagonal, at least) exhibit
  either a ``logarithmic-type'' behavior or a ``power-like'' behavior, depending on the point.
  As we shall detail in Example \ref{exa.grushink}, this happens for instance for the class of operators
  $$\mathcal{G} = (\de_{x_1})^2+ (x_1^k\,\de_{x_2})^2\quad (k\in \mathbb{N}).$$
  What is really pathological is the case $(n,r) = (2,0)$ in Theorem \ref{thm.estimatesDERIVgamma}, which
  means to obtain a global pointwise upper estimate of $\Gamma$; likewise, Theorem \ref{th.stimebasso} must be adapted in
  the case $n = 2$.
  On the contrary, the case $n = 2$ and $r\geq 1$ goes like in Theorem \ref{thm.estimatesDERIVgamma}, as
  the following result shows.
\begin{theorem} \label{thm.derivn2}
  Let $n = 2$ and let $\Gamma$ be the fundamental solution of $\LL$ in \eqref{sec.one:mainThm_defGamma}.
 Then, for any integer $r\geq 1$, there exists $C>0$ \emph{(}depending on $r$ and on the set $X=\{X_1,\ldots,X_m\}$\emph{)} such that
\begin{equation}\label{goaln2}
   \Big\vert Z_1\cdots Z_r\Gamma(x;y) \Big\vert
   \leq C\,
   \frac{d_{X}(x,y)^{2-r}}
   {\big\vert B_{X}(x,d_{X}(x,y)) \big\vert },
\end{equation}
 for any $x,y\in\mathbb{R}^{2}$ \emph{(}with $x\neq y$\emph{)} and any
 choice of
\begin{equation*}
 Z_1,\ldots,Z_r\in\Big\{X^x_1,\ldots,X^x_m,X^y_1,\ldots,X^y_m\Big\}.
\end{equation*}
In particular,
 for every fixed $x\in\R^2$ we have
 \begin{equation} \label{eq.dergammavanishn2}
  \lim_{|y|\to\infty}Z_{1}\cdots Z_{r}\Gamma (x;y) = 0.
 \end{equation}
  \end{theorem}
 \begin{proof}
  An inspection of the proofs of the results in Section \ref{sec:pointwise.estimates}
  shows that, when $n = 2$ and $r\geq 1$,
  \begin{itemize}
    \item Lemma \ref{Thm 0} remains unvaried, and
     Proposition \ref{Thm 00} follows from Lemma \ref{Thm 0};
    \item Proposition \ref{Prop 1} holds unvaried: indeed, in
    \eqref{eq.tojustifyn2} we have $M(K,r) < \infty$,
    as $n+r-2 = r\geq 1$;
    \item Lemma \ref{Lemma NSW} holds true for any $\beta\leq n+1 = 3$.
    This enables us to obtain Proposition \ref{Prop 2} (for $n = 2$ and $r\geq 1$),
  by combining Lemma \ref{Lemma NSW} with Lemma
  \ref{Lemma nonintegrale} (with the choice $\beta = 2-r < 2 = n$).
 \end{itemize}
  Finally, with Propositions \ref{Thm 00}-to-\ref{Prop 2} at hand, the proofs of
  \eqref{goaln2} and \eqref{eq.dergammavanishn2} follow as in Section \ref{sec:pointwise.estimates}.
 \end{proof}
 In the above arguments, we have incidentally proved the following:
\begin{corollary}\label{corollatastimaamambis44444}
 Let $n=2$. For every integer $r\geq 1$, there exists $C= C(r)>0$ such that
\begin{equation} \label{eq.mainestimstepII33333333jjjjjj}
 \int_{\R^p}d_{\widetilde{X}}^{2-Q-r}
 \left((x,0)^{-1}*(y,\eta)\right)\d\eta\leq C\,\frac{d_{X}(x,y)^{2-r}}{\left\vert B_{X}\left(x,d_{X}(x,y)\right)\right\vert},
\end{equation}
for every $x,y\in \R^2$ with $x\neq y$.
\end{corollary}
 As for the estimates of $\Gamma$, both from above and from below,
 one cannot expect global results: this is due to the fact that, when $n = 2$,
 $\Gamma$ can behave logarithmically near the diagonal and much differently off the diagonal.
 Thus,  we are firstly content with the following estimates (Theorem \ref{thm.estimateGamman2}) valid on compact sets $K$, uniformly as $x$ and $y$ vary in $K$.
 Secondly, we shall prove optimal estimates (in that both upper and lower bounds are of the same form) near the diagonal and depending upon the pole $x$
 (Theorem \ref{thm.estimateGamman2bis}).
 \begin{theorem} \label{thm.estimateGamman2}
 Let the assumptions of Theorem \ref{thm.derivn2} apply, and let $K\subseteq\R^n$ be any compact set. Then,
 there exist a structural constant $C_0 > 0$ and a real $R_0 = R_0(K) > 0$ such that
 \begin{equation} \label{goaln2r0}
   \Gamma(x;y)\leq C_0\,\frac{d_{X}(x,y)^{2}}{\big|B_{X}(x,d_{X}(x,y))\big|}
  \cdot \log\Big(\frac{R_0}{d_X(x,y)}\Big),
\end{equation}
 for every $x,y\in K$ with $x\neq y$.

 Moreover, there exist
 positive numbers $C_1 = C_1(K)$ and $R_1 = R_1(K)$ such that
 \begin{equation} \label{goaln2r0low}
  \Gamma(x;y)\geq C_1\,\log\Big(\frac{R_1}{d_X(x,y)}\Big)
\end{equation}
 for every $x,y\in K$ with $x\neq y$.
 In particular, for every fixed $x\in\R^2$ we have
\begin{equation} \label{eq.Gammapolexn2}
   \lim_{y\to x}\Gamma(x;y) = \infty.
\end{equation}
\end{theorem}
 It can be noticed that a serious obstruction in globalizing \eqref{goaln2r0}-\eqref{goaln2r0low}
 is the lack of homogeneity
 of logarithmic members in the right-hand sides.
 However, this does not prevent $\Gamma(x;y)$ from vanishing as $|y|$ (or $|x|$) goes to infinity,
 as is proved in \cite{BB} (see (3) in Theorem A).
 \begin{proof}
  We begin by establishing \eqref{goaln2r0}.
  Scrutinizing the proofs of the results in Section \ref{sec:pointwise.estimates} with $(n,r) = (2,0)$
 we see that
  \begin{itemize}
    \item Proposition \ref{Thm 00} follows from the representation
    formula \eqref{sec.one:mainThm_defGamma22222};
    \item Proposition \ref{Prop 1} holds unvaried, as $n+r-2 = 0$ (see again
    \eqref{eq.tojustifyn2});
    \item Lemma \ref{Lemma NSW} holds true with $\beta = 2 < n+1$.
 \end{itemize}
 Then, fixing any compact set $K$, we obtain the following chain of inequalities:
 \begin{align*}
   & \Gamma(x;y) \stackrel{\eqref{stimadel_THEO0}}{\leq}
   c
   \int_{\mathbb{R}^{p}}
    d_{\widetilde{X}}^{2-Q}\Big(  (x,0)^{-1}*(y,\eta)  \Big)\,\d\eta \\
   & \,\,\,=c
   \int_{|\eta|\geq 1}
    d_{\widetilde{X}}^{2-Q}\Big(  (x,0)^{-1}*(y,\eta)  \Big)\,\d\eta
   + c
   \int_{|\eta| < 1}
    d_{\widetilde{X}}^{2-Q}\Big(  (x,0)^{-1}*(y,\eta)  \Big)\,\d\eta \\
   & \stackrel{\eqref{eq.mainestimstepI}}{\leq}
    c_1\,\frac{d_{X}(x,y)^{2}}{\big|B_{X}(x,d_{X}(x,y))\big|}+
    c
   \int_{|\eta| < 1}
    d_{\widetilde{X}}^{2-Q}\Big(  (x,0)^{-1}*(y,\eta)  \Big)\,\d\eta \\
   & \stackrel{\eqref{eq.lemmaSanchezCalleeNSW}}{\leq}
   c_1\,\frac{d_{X}(x,y)^{2}}{\big|B_{X}(x,d_{X}(x,y))\big|}+
   c_2\,\int_{d_X(x,y)}^{R_0}\frac{\rho}{\big|B_{X}(x,\rho)\big|}\,\d\rho.
\end{align*}
 From this, we obtain the upper estimate in
 \eqref{goaln2r0} by applying Lemma \ref{Lemma nonintegrale} in the case $\beta = n = 2$, and by possibly enlarging
 $R_0$ in such a way that $R_0>3\sup\{d_X(x,y):x,y\in K\}$.

 We then prove \eqref{goaln2r0low}. To this end, let $R = R(K) > 0$ be as in
 Proposition \ref{prop1.stimabasso}. If $x,y\in K$ and $d_X(x,y)\geq 1$ (whence $R\geq 2$), then
 \eqref{goaln2r0low} is trivially satisfied with the choice
 $$C_1(K) := \frac{1}{\log(R)}\inf\big\{\Gamma(x;y):\,\text{$x,y\in K$ and $d_X(x,y)\geq 1$}\big\} > 0.$$
 If, instead, $x,y\in K$ and $0< d_X(x,y) < 1$,
 an inspection of the proof
 of Proposition \ref{prop1.stimabasso} shows that \eqref{eq.11propostim1}
 is valid also in the case $n = 2$, thus giving
 $$ \Gamma(x;y) \geq C\int_{d_X(x,y)}^{R}\frac{\rho}{|B_X(x,\rho)|}\,\d \rho
  \geq C\int_{d_X(x,y)}^{\min\{R,1\}}\frac{\rho}{|B_X(x,\rho)|}\,\d \rho = (\star).$$
 On the other hand, if $0<\rho\leq 1$, by Theorem B we have
 $$|B(x,\rho)|\leq \gamma_2\,\sum_{h = 2}^qf_h(x)\rho^h\leq
  \gamma'_2\,\rho^2, \qquad \text{where $\gamma'_2 :=
  \gamma_2\Big(f_q+\max_{x\in K}\sum_{h = 2}^{q-1}f_h(x)\Big)$}.$$
 Notice that $\gamma_2'\geq f_q > 0$. Thus, we obtain
 \begin{align*}
   (\star) \geq \frac{C}{\gamma_2'}\int_{d_X(x,y)}^{\min\{R,1\}}\frac{1}{\rho}\,\d\rho =
   \frac{C}{\gamma_2'}\log\Big(\frac{\min\{R,1\}}{d_X(x,y)}\Big).
\end{align*}
 Summing up, \eqref{goaln2r0low} is satisfied by possibly replacing $C_1$ with
 $\min\{C_1, C/\gamma_2'\}$, and with the choice $R_1 := \min\{R(K),1\}$.
 Finally, we show \eqref{eq.Gammapolexn2}. We apply \eqref{goaln2r0low} with the choice $K = \{y:\|y-x\|\leq 1\}$:
 \begin{align*}
  \liminf_{y\to x}
 \Gamma(x;y) & \geq
 C_1(K)\liminf_{y\to x} \log\Big(\frac{R_1(K)}{d_X(x,y)}\Big),
 \end{align*}
 for a suitable $R_1(K) > 0$. The latter $\liminf$ is clearly $\infty$, and the proof is complete.
 \end{proof}
\begin{remark}\label{re.featureuniformvsondiagonal}
 The main feature of Theorem \ref{thm.estimateGamman2} is to provide uniform estimates,
 valid as \emph{both $x$ and $y$ may vary} (in some compact set). In a different spirit, we next consider the case
 when $x$ is \emph{fixed} and $y$ is near the pole $x$.
 We shall see
 in Theorem \ref{thm.estimateGamman2bis}  that a different (and more precise) behavior arises.  Roughly put,
 the more ``rigid'' situation of a fixed pole $x$ will allow us to obtain upper and lower estimates
 with the same type of bounding functions (compare \eqref{goaln2r0bis} to \eqref{goaln2r0}-\eqref{goaln2r0low}).
 However, as we shall show in Example \ref{exa.grushin}, this is not in contrast
 with the different nature of the uniform  estimates of $\Gamma$, where the variability
 of the lower and upper bounds \eqref{goaln2r0}-\eqref{goaln2r0low} depicts the more general situation of both variable $x$'s and $y$'s.
\end{remark}
 \begin{theorem}\label{thm.estimateGamman2bis}
 Let $n=2$ and let the assumptions of Theorem \ref{thm.derivn2} apply. Let $f_2$
 be the nonnegative function introduced in Theorem B. Then, for any $x\in\R^n$,
 there exist positive constants $\gamma_1(x),\gamma_2(x)$ and a small
 $\varepsilon(x) < 1$ \emph{(}all depending on $x$ only\emph{)} such that
 \begin{equation} \label{goaln2r0bis}
   \gamma_1(x)\,F(x,y)\leq \Gamma(x;y)\leq \gamma_2(x)\,F(x,y),
\end{equation}
 for any $y$ such that $0<d_X(x,y)<\varepsilon(x)$,
 where
 \begin{equation}\label{goaln2r0bissssss}
     F(x,y)=      \begin{cases}
            \log\left(\dfrac{1}{d_X(x,y)}\right) & \text{if $f_2(x)>0$,} \\[0.4cm]
            \dfrac{d_X(x,y)^{2}}{\big|B_{X}(x,d_X(x,y))\big|} & \text{if $f_2(x)=0$.}
          \end{cases}
 \end{equation}
 In the case $f_2(x)=0$, the estimate \eqref{goaln2r0bis} holds true
 with $\varepsilon(x)=1/2$ and $\gamma_1(x)$ independent of $x$.
 In the case $f_2(x)=0$, $F(x,y)$ diverges like $d_X(x,y)^{2-k}$, for
 some $k\in \{3,\ldots,q\}$.
\end{theorem}
\begin{proof}
 Let $x\in\R^n$ be arbitrarily fixed. We take any $y\in\R^n$ conveniently close to $x$, namely
\begin{equation}\label{sceltayconveniently}
    0<d_{X}(x,y)\leq 1/2.
\end{equation}
 Next, we take a compact set $K$ which is a neighborhood of $0\in\R^n$,
 and such that $K$ contains the closure of $B_X(x,1/2)$.
 Finally, we choose $R_0(x)\gg 1$ such that
 \begin{equation*}
  d_{\widetilde{X}}
  \left(\left(x,0\right)^{-1}*\left(y,\eta\right)\right) <
  \frac{R_0(x)}{2},
 \end{equation*}
 uniformly for any $y$ as in \eqref{sceltayconveniently} and any $|\eta|\leq 1$.
 This implicitly implies that $d_X(x,y)\leq R_0(x)/2$. With all these choices,
 in due course of the proof of
 Theorem \ref{thm.estimateGamman2}, we proved the existence of two constants $c_1(x),c_2(x)>0$
 and of a structural constant $c>0$ such that
\begin{equation}\label{stimosainitinere0}
 c\,\int_{d_{X}(x,y)}^{1}
 \frac{\rho}{|B_{X}(x,\rho)|}\,\d\rho
 \leq\Gamma(x;y)\leq
 c_{1}(x)\,\frac{d_{X}(x,y)^{2}}{\big|B_{X}(x,d_{X}(x,y))\big|}+
 c_{2}(x)\,\int_{d_{X}(x,y)}^{R_0(x)}\frac{\rho}{\big|B_{X}(x,\rho)\big|}\d\rho.
\end{equation}
 We observe that $c_1(x),c_2(x)$ depend on $x$ only through
 the compact set $K$ containing $x$.

 On account of \eqref{eq.NSWmodificata}, we remind that (as $n=2$) we have nonnegative functions $f_k$ such that
\begin{equation}\label{NSWnelcasonug2}
 \gamma_1\,\sum_{k=2}^{q}f_{k}(x)\,\rho^{k} \leq {\big|B_{X}(x,\rho)\big|} \leq \gamma_2\,\sum_{k=2}^{q}f_{k}(x)\,\rho^{k},\qquad \forall \,x\in\R^n,\,\,\rho>0.
\end{equation}
 Then we distinguish two cases:\medskip

 (I).\,\, \emph{Suppose that $f_{2}(x)>0$.} Then \eqref{NSWnelcasonug2} gives
$$
\gamma_1\,f_{2}(x)\,\rho^{2}
 \leq \big|B_{X}(x,\rho)\big|\leq
 \gamma_2\,\rho^{2} \left(f_{2}(x)+\sum_{k=3}^{q} f_{k}(x)\,\rho^{k-2}\right).$$
 Hence the choices $\gamma_1(x):=\gamma_1\,f_{2}(x)$ and $\gamma_2(x):=\gamma_2\sum_{k=2}^{q} f_{k}(x)\,R_0(x)^{k-2}$
 are two positive constants (only depending on $x$) such that
\begin{equation}\label{stimosainitinere12}
 \gamma_1(x)\,\rho^{2}
 \leq
 \big|B_{X}(x,\rho)\big|\leq
 \gamma_2(x)\,\rho^{2},\qquad \text{for any $0<\rho\leq R_0(x)$.}
\end{equation}
 Therefore, the latter inequalities give (taking $\rho=d_X(x,y)\leq 1/2<1\ll R_0(x)$)
\begin{equation}\label{stimosainitinere1}
 \frac{1}{\gamma_2(x)}\leq \frac{d_{X}(x,y)^{2}}{\big|B_{X}(x,d_{X}(x,y))\big|}\leq \frac{1}{\gamma_1(x)},\qquad
 \forall\,\, y\in \overline{B_X(x,1/2)}.
\end{equation}
 Again by \eqref{stimosainitinere12} one also has
\begin{equation*}
 \frac{1}{\gamma_2(x)\,\rho}
 \leq
 \frac{\rho}{\big|B_{X}(x,\rho)\big|}
 \leq
 \frac{1}{\gamma_1(x)\,\rho},\qquad \text{for any $0<\rho\leq R_0(x)$.}
\end{equation*}
 Thus, upon an integration of the latter, we can give a lower bound for the
 far left-hand side of  \eqref{stimosainitinere0} and an upper bound
 for the second summand in the far right-hand side of  \eqref{stimosainitinere0}:
 gathering these bounds together with \eqref{stimosainitinere1}, we deduce from
 \eqref{stimosainitinere0} the following estimates:
\begin{equation*}
 \frac{c}{\gamma_2(x)}\log\left(\frac{1}{d_X(x,y)}\right)
  \leq\Gamma(x;y)\leq
 \frac{c_{1}(x)}{\gamma_1(x)}+
 \frac{c_{2}(x)}{\gamma_1(x)}\,
 \log\left(\frac{R_0(x)}{d_X(x,y)}\right),
\end{equation*}
 holding true for any $y$ as in \eqref{sceltayconveniently}.
 This is sufficient\footnote{First
 we choose $\varepsilon(x)>0$ such that
 $\frac{c_{1}(x)}{\gamma_1(x)}\leq \frac{c_{2}(x)}{\gamma_1(x)}\,
 \log\left(\frac{R_0(x)}{d_X(x,y)}\right)$ for $0<d_X(x,y)\leq \varepsilon(x)$; for instance
 the choice $\varepsilon(x):=R_0(x)\,\exp(-c_1(x)/c_2(x))$ does the job. Then one
 can use the inequality $\log\left(\frac{R_0(x)}{d_X(x,y)}\right)\leq \alpha(x)\log\left(\frac{1}{d_X(x,y)}\right)$
 valid for $d_X(x,y)\leq \varepsilon(x)<1$ (and any $R_0(x)\geq 1$), if one chooses $\alpha(x)=1-\log(R_0(x))/\log(\varepsilon(x))$.}
 to infer the existence of positive constants $\gamma_i(x)$ ($i=1,2$) and
 $\varepsilon(x)\ll 1$ such that \eqref{goaln2r0bis} holds, with $F(x,y)$ as in the first case
 of \eqref{goaln2r0bissssss}.\medskip

 (II).\,\, \emph{Suppose that $f_{2}(x)=0$.}
  In this case, \eqref{NSWnelcasonug2} gives
\begin{equation}\label{NSWnelcasonug2rifa}
 \gamma_1\,\sum_{k=3}^{q}f_{k}(x)\, r ^{k} \leq {\big|B_{X}(x, r )\big|} \leq \gamma_2\,\sum_{k=3}^{q}f_{k}(x)\, r ^{k},\qquad \forall \,x\in\R^n,\,\, r >0.
\end{equation}
 We fix $x\in\R^n$ and $ r >0$, and we set $\Lambda(x, r ):=\sum_{k=3}^{q}f_{k}(x)\, r ^{k}$;
 from \eqref{NSWnelcasonug2rifa} (rewritten with $\lambda\, r $ replacing $ r $) we derive that, for any $\lambda\geq 1$, one has
\begin{align*}
  {\big|B_{X}(x,\lambda r )\big|} \leq \gamma_2\, \Lambda(x,\lambda r )= \gamma_2\,\sum_{k=3}^{q}f_{k}(x)\, r ^{k}\lambda^k
 \leq\lambda^q \gamma_2\,\Lambda(x, r )\stackrel{\eqref{NSWnelcasonug2rifa}}{\leq}\lambda^q\frac{\gamma_2}{\gamma_1}\, \big|B_{X}(x, r )\big|.
\end{align*}
 Analogously, again for any $\lambda\geq 1$,
\begin{align*}
  {\big|B_{X}(x,\lambda r )\big|} \geq \gamma_1\, \Lambda(x,\lambda r )= \gamma_1\,\sum_{k=3}^{q}f_{k}(x)\, r ^{k}\lambda^k
 \geq\lambda^3 \gamma_1\,\Lambda(x, r )\stackrel{\eqref{NSWnelcasonug2rifa}}{\geq}\lambda^3\frac{\gamma_1}{\gamma_2}\, \big|B_{X}(x, r )\big|.
\end{align*}
  This gives at once, for any $\lambda\geq 1$,
\begin{equation}\label{NSWnelcasonug2rifacapito}
 \lambda^3\frac{\gamma_1}{\gamma_2}\, \big|B_{X}(x, r )\big|\leq {\big|B_{X}(x,\lambda r )\big|}\leq \lambda^q\frac{\gamma_2}{\gamma_1}\, \big|B_{X}(x, r )\big|
\end{equation}
 If $\rho\geq d_X(x,y)>0$, we are entitled to choose $\lambda=\rho/d_X(x,y)\geq 1$ and $r=d_X(x,y)$ in
 \eqref{NSWnelcasonug2rifacapito}, getting
\begin{equation}\label{NSWnelcasonug2rifacapito333333333333333333333333}
 \frac{\gamma_1}{\gamma_2}\,\frac{\rho^3}{d_X(x,y)^3}\,\big|B_{X}(x, d_X(x,y) )\big|\leq {\big|B_{X}(x,\rho)\big|}\leq
 \frac{\gamma_2}{\gamma_1}\,\frac{\rho^q}{d_X(x,y)^q}\,\big|B_{X}(x, d_X(x,y) )\big|.
\end{equation}
 The first inequality in \eqref{NSWnelcasonug2rifacapito333333333333333333333333}
 allows us to estimate the second summand in the far right-hand side of  \eqref{stimosainitinere0}
 in the following way:
\begin{align*}
 &\int_{d_{X}(x,y)}^{R_0(x)}\frac{\rho}{\big|B_{X}(x,\rho)\big|}\,\d\rho
 \leq
 \frac{\gamma_2}{\gamma_1}\,\frac{d_X(x,y)^3}{\big|B_{X}(x, d_X(x,y) )\big|}
 \int_{d_{X}(x,y)}^{R_0(x)}\frac{\rho}{\rho^3}
 \,\d\rho\\
 &\qquad =\frac{\gamma_2}{\gamma_1}\,\frac{d_X(x,y)^3}{\big|B_{X}(x, d_X(x,y) )\big|}
 \frac{R_0(x)-d_{X}(x,y)}{d_{X}(x,y)R_0(x)}\leq
 \frac{\gamma_2}{\gamma_1}\,\frac{d_X(x,y)^2}{\big|B_{X}(x, d_X(x,y) )\big|}.
\end{align*}
 The second inequality in \eqref{NSWnelcasonug2rifacapito333333333333333333333333}
 gives the following estimate for the first summand in the far left-hand side of  \eqref{stimosainitinere0}
 (we are also using \eqref{sceltayconveniently}, so that $1\geq 2\,d_X(x,y)$):
\begin{align*}
 &\int_{d_{X}(x,y)}^{1}\frac{\rho}{\big|B_{X}(x,\rho)\big|}\,\d\rho
 \geq
 \frac{\gamma_1}{\gamma_2}\,\frac{d_X(x,y)^q}{\big|B_{X}(x, d_X(x,y) )\big|}
 \int_{d_{X}(x,y)}^{1}\frac{\rho}{\rho^q}
 \,\d\rho\\
 &\qquad \geq
 \frac{\gamma_1}{\gamma_2}\,\frac{d_X(x,y)^q}{\big|B_{X}(x, d_X(x,y) )\big|}
 \int_{d_{X}(x,y)}^{2\,d_X(x,y)}\frac{1}{\rho^{q-1}}
 \,\d\rho=
 \frac{\gamma_1\,(1-2^{2-q})}{\gamma_2\,(q-2)}\,\frac{d_X(x,y)^2}{\big|B_{X}(x, d_X(x,y) )\big|}.
\end{align*}
 Summing up, \eqref{stimosainitinere0} gives
 \begin{equation*}
 \frac{c\,\gamma_1\,(1-2^{2-q})}{\gamma_2\,(q-2)}\,\frac{d_{X}(x,y)^{2}}{\big|B_{X}(x,d_{X}(x,y))\big|}
 \leq\Gamma(x;y)\leq
 \frac{d_{X}(x,y)^{2}}{\big|B_{X}(x,d_{X}(x,y))\big|}
 \left(c_{1}(x)+ \frac{c_{2}(x)\gamma_2}{\gamma_1}\right).
\end{equation*}
 This proves \eqref{goaln2r0bis}, with $F(x,y)$ as in the second case
 of \eqref{goaln2r0bissssss}. Incidentally, this also proves that \eqref{goaln2r0bis} holds true
 with $\varepsilon(x)=1/2$ and $\gamma_1(x)$ independent of $x$. As for the last
 assertion of the theorem, this immediately  follows from \eqref{goaln2r0bissssss} and \eqref{NSWnelcasonug2rifa}.
\end{proof}
\begin{example}\label{exa.grushink}
 Let us consider the vector fields on $\mathbb{R}^{2}$
$$
 X_{1}=\partial_{x_{1}},\quad X_{2}=x_{1}^{k}\,\partial_{x_{2}},
$$
 for a fixed integer $k\geq 1$, and the corresponding PDO on $\mathbb{R}^{2}$
$$
\mathcal{L}=\partial_{x_{1}}^{2}+x_{1}^{2k}\,\partial
_{x_{2}}^{2}.
$$
 It is readily seen that $X_{1}$ and $X_{2}$ are homogeneous of degree $1$ with
 respect to the dilations
$$
 \delta_{\lambda}(x_{1},x_{2})=(\lambda x_{1},\lambda^{k+1}x_{2}).
$$
 Obviously, assumptions (H1)-to-(H3) of Section \ref{sec:introductionMain}
 are satisfied (here $q=N=k+2$ are both $>2$). In particular (see Remark
 \ref{Remark explicit f_k} for the following construction), the possible bases
 of $\mathbb{R}^{2}$ built by commutators of $X_{1},X_{2}$, are:
\begin{alignat*}{3}
 &X_{1}=\partial_{x_{1}},\,\,X_{2}=x_{1}^{k}\,\partial_{x_{2}}         & &\qquad \text{of weight $2$}  & & \qquad\text{with $f_{2}(x)= |x_{1}|^{k}$}\\
 &X_{1},\,\, [ X_{1},X_{2} ]=k\,x_{1}^{k-1}\,\partial_{x_{2}}           & &\qquad \text{of weight $3$}  & & \qquad\text{with $f_{3}(x)= k\,| x_{1}|^{k-1}$}\\
 &X_{1},\,\,[X_{1},[X_{1},X_{2}]]=k(k-1)\,x_{1}^{k-2}\,\partial_{x_{2}} & &\qquad \text{of weight $4$}  & & \qquad\text{with $f_{4}(x)=k(k-1)\,|x_{1}|^{k-2}$}\\
 &\vdots                                                                 &&\qquad  \vdots               &  & \qquad\vdots\\
 &X_{1},\,\, [X_{1},[X_{1},\ldots [X_{1},X_{2}]]]=k!\,\partial_{x_{2}}   & &\qquad \text{of weight $k+2$}& & \qquad\text{with $f_{k+2}=k!$}
\end{alignat*}
Therefore, for every $x$ and $\rho$, one has
\begin{equation}\label{stimaballanug2}
 |B_X(x,\rho)|\approx \Lambda (  x,\rho )  = |x_{1}|^{k} \rho ^{2}+k | x_{1}|^{k-1}\rho^{3}+k(k-1)|x_{1}|^{k-2}\rho^{4}+\cdots +k!\,\rho^{k+2}. 
\end{equation}
 Thus, according to our Theorem  \ref{thm.estimateGamman2bis}, we have, when $y$ is sufficiently close to $x$,
 $$ \Gamma(x;y) \approx_x
 \begin{cases}
            \log\left(\dfrac{1}{d_X(x,y)}\right) & \text{if $x_1\neq 0$,} \\[0.4cm]
            \dfrac{1}{d_X(x,y)^k} & \text{if $x_1=0$.}
          \end{cases}
$$
  where $\approx_x$ means upper/lower bounds with constant possibly depending on the fixed $x$;
 this exhibits a different asymptotic behavior for $y\rightarrow x$, at different points $x$.
\end{example}
  Example \ref{exa.grushink} exhibits the different behavior
  (logarithmical vs.\,\,power-like) of $\Gamma(x;y)$ as $y$ ap\-proa\-ches different poles $x$.
  Thus, the lower estimate \eqref{goaln2r0low} cannot be improved to become analogous to
  the upper estimate \eqref{goaln2r0}.
  A question remains open whether the upper estimate \eqref{goaln2r0} is optimal or not:
  in this regard, Theorem \ref{thm.estimateGamman2bis} would induce one to think
  that the upper bound be either logarithmic or power-like, and it does not seem to
  forecast the presence of a product of these bounds. However, the next Example  \ref{exa.grushin}
  will show that, if $x$ and $y$ can vary the same time (which is possible, in the
 case of uniform estimates, as we already pointed out in Remark \ref{re.featureuniformvsondiagonal}),
 then one can  exactly obtain the product of a logarithm and a power of $d_X(x,y)$.

 Indeed, when one takes $k = 1$ in Example \ref{exa.grushink},
 the global fundamental solution of
 $$\mathcal{G} = (\de_{x_1})^2+(x_1\de_{x_2})^2$$ can
 be explicitly computed, so that we can verify the optimality of our estimates
 in Theorems \ref{thm.estimateGamman2} and \ref{thm.estimateGamman2bis}:
 in this sense, Example \ref{exa.grushink} differs from Example \ref{exa.grushin}
 in that in the latter we start from what is known about $\Gamma$ and then we check our results, rather than
 confining ourselves in showing what our results state.
\begin{example}[The Grushin case for $k = 1$]\label{exa.grushin}
 Take $k = 1$ in Example \ref{exa.grushink}, and consider the associated vector fields $X_1=\de_{x_1}$ and $X_2=x_1\de_{x_2}$.
 Due to the results by Franchi, Lanconelli performed for the vector fields
 $\de_{x_1}$ and $|x_1|\de_{x_2}$ in \cite{FranchiLanconelli} (see also
 Kogoj, Lanconelli \cite{KogojLanconelli}),\footnote{It is not difficult to recognize that the 
 CC-distance induced by $\de_{x_1}$ and $|x_1|\de_{x_2}$ coincides with $d_X$.}
  one can get the explicit estimate for the Carnot-Carathéodory distance $d_X$:\footnote{Here and in the sequel,
 $\approx$ means upper/lower estimates, up to some universal constants; the variant $\approx_x$
 means that upper/lower estimates are true, modulo constants possibly depending on $x$. With the notation `$f(x)\sim g(x)$ as $x\to x_0$' we mean, as usual,
 $f(x)=g(x)\,\omega(x)$ with $\omega(x)\to 1$ as $x\to x_0$.}
\begin{equation}\label{kLexplici}
    d_X(x,y)\approx
 |x_1-y_1|+\sqrt{|x_1|^2+|x_2-y_2|}-|x_1|,\quad \text{for every $x,y\in\R^2$.}
\end{equation}
 Moreover, Theorem B gives (see also \eqref{stimaballanug2} with $k=1$)
\begin{equation}\label{kLexplici2222222}
   \big|B_X(x,d_X(x,y))\big|\approx |x_1|\,d_X(x,y)^2+d_X(x,y)^3,\quad \text{for every $x,y\in\R^2$.}
\end{equation}
 With the notation in \eqref{eq.NSWmodificata}, we have $f_2(x)=|x_1|$ and $f_3(x)\equiv 1$.
 Furthermore, the lifting Carnot group as in Theorem A is $\G = (\R^3\equiv \R^2_x\times \R_\xi,*)$, where
  \begin{equation*}
   (x_1,x_2,\xi)*(x'_1,x'_2,\xi') = (x_1+x'_1, x_2+x'_2+x_1\xi', \xi+\xi').
  \end{equation*}
  Thus, the vector fields $\widetilde{X}_1,\widetilde{X}_2$ lifting $X_1$ and $X_2$ are
  \begin{equation*}
   \widetilde{X}_1 = \de_{x_1}, \qquad \widetilde{X}_2 = x_1\,\de_{x_2}+\de_\xi.
  \end{equation*}
  The operator $\LL= X_1^2+X_2^2$ is lifted to
  the sublaplacian $\LL_\G = \widetilde{X}_1^2+\widetilde{X}_2^2$.
  The latter is (up to a change of variable)
  the Kohn Laplacian on the first Heisenberg group $\mathbb{H}^1$; after simple computations (manipulating the well-known fundamental solution
  for the Kohn Laplacian on $\mathbb{H}^1$), one finds the fundamental solution
  with pole at the origin of $\LL_\G$:
  $$\Gamma_\G(x,\xi) = \gamma_0\, \Big((x_1^2+\xi^2)^2+16\,(x_2-\tfrac{1}{2}\,x_1\xi)^2 \Big)^{-1/2},
  \quad (x,\xi)\neq (0,0),$$
  where $\gamma_0 > 0$ is a suitable constant.
  According to Theorem A, the function
\begin{equation} \label{sec.th:GammaGrushin}
  \Gamma(x_1,x_2;y_1,y_2) = \gamma_0\,\int_{\R}\frac{\d\eta}
  {\sqrt{ ((x_1-y_1)^2+\eta^2 )^2+ 4\,
    (2\,x_2 - 2\,y_2 + \eta\,(x_1+y_1) )^2}},
 \end{equation}
   is the  unique fundamental solution for the Grushin operator $\LL$ vanishing at infinity.
   The integral in \eqref{sec.th:GammaGrushin} can be expressed in terms of
   elliptic functions in a standard way: more precisely,
   \begin{equation} \label{sec.th:GammaGrushinExplicit}
    \Gamma(x;y) =\frac{\gamma_0\,\sqrt{2}}{\sqrt[4]{(x_1^2+y_1^2)^2+4\,(x_2-y_2)^2}}\cdot
    \mathrm{K}\left(\frac{1}{2}+
    \frac{x_1y_1}{\sqrt{(x_1^2+y_1^2)^2+4\,(x_2-y_2)^2}}\right),
   \end{equation}
   where $\mathrm{K}$ denotes the complete elliptic integral of the first
   kind, that is,
   $$\mathrm{K}(m) := \int_0^{\pi/2} \frac{\d\theta}{\sqrt{1-m\,\sin^2(\theta)}}\,,
   \quad \text{for $-1 < m < 1$}.$$
   This gives back a formula obtained by Greiner \cite{Greiner}
   (see also \cite{BealsGaveauGreiner1, BealsGaveauGreiner3, BealsGaveauGreinerKannai2, BauerFurutaniIwasaki}).
  In the sequel, we write \eqref{sec.th:GammaGrushinExplicit} as $\Gamma(x;y)=H(x,y)\cdot\mathrm{K}(x,y)$, with the obvious meaning.
  When $x_1=0$ \eqref{sec.th:GammaGrushinExplicit} gives
\begin{equation*} 
    \Gamma(0,x_2;y_1,y_2) = \frac{\gamma_0\sqrt2\,\mathrm{K}(1/2)}{\sqrt[4]{y_1^4+4\,(x_2-y_2)^2}}.
\end{equation*}
 Setting $\Omega:=\{x\in\R^2:x_1\neq 0\}$, if $x$ is fixed in $\Omega$, then
 only the factor $\mathrm{K}(x,y)$ diverges  as $y\to x$ (while $H(x,y)$ remains bounded); conversely,
 if $x$ is fixed outside $\Omega$ (i.e., $x_1=0$), only the factor ${H}(x,y)$ diverges
 as $y\to x$ (while $\mathrm{K}(x,y)=\mathrm{K}(1/2)$).

  In Remark \ref{re.asymptoKK}, we will show that
\begin{equation}\label{asimptoKKKKKKKKKK}
 \mathrm{K}(m)\sim -\frac{1}{2}\ln(1-m)\quad \text{as $m\to 1^-$.}
\end{equation}
 Thus, fixing $x\in \R^2$, \eqref{asimptoKKKKKKKKKK} gives a very precise asymptotic behavior of $\Gamma(y;x)$
 as $y\to x$:
\begin{equation}\label{asimptoKKKKKKKKKKasy}
 \Gamma(x;y)\sim
   \begin{cases}
     -\dfrac{\gamma_0}{2\,|x_1|}\log\left(\dfrac{1}{2}-\dfrac{x_1y_1}{\sqrt{(x_1^2+y_1^2)^2+4\,(x_2-y_2)^2}}\right), &\qquad \text{if $x_1\neq 0$}\\[0.5cm]
     \dfrac{\gamma_0\sqrt2\,\mathrm{K}(1/2)}{\sqrt[4]{y_1^4+4\,(x_2-y_2)^2}},&\qquad \text{if $x_1=0$.}
   \end{cases}
 \end{equation}

 We now compare the above formulas with our estimates in Theorems \ref{thm.estimateGamman2} and \ref{thm.estimateGamman2bis}.\medskip

 $\bullet$ \emph{Theorem \ref{thm.estimateGamman2bis}:} The estimates in \eqref{goaln2r0bis} have full feedback from
 \eqref{asimptoKKKKKKKKKKasy}. Indeed, due to \eqref{kLexplici2222222}, the function $F(x,y)$ in \eqref{goaln2r0bissssss} is in the present case
\begin{equation*}\label{goaln2r0bisssssskkkk}
     F(x,y)=      \begin{cases}
            \log\left(\dfrac{1}{d_X(x,y)}\right) & \text{if $x_1\neq 0$,} \\[0.4cm]
            \dfrac{1}{d_X(x,y)} & \text{if $x_1=0$.}
          \end{cases}
 \end{equation*}
 Thus \eqref{goaln2r0bis} is in accordance with \eqref{asimptoKKKKKKKKKKasy} for the following reasons (here we used \eqref{kLexplici}
 and some Taylor approximation\footnote{One can show that both following functions have the same Taylor expansion as $y\to x$:
 $$|x_1-y_1|+\sqrt{|x_1|^2+|x_2-y_2|}-|x_1|\approx \left(\frac{1}{2}-\frac{x_1y_1}{\sqrt{(x_1^2+y_1^2)^2+4\,(x_2-y_2)^2}}\right)^{1/2}.$$}):
\begin{align*}
 &\textrm{- if $x_1\neq 0$:}\quad \log\left(\dfrac{1}{d_X(x,y)}\right){\approx}
 \,\,\,\log\left(\dfrac{1}{|x_1-y_1|+\sqrt{|x_1|^2+|x_2-y_2|}-|x_1|}\right)\\
 &\quad\qquad\qquad \qquad\qquad\qquad\quad\,\,\,\,\approx_x
  -\frac{1}{2}\log\left(\dfrac{1}{2}-\dfrac{x_1y_1}{\sqrt{(x_1^2+y_1^2)^2+4\,(x_2-y_2)^2}}\right);\\
 &\textrm{- if $x_1=0$:}\quad d_X(x,y)\approx |y_1|+\sqrt{|x_2-y_2|}\approx \sqrt[4]{y_1^4+4\,(x_2-y_2)^2}.
\end{align*}

$\bullet$ \emph{Theorem \ref{thm.estimateGamman2}:} Firstly we consider the estimate in \eqref{goaln2r0}.
 At first glance, Theorem \ref{thm.estimateGamman2bis} does not seem to match with the product appearing in
 \eqref{goaln2r0}. However, the peculiarity of the latter is to provide an estimate which is uniform as
 \emph{both $x$ and $y$ vary} (in some compact set). We show that a suitable choice of variable $x$'s and $y$'s
 confirm the product behavior of $\Gamma$. Indeed, let us take
 $$x(\epsilon):=(\epsilon,2\,\epsilon^4)\quad\text{and}\quad y(\epsilon):=(\epsilon, \epsilon^4).$$
  As $\epsilon\to 0^+$, \eqref{sec.th:GammaGrushinExplicit} gives (on account of \eqref{asimptoKKKKKKKKKK})
$$\Gamma\big(x(\epsilon);y(\epsilon)\big)=\frac{\gamma_0}{\sqrt[4]{\epsilon^4+\epsilon^8}}\cdot
\mathrm{K}\bigg(\frac{1}{2}+\frac{1}{2\sqrt{1+\epsilon^4}}\bigg)\sim
  \frac{2\,c}{\epsilon}\,\log \bigg(\frac{1}{\epsilon}\bigg)=:h(\epsilon).$$
 On the other hand, due to \eqref{kLexplici}, we have
\begin{equation}\label{asimptoKKKKKKKKKKstimoepsi}
 d_X(x(\epsilon),y(\epsilon))\approx \sqrt{\epsilon^2+\epsilon^4}-\epsilon \sim \frac{\epsilon^3}{2},\quad \text{as $\epsilon\to 0^+$}.
\end{equation}
 Thus, the estimate \eqref{asimptoKKKKKKKKKKstimoepsi} together with \eqref{kLexplici2222222} give
\begin{equation*}
   \frac{d_{X}(x(\epsilon),y(\epsilon))^{2}}{\big|B_{X}(x(\epsilon),d_{X}(x(\epsilon),y(\epsilon)))\big|}
  \cdot \log\Big(\frac{1}{d_X(x(\epsilon),y(\epsilon))}\Big)
 \approx \frac{1}{\epsilon+\epsilon^3}\cdot \log \Big(\frac{1}{\epsilon^3}\Big) \approx \frac{\log(1/\epsilon)}{\epsilon}.
\end{equation*}
 Since the latter term is $\approx$ $h(\epsilon)$,
 these computations fully confirm \eqref{goaln2r0} in Theorem \ref{thm.estimateGamman2}.\medskip

 Secondly we consider the estimate in \eqref{goaln2r0low}: we observe that the function bounding $\Gamma$ from below in \eqref{goaln2r0low} cannot be
 of the same product form as in \eqref{goaln2r0}. Indeed, suitable choices of $x$'s and $y$'s show
 either logarithmic or power-like behaviors; for example, one can easily recognize what follows
 (starting from the explicit formula \eqref{sec.th:GammaGrushinExplicit}):
\begin{align*}
    &\Gamma\big((1,2\,\epsilon);(1,\epsilon)\big)=
   \frac{\gamma_0\sqrt2}{\sqrt[4]{4+4\,\epsilon^2}}\cdot
   \mathrm{K}\bigg(\frac{1}{2}+\frac{1}{\sqrt{4+4\,\epsilon^2}}\bigg) \sim \gamma_0\ln\Big(\frac{1}{\epsilon}\Big),\quad \text{as $\epsilon\to 0^+$;}\\
    &\Gamma\big((0,\epsilon);(\epsilon,\epsilon)\big)=
  \frac{\gamma_0\,\sqrt2\,\mathrm{K(1/2)}}{\epsilon}\quad \text{for every $\epsilon>0$}.
\end{align*}
%
\end{example}
 For completeness, we provide a full argument proving \eqref{asimptoKKKKKKKKKK} in the next remark.
\begin{remark}\label{re.asymptoKK}
 After the change of variable $t=\sin\theta$, one has
\begin{equation*}
   \mathrm{K}(m) = \int_0^{1} \frac{1}{\sqrt{1-m\,t^2}}\,\frac{\d t}{\sqrt{1-t^2}}.
\end{equation*}
 We rewrite the integrand function as
 $$\frac{f(t,m)}{2\,\sqrt{1-t}\sqrt{1-\sqrt{m}\,t}},\quad \text{where}\quad
  f(t,m)=\frac{2}{\sqrt{1+t}\sqrt{1+\sqrt{m}\,t}}.$$
 As $f(1,1)=1$, for any fixed $\epsilon>0$ we choose $0<m(\epsilon),\delta(\epsilon)\ll 1$ such that
 $1-\epsilon\leq f(t,m)\leq 1-\epsilon$, whenever $1-m(\epsilon)\leq m\leq 1$ and $1-\delta(\epsilon)\leq t\leq 1$.
 We have
\begin{align*}
 \mathrm{K}(m) = \int_0^{1-\delta(\epsilon)} \frac{1}{\sqrt{1-m\,t^2}}\,\frac{\d t}{\sqrt{1-t^2}}+
  \int_{1-\delta(\epsilon)}^{1} \frac{f(t,m)}{2\,\sqrt{1-t}\sqrt{1-\sqrt{m}\,t}}\,\d t=:A(m,\epsilon)+B(m,\epsilon).
\end{align*}
 The first summand satisfies
\begin{equation*} 
 0\leq A(m,\epsilon)\leq \int_0^{1-\delta(\epsilon)} \frac{\d t}{1-t^2}=:A'(\epsilon),\quad\forall\,\,m\in [0,1].
\end{equation*}
 Due to the estimate on $f(t,m)$, the second summand satisfies
\begin{equation*}
 (1-\epsilon)\int_{1-\delta(\epsilon)}^{1} \frac{\d t}{2\,\sqrt{1-t}\sqrt{1-\sqrt{m}\,t}}
 \leq B(m,\epsilon)\leq
 (1+\epsilon)\int_{1-\delta(\epsilon)}^{1} \frac{\d t}{2\,\sqrt{1-t}\sqrt{1-\sqrt{m}\,t}},
\end{equation*}
 whenever $1-m(\epsilon)\leq m\leq 1$.
  A direct computation gives
\begin{align*}
 \int_{1-\delta(\epsilon)}^{1} \frac{\d t}{2\,\sqrt{1-t}\sqrt{1-\sqrt{m}\,t}}&=
 \frac{1}{\sqrt[4]{m}}\left\{
 \ln\bigg(\sqrt{1-(1-\delta(\epsilon))\sqrt{m}}+\sqrt{\delta(\epsilon)}\sqrt[4]{m}\bigg)
 -\ln\sqrt{1-\sqrt{m}}
\right\}\\
 &=:C(m,\epsilon)-\frac{\ln\sqrt{1-\sqrt{m}}}{\sqrt[4]{m}}=:C(m,\epsilon)+q(m),
\end{align*}
  where $C(m,\epsilon)$ satisfies (uniformly for $m\in[\tfrac{1}{16},1]$)
 $$\ln\left(\frac{3}{2}\,\sqrt{\delta(\epsilon)}\right)
 =:
 C'(\epsilon)
\leq C(m,\epsilon) \leq
 C''(\epsilon):= 2\ln\left(\frac{1}{2}\sqrt{3+\delta(\epsilon)}+\sqrt{\delta(\epsilon)}\right).$$
 Gathering together all the estimates on $A,B,C$  we get
\begin{align*}
  \frac{\mathrm{K}(m)}{q(m)}&=\frac{A(m,\epsilon)}{q(m)}+\frac{B(m,\epsilon)}{q(m)},
\end{align*}
 where (since $q(m)\longto\infty$ as $m\to 1^-$)
\begin{align*}
  0\leq \frac{A(m,\epsilon)}{q(m)}\leq \frac{A'(\epsilon)}{q(m)}\longto 0\quad \text{as $m\to 1^-$},
\end{align*}
 and moreover (if $1-m(\epsilon)\leq m\leq 1$)
\begin{align*}
 (1-\epsilon)\bigg(\frac{C(m,\epsilon)}{q(m)}+1\bigg)\leq \frac{B(m,\epsilon)}{q(m)}\leq (1+\epsilon)\,\bigg(\frac{C(m,\epsilon)}{q(m)}+1\bigg),
\end{align*}
 where
 $$
 \frac{C'(\epsilon)}{q(m)}\leq \frac{C(m,\epsilon)}{q(m)}\leq \frac{C''(\epsilon)}{q(m)},\qquad \text{and}
\quad \frac{C'(\epsilon)}{q(m)},\frac{C''(\epsilon)}{q(m)}\longto 0\quad \text{as $m\to 1^-$}.
$$
 The above computations show that (for some $0<m''(\epsilon)\ll 1$)
 $$(1-\epsilon)^2\leq \frac{\mathrm{K}(m)}{q(m)}\leq \epsilon+(1+\epsilon)^2,\quad \text{whenever $ 1-m''(\epsilon)\leq m\leq 1$.} $$
 This proves  that $\frac{\mathrm{K}(m)}{q(m)} \longto 1$, as $m\to 1^-$, that is, $\mathrm{K}(m)\sim q(m)$ as $m\to 1^-$.
 On the other hand,
 $$q(m)=-\frac{\ln\sqrt{1-\sqrt{m}}}{\sqrt[4]{m}}\sim - \ln\sqrt{1-\sqrt{m}}\sim -\frac{1}{2}\ln(1-m)\quad \text{as $m\to 1^-$.} $$
 This ends the proof of \eqref{re.asymptoKK}.
\end{remark}
\section{Applications to potential theory} \label{sec:potentialtheory}
 We let $\LL:=X_1^2+\cdots+X_m^2$,
 where $X_1,\ldots,X_m$ are smooth vector fields satisfying our assumptions (H.1)-to-(H.3) in
 Section \ref{sec:notations.review}. As usual, $\Gamma$ is its global fundamental solution as in
 Theorem A. The assumption $q>2$ (see \eqref{eq.defq}) is still valid throughout the section.
 In what follows, we say that a function $u$ is $\LL$-harmonic on an open set $\Omega\subseteq\R^n$ if
 $u\in C^2(\Omega)$ and $\LL u = 0$ in $\Omega$.

 The aim of this section is to show that our operator $\LL$ enjoys
 all the axioms in
 \cite{AbbondanzaBonfiglioli, BattagliaBonfiglioli, BattBonf, BLJems}, thus allowing us to derive the following potential-theoretic results:
 several characterizations of the cone of the $\LL$-subharmonic functions in \cite{BLJems};
 the topological properties of the sheaf of the $\LL$-harmonic functions in \cite{BattagliaBonfiglioli};
 the rigidity inverse-mean-value theorem in \cite{AbbondanzaBonfiglioli};
 the non-homogeneous and invariant Harnack inequality for $\LL$ in \cite{BattBonf}.


 To this end, all that we have to do is to check that the axiomatic assumptions
 in  \cite{AbbondanzaBonfiglioli, BattagliaBonfiglioli, BattBonf, BLJems} are fulfilled in our case; this is contained in the following
 list (A.1)-to-(A.9). We are not interested in describing how any single axiom is involved in obtaining Theorems
 \ref{main-theorem}-to-\ref{theo.harnackWW}; the interested reader will find details in the mentioned papers.
 We only confine ourselves to a few remarks:
 the verification of axiom (A.8) is quite delicate, and it requires the upper estimates of $\Gamma$ and of its first derivatives
 given in Theorem \ref{thm.estimatesDERIVgamma}, together with the lower estimates of $\Gamma$ in Theorem \ref{th.stimebasso};
 the validity of a global Poincaré inequality as in axiom (A.9) is obviously of independent interest. 

 \bigskip

 \textbf{(A.1)}\,\,As a consequence of our assumption (H.1), all the $X_i$'s have null divergence; hence, $\LL$
 is a purely divergence-form operator, that is,
 $$\LL=\sum_{i,j=1}^n
    \de_{x_i}(a_{i,j}(x)\,\de_{x_j})=\text{div}(A(x)\,\nabla),$$
 where $A(x) = S(x)\cdot S(x)^T$ and $S(x)$ is the $n\times m$ matrix whose $i$-th column is $X_i(x)$. \medskip

 \textbf{(A.2)}\,\,$\LL$ admits a direction of strict ellipticity; in other words, there exists $i\in\{1,\ldots,n\}$ such that $a_{i,i}$ is a positive
 constant. This is another consequence of the $\dela$-homogeneity of $X_1,\ldots,X_m$ (see e.g., \cite[\S 1.3.1]{BLUlibro}).
 \medskip

 \textbf{(A.3)}\,\,$\LL$ is $C^\infty$-hypoelliptic on every open subset $\Omega$ of $\R^n$: this a consequence
 of the H\"ormander Hypoellipticity Theorem, which we are allowed to invoke
 owing to assumption (H.2).
 As a matter of fact, our $\LL$ has polynomial (hence, real analytic) coefficient
 functions (due to our homogeneity assumption), so that hypoellipticity is indeed equivalent
 to (H.2); see, e.g., \cite{Derridj}. \medskip

 \textbf{(A.4)}\,\,$\LL$ satisfies the so-called Regularity Axiom, namely, there exists a basis $\mathcal{B}$
 for the Euclidean topology of $\R^n$, whose elements are bounded open sets $\Omega$, such that
 the Dirichlet problem
\begin{equation*}
 \begin{cases}
  \LL u = 0 & \text{on $\Omega$}, \\
  u = \varphi & \text{on $\de\Omega$}
 \end{cases}
\end{equation*}
 admits a (unique) solution $u\in C^2(\Omega)\cap C(\overline{\Omega})$, for every $\varphi\in C(\de\Omega)$.
 This is true of any sums of squares of H\"ormander vector fields, as proved by Bony in the seminal paper
  \cite{Bony}. \medskip

 \textbf{(A.5)}\,\,$\LL$ satisfies the so-called Doob Convergence Axiom:
 if $\{u_k\}_k$
  is an increasing sequence of $\LL$-har\-mo\-nic functions on
 an open set $\Omega\subseteq \R^n$, then $u:=\sup_k u_k$ is
 $\LL$-harmonic in $\Omega$, provided that $u$ is finite in a dense
 subset of $\Omega$. This is a consequence of the Harnack inequality proved by
 Bony in \cite{Bony} for H\"ormander sums of squares (see also \cite{BBB}). \medskip

 \textbf{(A.6)}\,\,$\LL$ is endowed with a global fundamental solution $\Gamma$, with the
 following basic properties: $\Gamma > 0$ out of the diagonal;
 $\Gamma$ is locally integrable in $\R^n\times\R^n$; $\Gamma(x;\cdot)$ vanishes at
 infinity, for every fixed $x$. This is contained in our Theorem A. \medskip

 \textbf{(A.7)}\,\,$\Gamma(x;\cdot)$ has a pole at every fixed $x$.
 This is contained in Theorems \ref{th.stimebasso} and
 \ref{thm.estimateGamman2}.\bigskip

 We are now ready to derive from axioms (A.1)-to-(A.7) plenty of potential theoretic results
 for $\LL$ (see Theorems \ref{main-theorem} and \ref{th.Koebe}).
 In order to do this, as a crucial tool of this section, we replace the geometry of the CC-balls $B_X(x,r)$
 with the superlevel sets of $\Gamma(x;\cdot)$, that is,
 $$\Omega_r(x):=\big\{y\in\R^n\,:\,\Gamma(x;y)>1/r\big\}\cup\{x\},\qquad x\in\R^n,\,\,r>0.$$
 One of the greatest advantages of $\Omega_r(x)$ is that its boundary
  $$\de\Omega_r(x)=\big\{y\in\R^n\,:\,\Gamma(x;y)=1/r\big\}.$$
 is a smooth manifold of dimension $n-1$, at least for almost every $r$
 (by Sard's Lemma).

 For simplicity, we tacitly agree that all the statements of this section involving
 $\de\Omega_r(x)$ hold for those $r>0$ for which
 $\de\Omega_r(x)$ is smooth (hence, for almost every $r$).
 Let now $x\in\R^n$, and let us consider the functions, defined for $y\neq
 x$,
 $$\Gamma_x(y):=\Gamma(x;y),\qquad \mathcal{K}_x(y):=\frac{\sum_{j=1}^m |X_j\Gamma_x(y)|^2}{|\nabla \Gamma_x(y)|}.$$
 Let $u$ be upper semicontinuous on $\Omega$.
 For every fixed $\alpha>0$,  every $x\in\R^n$ and $r>0$ such that
 $\overline{\Omega_r(x)}\subset\Omega$, we introduce the following mean-value integral operators (here $H^{n-1}$ is the standard
 $(n-1)$-dimensional Hausdorff measure in $\R^n$):
\begin{align*}
  m_r(u)(x)&=\int_{\de\Omega_r(x)} u(y)\,\mathcal{K}_x(y)\,\d H^{n-1}(y),&
  M_r(u)(x)&=\frac{\alpha+1}{r^{\alpha+1}} \int_0^r\rho^\alpha\,m_\rho(u)(x)\,\d \rho.
\end{align*}
 An alternative representation of $M_r$ is the following one:
\begin{align}\label{defi-KKsolidoaltern}
  M_r(u)(x)&=\frac{\alpha+1}{r^{\alpha+1}}
  \int_{\Omega_r(x)}
  u(y)\,K^\alpha_x(y)\,\d y,\quad \text{where}\quad
 {K}^\alpha_x(y) :=\frac{\sum_{j=1}^m |X_j\Gamma_x(y)|^2}{\Gamma_x^{2+\alpha}(y)}.
\end{align}
 Furthermore,  for every
 $x\in\R^n$ and every $r>0$, we set
\begin{gather}\label{qQom}
\begin{split}
  q_r(x)&=\int_{\Omega_r(x)} \Big(\Gamma_x(y)-\frac{1}{r}\Big)\,\d y,\qquad Q_r(x)=\frac{\alpha+1}{r^{\alpha+1}}
  \int_0^r\rho^\alpha\,q_\rho(x)\,\d\rho,\\
  \omega_r(x)&=\frac{1}{\alpha\,r^{\alpha+1}}
  \int_{\Omega_r(x)}\big(r^{\alpha}-\Gamma^{-\alpha}_x(y)\big)\,\d y.
\end{split}
\end{gather}
\begin{remark}
 The above operators $m_r,M_r$ permit to obtain the following analogs of the classical
 Gauss-Green formulas for Laplace's operator
 (see, e.g., \cite{BLJems}):
\begin{align*}
  u(x)&=m_r(u)(x)-\int_{\Omega_r(x)}\Big(\Gamma_x(y)-\frac{1}{r}\Big)\,\LL u(y)\,\d y,\\
  u(x)&=M_r(u)(x)-\frac{\alpha+1}{r^{\alpha+1}}\int_0^r \rho^{\alpha}\bigg(
 \int_{\Omega_\rho(x)}\Big(\Gamma_x(y)-\frac{1}{\rho}\Big)\,\LL u(y)\,\d
 y\bigg)\d\rho,
\end{align*}
 holding true for every function $u$ of class $C^2$ on an open set
 containing $\overline{\Omega_r(x)}$.
\end{remark}
 From now on, $\Omega$ will denote an open set.
 An upper semicontinuous function $u:\Omega\to[-\infty,\infty)$ is called
  $\LL$-subharmonic in $\Omega$ if $u \not\equiv -\infty$ on every component of $\Omega$, and
 the following holds:
 for every bounded open set $V\subset \overline{V}\subset \Omega$ and for every
 $\LL$-harmonic function $h\in C^2(V)\cap C(\overline{V})$ such that $u\leq h$ on $\de
 V$, one has
  $u\leq h$ in $V$.

 For simplicity, the following result, providing
 characterizations of $\LL$-subharmonicity, is stated for continuous functions $u$, but it also holds
 for a u.s.c. function, with minor modification (see \cite{BLJems}); this is a consequence of axioms (A.1)-to-(A.7).
\begin{theorem}\label{main-theorem}
 Suppose that $u\in C(\Omega)$ and let $q_r,Q_r,\omega_r$ be as in \eqref{qQom}.
 Let also $$R(x):=\sup\{r>0\,:\,\,\Omega_r(x)\subseteq \Omega\}.$$
 Then, the following conditions are equivalent:
\begin{enumerate}
  \item  $u$ is $\LL$-subharmonic in $\Omega$.

 \item
 $\LL u\geq 0$ in the weak
 sense of distributions.

  \item $u(x)\leq m_r(u)(x)$, for every $x\in \Omega$ and $r\in (0,R(x))$.

  \item $u(x)\leq M_r(u)(x)$, for every $x\in \Omega$ and $r\in (0,R(x))$.

  \item
  $M_r(u)(x)\leq m_r(u)(x)$, for every $x\in \Omega$ and $r\in (0,R(x))$.

  \item
  $r\mapsto m_r(u)(x)$
  is monotone increasing on $(0,R(x))$, for every $x\in \Omega$.

\item
  $r\mapsto M_r(u)(x)$
  is monotone increasing on $(0,R(x))$, for every $x\in \Omega$.

  \item  For every $x\in \Omega$ it holds that
  \begin{equation*}\label{blaschke}
    \limsup_{r\to 0} \frac{m_r(u)(x)-u(x)}{q_r(x)}\geq 0.
  \end{equation*}

  \item  For every $x\in \Omega$ it holds that
  \begin{equation*}\label{privaloff}
    \limsup_{r\to 0} \frac{M_r(u)(x)-u(x)}{Q_r(x)}\geq 0.
  \end{equation*}

  \item
  For every $x\in \Omega$ it holds that
  \begin{equation*}\label{reade}
    \liminf_{r\to 0} \frac{m_r(u)(x)-M_r(u)(x)}{\omega_r(x)}\geq 0.
  \end{equation*}

\end{enumerate}
\end{theorem}
 As for the sheaf of the $\LL$-harmonic functions, we have the following result, obtained by collecting
 the axiomatic investigations in \cite{BattagliaBonfiglioli}: this provides a characterization
 of $\LL$-harmonicity, together with some topological properties of the vector space of
 the $\LL$-harmonic functions on an open set. This is again a consequence of axioms (A.1)-to-(A.7).
\begin{theorem}\label{th.Koebe}
 Suppose that $u\in C(\Omega)$. Then, the following conditions are equivalent:
\begin{enumerate}
  \item  $u$ is $C^\infty$ and is $\LL$-harmonic in $\Omega$.

  \item $u(x)= m_r(u)(x)$, for every $x\in \Omega$ and $r\in (0,R(x))$.

  \item $u(x)= M_r(u)(x)$, for every $x\in \Omega$ and $r\in (0,R(x))$.
\end{enumerate}
 Furthermore, we have the following compactness Montel-type result for $\LL$. Let $\mathcal{F}$ be a family of $\LL$-harmonic functions on $\Omega$ which is
 locally bounded, that is,
\begin{equation*}
 \sup_{f\in \mathcal{F}} \Big(\sup_{K}|f|\Big)<\infty,\quad \text{for every compact set $K\subset \Omega$.}
\end{equation*}
 Then $\mathcal{F}$ is a normal family, that is, for every
 sequence $\{f_n\}_n$ in $\mathcal{F}$, there exists a subsequence of $\{f_n\}_n$ which is
 uniformly convergent on the compact subsets of $\Omega$.

 Finally, the set of the $\LL$-harmonic functions
 in $\Omega$ endowed with the $L^1_\loc$-topology inherited from $C(\Omega)$
 \emph{(}or, equivalently, endowed with the $L^\infty_{\mathrm{loc}}$-topology\emph{)} is a Heine-Borel topological vector space.
\end{theorem}
 In order to have the next Theorem \ref{th.potato}, we need to verify another axiom: \medskip

 \textbf{(A.8)}\,\,There exists $\alpha > 0$ such that
  the sequence of functions
   $$f_k(x):=\int_{\Omega_{1/k}(x)}\Gamma(x;y)\,K^\alpha_0(y)\,\d y$$
  vanishes as $k\to \infty$, uniformly in $x$ (when $x$ lies in a compact set). \medskip

 Let us prove that (A.8) is fulfilled in our case, when $n > 2$, with the choice $\alpha > 2/(q-2)$.
 To begin with, by the very definition
 of $K^\alpha_0$, and owing to Theorems \ref{thm.estimatesDERIVgamma} and \ref{th.stimebasso}, one easily proves
 $$K^\alpha_0(y) \leq C\,\frac{\big|B_X\big(0,d_X(0,y)\big)\big|^\alpha}{d_X(0,y)^{2+2\alpha}}
   \stackrel{\eqref{propertiesdCC}}{=} |B_X(0,1)|^\alpha\,d_X(0,y)^{\alpha(q-2)-2}.$$
  Since, by assumption $\alpha(q-2)-2 > 0$, we conclude that $K^\alpha_0$ is bounded on any compact set.
  Then, if $x$ lies in a fixed compact set $F\subseteq\R^n$ and $y\in \Omega_{1/k}(x)$, one has
 $$
 \sup_{y\in\Omega_{1/k}(x)}K_0^\alpha(y) \leq M(F), \quad \text{for every $x\in F$ and every $k\in\N$}.
 $$
 Thus (A.8) is fulfilled if we show that
 \begin{equation} \label{eq.toprovealpostodiA8}
  \int_{\Omega_{1/k}(x)}\Gamma(x;y)\,\d y\longto 0\quad \text{as $k\to\infty$, uniformly for $x\in F$}.
 \end{equation}
 To prove this, we first remark that, due to Theorem B, if $y\in\Omega_{r}(x)$ one has
 \begin{align*}
  1/r < \Gamma(x;y) & \stackrel{\eqref{goal}}{\leq} C\,
   \frac{d_{X}(x,y)^{2}}
   {\big\vert B_{X}(x,d_{X}(x,y)) \big\vert} \\
 & \stackrel{\eqref{eq.NSWmodificata}}{\leq}
  \frac{C}{\gamma_1}\,\bigg(\sum_{h = n}^qf_h(x)\,d_X(x,y)^{h-2}\bigg)^{-1}
  \leq \frac{C}{f_q\gamma_1}\,d_X(x,y)^{2-q}.
 \end{align*}
  This shows that (with the structural choice $\theta^{q-2} := C/(f_q\gamma_1)$)
 \begin{equation} \label{eq.inclusionballsBIS}
   \Omega_r(x)\subseteq B_X\big(x,\theta\,r^{1/(q-2)}\big), \qquad\text{for any $x\in\R^n$ and any $r > 0$}.
 \end{equation}
  Due to \eqref{eq.inclusionballsBIS}, we derive \eqref{eq.toprovealpostodiA8} from Theorem \ref{thm.estimatesDERIVgamma} and the following result (with $p = 2$).
 \begin{lemma} \label{lem.integrabA8}
  For every $p > 0$, every $x\in\R^n$ and every $r > 0$, one has
 \begin{equation} \label{eq.integrabA8}
  \int_{B_X(x,r)}\frac{d_{X}(x,y)^{p}}
   {\big\vert B_{X}(x,d_{X}(x,y)) \big\vert}\,\d y \leq {c}_p\,r^p, \qquad \text{where $c_p = C_d\,\frac{2^p}{2^p-1}$},
\end{equation}
 and $C_d$ is the doubling constant in \eqref{globaldoubling}.
 \end{lemma}
 \begin{proof}
  We have the following argument, only based on the doubling inequality:
 \begin{align*}
   & \int_{B_X(x,r)}\frac{d_{X}(x,y)^{p}}
   {\big\vert B_{X}(x,d_{X}(x,y)) \big\vert}\,\d y = \sum_{k = 0}^\infty\int_{\frac{r}{2^{k+1}}\leq d_X(x,y) <\frac{r}{2^k}}
   \frac{d_{X}(x,y)^{p}}{\big\vert B_{X}(x,d_{X}(x,y)) \big\vert}\,\d y \\
   & \leq \sum_{k = 0}^\infty\frac{\big\vert B_{X}(x,r/2^k) \big\vert}{\big\vert B_{X}(x,r/2^{k+1}) \big\vert}\,\Big(\frac{r}{2^k}\Big)^p
  \leq C_d\,r^p\,\sum_{k = 0}^\infty\Big(\frac{1}{2^p}\Big)^k = c_p\,r^p.
 \end{align*}
 This completes the proof of \eqref{eq.integrabA8}.
 \end{proof}
 Thanks to the validity of axioms (A.1)-to-(A.8) in our framework (with $n > 2$), we have the following
 rigidity-type result for $\LL$ (also referred to as an inverse-mean-value theorem); see \cite{AbbondanzaBonfiglioli}.
\begin{theorem}\label{th.potato}
 Let $n > 2$. We choose $\alpha > 2/(q-2)$ and, with reference to \eqref{defi-KKsolidoaltern},
 we consider the measure $\d\mu(y) = K_0^\alpha(y)\,\d y$. Then \emph{(3)} in Theorem \ref{th.Koebe} gives
 $$
  u(0)=\frac{1}{\mu(\Omega_r(0))}
  \int_{\Omega_r(0)}
  u(y)\,\d \mu(y),
 $$
for every $u$ which is $\LL$-harmonic on a neighborhood of $\overline{\Omega_r(0)}$.

 Conversely, we have the following characterization of the superlevel sets $\Omega_r(0)$ of $\Gamma$.
 Suppose that $D$ is a bounded open neighborhood of $0$ such that
\begin{equation}\label{casopartic3}
 u(0)=\frac{1}{\mu(D)}
  \int_{D}
  u(y)\,\d \mu(y),
\end{equation}
 for every $u$ which is $\LL$-harmonic and $\mu$-integrable on $D$.
 Then $D=\Omega_r(0)$ for some $r>0$.

 More precisely, it suffices to
 suppose that \eqref{casopartic3} holds for the family of
 $\LL$-harmonic functions
 $$\Big\{D\ni y\mapsto \Gamma_x(y)\Big\}_{x\notin D}. $$
\end{theorem}
 As a last application, in order to have the following Harnack Theorem \ref{theo.harnackWW} we need
 to check the validity of the next axiom in our framework:\medskip

 \textbf{(A.9)}\,\,The global doubling inequality \eqref{globaldoubling} holds, together with the following result:
\begin{theorem}[Global Poincaré inequality]
 Let $X=\{X_1,\ldots,X_m\}$ satisfy axioms \emph{(H.1)} and \emph{(H.2)}.
 There exists a constant $C_P>0$ such that, for
  every $x\in \R^n$, $r>0$ and every $u$ which is $C^1$ in a neighborhood of $B_{X}(x,2r)$, one has
\begin{equation}\label{GDGlobalPoincareXX}
 \meanint_{B_X(x,r)}\Big|u(y)- u_{B_X(x,r)}\Big|\,\d y \leq C_P\,r\, \meanint_{B_{X}(x,2r)} \Big|
 X u(y)\Big|\,\d y.
\end{equation}
 As usual we have set $|Xu|:=\sqrt{\sum_{j=1}^m|X_j u|^2}$, and \emph{(}if $B$ is any $d_X$-ball\emph{)}
\begin{align*}
 &u_{B}:=\meanint_{B}u:=\frac{1}{|B|}\int_{B} u(y)\,\d y.
\end{align*}
\end{theorem}
 We have already proved, via a homogeneity argument, that the global doubling inequality \eqref{globaldoubling} holds
 in our case. The validity of \eqref{GDGlobalPoincareXX} follows likewise: indeed, from
 the results in \cite{HajlaszKoskela} one knows of the existence of a neighborhood $U_0$ of $0$, a constant $P_0>0$
 and $r_0>0$ such that
\begin{equation}\label{GDGlobalPoincareXXBIS}
 \int_{B_X(\xi,\rho)}\Big|v(y)- v_{B_X(\xi,\rho)}\Big|\,\d y \leq P_0\,\rho\, \int_{B_X(\xi,2\rho)} \Big|
 X v(y)\Big|\,\d y,
\end{equation}
 for every $\xi\in U_0$, $\rho\in [0,r_0]$ and $v\in C^1(\overline{B_{2\rho}(\xi)})$.
 Let $x\in\R^n$, $r>0$ and $u\in C^1(\overline{B_{X}(x,2r)})$; there certainly exists $0<\lambda\ll 1$ such that
 $\dela(x)\in U_0$ and $\lambda r<r_0$.
 It is easy to check (via \eqref{propertiesdCC}) that
\begin{equation*}
 \meanint_{B_X(x,r)}u=\meanint_{B_X(\dela(x),\lambda r)} v,\qquad \text{where $v=u\circ \delta_{1/\lambda}$.}
\end{equation*}
 The change of variable $y=\delta_{1/\lambda}(z)$ (applied twice) proves the following computation:
\begin{align*}
    &\int_{B_X(x,r)}\Big|u(y)- u_{B_X(x,r)}\Big|\,\d y=
  \lambda^{-q}\int_{B_X(\dela(x),\lambda r)}\Big|v-v_{B_X(\dela(x),\lambda r)}\Big| \,\d z\\
 &\stackrel{\eqref{GDGlobalPoincareXXBIS}}{\leq}
\lambda^{-q}\,P_0\,\lambda\,r\,\int_{{B_X(\dela(x),2\lambda r)}}|Xv|\,\d z
 = \lambda^{-q}\,P_0\,\lambda\,r\,\int_{{B_X(\dela(x),2\lambda r)}}\frac{1}{\lambda}\,|Xu|(\delta_{1/\lambda}(z))\,\d z\\
 &= P_0\,r\,\int_{{B_X(x,2r)}}\,|Xu|(y)\,\d y.
\end{align*}
 Summing up, we have proved that (for general $x,r,u$ as above)
\begin{equation*}
 \int_{B_X(x,r)}\Big|u(y)- u_{B_X(x,r)}\Big|\,\d y \leq P_0\,r\, \int_{B_X(x,2r)}
 \Big|X u\Big|\,\d y.
\end{equation*}
 If we divide both sides by $|B_X(x,r)|$ and if we apply the global doubling inequality \eqref{globaldoubling}, we get at once
 \eqref{GDGlobalPoincareXX} of axiom (A.9), with $C_P=P_0\,C_d$, where $C_d$
 is as in \eqref{globaldoubling}.\medskip

 Thanks to the axiomatic theory carried out in \cite{BattBonf}, due to axioms (A.1) and (A.9) only (plus our hypotheses (H.1) and (H.2)),
 we can deduce the following
 non-homogeneous and invariant Harnack inequality for $\LL$. For simplicity, we state it for
 classical solutions, but it also holds for $W^1_\loc$-weak solutions
 (in the sense of the Sobolev spaces associated with $X_1,\ldots,X_m$); see \cite{BattBonf}.
\begin{theorem}[Global scale-invariant Harnack inequality]\label{theo.harnackWW}
 Let $g\in L^p(\Omega)$, with $p>\frac{q}{2}$.
 There exists a structural constant $C_p>0$
 such that, for every ball $B_X(x,R)$ satisfying $\overline{B_{X}(x,4R)}\subset \Omega$, one has
\begin{equation}\label{theo.harnackWW.EQ1}
    \sup_{B_X(x,R)}u\leq C\,\left\{\inf_{B_X(x,R)}u+R^2
  \bigg(\meanint_{B_X(x,4R)} |g|^p\bigg)^{1/p}
 \right\},
\end{equation}
 for any nonnegative solution $u$ of $\LL u=g$ in $\Omega$.

 The invariance of the above Harnack inequality
 proves at once the classical Liouville theorem for $\LL$: if $\LL v=0$ on $\R^n$ and if $\inf_{\R^n} v>-\infty$, then
 $v$ is constant.
\end{theorem}

\section{Applications to Singular Integrals} \label{sec:meanvalue}
Starting with the representation formula
$$
 -\phi(  x)  =\int_{\mathbb{R}^{n}}\Gamma(  x;y)\,\mathcal{L}\phi(  y)\,  \d y,
$$
 for every $\phi\in C_{0}^{\infty}(  \mathbb{R}^{n})  $ and
 $x\in\mathbb{R}^{n}$ (see Theorem A), it is reasonable to expect that one can
 prove a representation formula for second order derivatives of $\phi$,
 involving a singular integral, of the kind:
\begin{equation}\label{repr formula}
 X_{i}X_{j}\phi(  x)  =\operatorname{PV}\int_{\mathbb{R}^{n}}
 X_{i}^{x}X_{j}^{x}\Gamma(  x;y)\,  (  -\mathcal{L}\phi)(  y)\, \d y +c_{i,j}(  x) \, \mathcal{L}\phi(  x),
\end{equation}
 valid for any $i,j=1,2,\ldots,m$, for suitable bounded functions $c_{i,j}$.
 This means that it
 is worthwhile studying the properties of the singular kernel
\begin{equation*}
 k (  x,y)  :=X_{i}^{x}X_{j}^{x}\Gamma(  x;y)\qquad \text{(for fixed $i,j\in\{1,\ldots,m\}$)}.
\end{equation*}
 Actually, in view of the global doubling condition established in
 \eqref{globaldoubling}, the space $\mathbb{R}^{n}$, endowed with the
 CC-distance of the vector fields $X_{1},\ldots,X_{m}$ and the Lebesgue measure,
 is (globally) a space of homogeneous type in the sense of Coifman-Weiss \cite{CW}.
  Hence, once we have established a suitable set of properties of
 this kernel, the continuity on $L^{p}(  \mathbb{R}^{n})  $ ($1<p<\infty$) of the operator
\begin{equation*}
 T:f\mapsto\operatorname{PV}
  \int_{\mathbb{R}^{n}}X_{i}X_{j}\Gamma(\,\cdot\,;y)\,  f(  y)\,   \d y
\end{equation*}
 should follow, hopefully, just through the application of some existing general abstract theory. This fact
 would provide global estimates in $W_{X}^{2,p}$ spaces for the operator
 $\mathcal{L}$.

 In this section we shall prove some properties of the kernel $k(x,y)  $
  which are relevant to this aim. We shall not prove, here, a
 representation formula (\ref{repr formula}), nor shall we develop the
 subsequent theory which would give the alluded Sobolev estimates. Actually,
 this material would overburden the present paper, and will be the subject of a
 separate paper. We just point out that this technique can also be used to
 prove global Sobolev estimates for ``nonvariational operators'' of the kind
$$
 Lu=\sum_{i,j=1}^{m}a_{i,j}(  x)\,  X_{i}X_{j}u,
$$
 where $(a_{i,j})_{i,j}$ is a symmetric, uniformly positive matrix of
 bounded coefficient functions, possessing some minimal regularity.\medskip

 For notational simplicity, in this section we write $d$ and $B$ instead of $d_X$ and $B_X$.
 The theory of singular integrals usually requires to check that both a kernel
 $k(  x,y)  $ and its adjoint $k(  y,x)  $ satisfy some
 pointwise and/or integral properties. The result we prove is the following:

\begin{theorem}[Properties of the singular kernel $k(  x,y)  $]\label{Thm prop sing kern}
 Let the assumptions and notation of the previous
 sections be in force and, for some fixed $i,j\in\{  1,2,\ldots,m\}  $, let
$$
 k(  x,y)  :=X_{i}^{x}X_{j}^{x}\Gamma(x;y)\qquad\text{for $x,y\in\mathbb{R}^{n}$, $x\neq y$.}
$$
 There exist constants $A,B,C>0$ such that:
\begin{enumerate}
  \item[\emph{(i)}]
   for every $x,y\in\mathbb{R}^{n}$ \emph{(}$x\neq y$\emph{)} one has
  $$\displaystyle | k(  x,y)  | +| k(  y,x)|
  \leq\frac{A}{\big| B(  x,d(  x,y)  )\big| };$$

  \item[\emph{(ii)}]
  for every $x,x_{0},y\in\mathbb{R}^{n}$ such that $d( x_{0},y) \geq 2\,d(  x_{0},x) >0$, it holds
  $$\displaystyle
 | k(  x,y)  -k(  x_{0},y)  |+| k(  y,x)  -k(  y,x_{0})  |
  \leq B\,\frac{d(  x_{0},x)  }{d(  x_{0},y)  }\cdot
  \frac{1}{\big| B(  x_{0},d(  x_{0},y)  )  \big| };$$

  \item[\emph{(iii)}]
  for every $z\in\mathbb{R}^{n}$ and $0<r<R<\infty$, one has
 $$ \bigg| \int_{\{r<d( z,y)  <R\}} k(  z,y)  \,\d y \bigg| +\bigg| \int_{\{r<d(  z,x)  <R\}} k(  x,z) \,  \d x \bigg| \leq C.$$
\end{enumerate}
\end{theorem}
 Inequalities (i)-(ii) are usually called the \emph{standard estimates} of
 singular kernels, while (iii) is a kind of \emph{cancelation property}, and
 is crucial in order to give sense to the principal value integral
 \eqref{repr formula}. These three properties are one of the possible sets of
 reasonable assumptions to prove that the singular integral operator with
 kernel $k$ is continuous on $L^{p}(  \mathbb{R}^{n})  $ for every
 $p\in(  1,\infty)  $.

 We explicitly note that the above theorem still holds in the case $n=2$ (since
 we are dealing with second order derivatives of $\Gamma$).\medskip

 For the proof, we need to use the following:
\begin{theorem}[``Lagrange Theorem''; \protect{\cite[Theorem  1.55]{BBbook}}]\label{Ch1-Thm Lagrange}
 Let $X_{1},\ldots,X_{m}$ be any system of H\"{o}r\-man\-der vector fields in a domain $\Omega\subseteq\mathbb{R}^{n}$,
 and let $f\in C^{1}(  \overline{B(  x_{0},r)  })  $, with $B(  x_{0},r)  \Subset\Omega$. Then
\begin{equation}\label{Ch1-Lagrange}
 | f(  x)  -f(  x_{0})  | \leq \sqrt{m}\,d(  x,x_{0}) \,
  \sup_{B(  x_{0},r)}|Xf|,\qquad  \text{for every $x\in B(  x_{0},r)$,}
\end{equation}
 where $| Xf| =\sqrt{\sum\limits_{i=1}^{m}|X_{i}f|^{2}}$.
\end{theorem}
 For the sake of completeness, we give the proof of this known result.
\begin{proof}
 Let $x\in B(  x_{0},r)$, $d(  x,x_{0})  =:\delta <r$.
  For every fixed $0<\varepsilon<r-\delta$, there exists a curve
 $\gamma(  t)  $ joining $x_{0}$ to $x$ such that
 $$
 \gamma^{\prime}(  t)  =\sum_{i=1}^{m}a_{i}(  t)\,X_{i}({\gamma(  t)}),
 $$
 with $| a_{i}(  t)  | \leq \delta+\varepsilon$ for any $i=1,\ldots,m$. Then
$$
 f(  x)  -f(  x_{0})  =\int_{0}^{1}\frac{\d}{\d t}(f( \gamma(  t)  )  )\, \d t=
 \int_{0}^{1}\sum_{i=1}^{m}a_{i}(  t)\,  X_{i}f  (  \gamma(t)  )\, \d t.
$$
 By definition of the CC-distance, all the points of the path $\gamma(t)  $
 are inside the ball $B(  x_{0},r)  $, hence
$$
 | f(  x)  -f(  x_{0})  | \leq 
 \sup_{t\in [0,1]}\sqrt{\sum_{i=1}^m|a_i(t)|^2}\cdot
 \sqrt{\sum_{i=1}^m|X_if(\gamma(t))|^2}\leq
 \sqrt{m}\,(  \delta+\varepsilon)\sup_{B_{X}(  x_{0},r)}| Xf|,
$$
 and, for the arbitrariness of $\varepsilon$, \eqref{Ch1-Lagrange} follows.
\end{proof}
 With this result at hand, let us pass to the
\begin{proof}[Proof of Theorem \ref{Thm prop sing kern}]
 By Theorem \ref{thm.estimatesDERIVgamma}, we have
 $$
 | k(  x,y)  | \leq\frac{c}{\big| B( x,d(  x,y)  )  \big| },
 $$
 for every $x,y\in\mathbb{R}^{n}$, $x\neq y$.
 Since, as already noted in Remark \ref{rem.simmetriafrac} using the global doubling property,
$$
 \big| B(  x,d(  x,y)  )  \big| \approx \big| B(  y,d(  x,y)  )  \big| ,
$$
 this implies (i).
 To prove (ii), we apply Theorem  \ref{Ch1-Thm Lagrange} to the following functions
  (here $y$ is fixed)
$$
 f_{1}(  x)   :=X_{i}^{x}X_{j}^{x}\Gamma(  x;y),\qquad f_{2}(  x)   :=X_{i}^{y}X_{j}^{y}\Gamma(y;x),
$$
 and we apply the upper bound on the third order derivatives of
 $\Gamma$ proved in Theorem  \ref{thm.estimatesDERIVgamma}.
 Next we take any $x,x_0\in\R^n$ such that  $d( x_{0},y) \geq 2\,d(  x_{0},x) >0$.
 Letting
$$
 r=\frac{3}{2}\,d(  x,x_{0}),
$$
 we have that $x\in B(  x_{0},r)  $ and, by \eqref{Ch1-Lagrange},
\begin{equation}\label{owingtosingkernel.EQ0}
 | k(  x,y)  -k(  x_{0},y)  |
 =| f_{1}(  x)  -f_{1}(  x_{0})  |
 \leq\sqrt{m}\,d(  x,x_{0})\,\sup_{z\in B(x_{0},r)  }| Xf_{1}(  z)  |,
\end{equation}
    where, owing to Theorem  \ref{thm.estimatesDERIVgamma}, for any $k\in\{1,\ldots,m\}$ and any $z\in B(  x_{0},r)$,
\begin{equation}\label{owingtosingkernel.EQ1}
     | X_{k}f_{1}(  z)  | =| X_{k}^{z}X_{i}^{z}X_{j}^{z}\Gamma(  z,y)  | \leq
 \frac{c}{d(z,y)\,\big| B(  z,d(  y,z)  )  \big| }.
\end{equation}
 Since $z\in B(  x_{0},r)$ with $r=\frac{3}{2}\,d(x,x_{0})$, and since $d(  x_{0},y)  \geq2\,d(  x_{0},x)=\frac{4}{3}r$,
  we infer that
$$
 \frac{1}{4}\,d(x_{0},y) \leq d(  z,y)  \leq \frac{7}{4}\,d(x_{0},y),
$$
hence, by the doubling condition and by \eqref{owingtosingkernel.EQ1}, we get
\begin{equation}\label{owingtosingkernel.EQ2}
 \sup_{z\in B(  x_{0},r)  }| Xf_1(  z) | \leq \frac{c'}{d(  x_{0},y) \, \big| B(x_{0},d(  x_{0},y)  )  \big| }.
\end{equation}
 If we insert \eqref{owingtosingkernel.EQ2} in \eqref{owingtosingkernel.EQ0}, we obtain
\begin{equation}\label{owingtosingkernel.EQ3}
 | k(  x,y)  -k(  x_{0},y)  |
  \leq \frac{c''\,d(  x,x_{0})}{d(  x_{0},y) \, \big| B(x_{0},d(  x_{0},y)  )  \big| }.
\end{equation}
Analogously, one can reproduce this argument for
 $$
 | k(  y,x)  -k(  y,x_{0})  |
 =| f_{2}(  x)  -f_{2}(  x_{0})  |, $$
 this time the upper bound on the mixed third order derivatives $X_{k}f_{2}(z)  =  X_{k}^{z}X_{i}^{y}X_{j}^{y}\Gamma(y;z)$ being needed.
 Thus we get
\begin{equation}\label{owingtosingkernel.EQ4}
 | k(  y,x)  -k(  y,x_{0})  |
  \leq \frac{c'''\,d(  x,x_{0})}{d(  x_{0},y) \, \big| B(x_{0},d(  x_{0},y)  )  \big| }.
\end{equation}
 Gathering together \eqref{owingtosingkernel.EQ3} and \eqref{owingtosingkernel.EQ4}, we obtain the proof of (ii).

 The proof of (iii) is inspired to \cite[Prop.\,5.23]{BBMP}.
 By the representation formula \eqref{derixxxx0}, we have
\begin{equation}\label{owingtosingkernel.EQ4.5}
 k(  x,y)  =X_{i}^{x}X_{j}^{x}\Gamma(x;y)=\int_{\mathbb{R}^{p}}
  (\widetilde{X}_{i}\widetilde{X}_{j}\Gamma_\G )\big((y,0)^{-1}*(x,\eta)\big)\,\d \eta.
\end{equation}
 Then, thanks to Fubini's Theorem,\footnote{We are entitled to interchange the order of integration: indeed,
 the summands $D^{R}$ with or without an absolute value in the integrand function
 can be estimated analogously; as for the summand $C^{r,R}$ (which is null without the absolute value in the integrand),
 if we insert an absolute value in its integrand function, we obtain an integral that can be upper-bounded by $c\,\log(R/r)<\infty$.} we have
\begin{align*}
 \int_{\{r<d(  x,z)  <R\}}k(  x,z)  \,\d x &  =
 \int_{\mathbb{R}^{p}}\int_{\{r<d(  x,z)  <R\}}
  (\widetilde{X}_{i}\widetilde{X}_{j}\Gamma_\G )\big((z,0)^{-1}*(x,\eta)\big)  \,\d x \, \d \eta\\
 &  =\int_{\big\{r<d_{\widetilde{X}}((z,0),(x,\eta)) <R\big\}}
  (\widetilde{X}_{i}\widetilde{X}_{j}\Gamma_\G)\big((z,0)^{-1}*(x,\eta)\big)  \,\d x \, \d \eta\\
 & +\int_{\big\{d_{\widetilde{X}}((z,0),(x,\eta))>R,\,\,d(x,z)<R\big\}}
 (\widetilde{X}_{i}\widetilde{X}_{j}\Gamma_\G)\big((z,0)^{-1}*(x,\eta)\big) \,\d x \,\d \eta\\
 & -\int_{\big\{d_{\widetilde{X}}((z,0),(x,\eta))>r,\,\,d(x,z)<r\big\}}
 (\widetilde{X}_{i}\widetilde{X}_{j}\Gamma_\G)\big((z,0)^{-1}*(x,\eta)\big) \,\d x \,\d \eta\\
 &=: C^{r,R}(  z)  +D^{R}(  z)  -D^{r}(z)  .
\end{align*}
 However, we know that the singular kernel $\widetilde{X}_{i}\widetilde{X}_{j}\Gamma_\G $
 on the Carnot group $\mathbb{G}$ satisfies the
 vanishing property (see \cite[Propositions 1.5, 1.8]{Fo2})
\begin{align*}
 &
 \int_{\big\{r<d_{\widetilde{X}}((z,0),(x,\eta))<R\big\}}
 (\widetilde{X}_{i}\widetilde{X}_{j}\Gamma_\G)\big((z,0)^{-1}*(x,\eta)\big)  \,\d x \,\d \eta\\
 &   =\int_{\big\{r<d_{\widetilde{X}}((0,0),(x,\eta))<R\big\}}
 (\widetilde{X}_{i}\widetilde{X}_{j}\Gamma_\G)(x,\eta)\,\d x\, \d \eta=0,\qquad \text{for any $R>r>0$},
\end{align*}
 so that $C^{r,R}(z)  \equiv0$.  Owing to Corollaries \ref{corollatastimaamam} and \ref{corollatastimaamambis44444}, we have
\begin{align*}
 | D^{R}(  z)  |  &  \leq  c\int_{\big\{d_{\widetilde{X}}((z,0),(x,\eta))>R,\,\,d(x,z)<R\big\}}
  d_{\widetilde{X}}((z,0),(x,\eta))^{-Q}\,\d x \,\d \eta\\
 &  \leq \frac{c}{R}
 \int_{\{d(x,z)\leq  R\}}
  \left(\int_{\mathbb{R}^{p}}
  d_{\widetilde{X}}((x,0),(z,\eta))^{-Q+1}\,\d \eta\right)\,\d x \\
 &  \leq \frac{c}{R}
 \int_{\{d(x,z)  \leq  R\}}
 \frac{d(x,z)}{\big| B(  x,d(  x,z)  ) \big|} \,\d z \leq \frac{c}{R}\, R=c,
\end{align*}
 where in the last inequality we have exploited Lemma \ref{lem.integrabA8} (with $p=1$).
 The estimate of $|D^{r}(z)|$ is the very same, replacing $R$ with $r$. Summing up, we have proved that
$$
\bigg| \int_{\{r<d(  x,z)  <R\}}k(  x,z)  \,\d x\bigg| \leq |D^{R}(z)|+|D^{r}(z)|\leq c,
$$
 with $c$ independent of $R,r,z$.

 To prove the analogous bound on the integral
 with respect to $y$, let us write
$$
 \int_{\{r<d(z,y)<R\}} k(z,y)\,\d y \stackrel{\eqref{owingtosingkernel.EQ4.5}}{=}
  \int_{\mathbb{R}^{p}}\int_{\{r<d(z,y)<R\}}
  (\widetilde{X}_{i}\widetilde{X}_{j}\Gamma_\G)\big((y,0)^{-1}* (z,\eta)\big)\,\d y \,\d \eta;
$$
 via the change of variable $\eta=\Phi_{y,z}(\zeta)$ as in Lemma \ref{lem.tecnico}, the latter integral is equal to
\begin{align*}
&  \int_{\mathbb{R}^{p}}\int_{\{r<d(  z,y)  <R\}}
 (\widetilde{X}_{i}\widetilde{X}_{j}\Gamma_\G)
  \big((y,u)^{-1} * (z,0)\big)\,\d y \,\d u\\
 &  =\int_{\mathbb{R}^{p}}\int_{\{r<d(z,y)  <R\}}
 \big((\widetilde{X}_{i}\widetilde{X}_{j}\Gamma_\G)\circ \iota\big) \big((z,0)^{-1}*(y,u)\big)\,\d y \,\d u.
\end{align*}
 Then we can proceed as above, exploiting the fact that the kernel
 $(\widetilde{X}_{i}\widetilde{X}_{j}\Gamma_\G) \circ\iota$
  also satisfies the vanishing property on spherical annuli in $\mathbb{G}$. So we are done.
\end{proof}
\begin{remark}\label{Remark final}
 Note that the proof of the above point (iii) also exploits
 the explicit representation that we have for $X_{i}X_{j}\Gamma$
 (in terms of the analogous function for the sublaplacian on the lifting Carnot group),
 and it does not simply follow from growth conditions on the derivatives of $\Gamma$.
 Also, note that the proof of point (ii) also depends on the estimate on the
 \emph{mixed} third order derivatives of $\Gamma$, which required some extra
 work to be proved, compared to pure derivatives.
\end{remark}
\section{Extension to H\"{o}rmander operators with drift} \label{sec.stimeDrift}
As we have already announced in Remark \ref{rem.drift}, most of the results in this
paper still hold for homogeneous H\"{o}rmander operators of the kind
$$
\mathcal{L}=\sum_{i=1}^{m}X_{i}^{2}+X_{0},
$$
i.e., possessing a `drift' term $X_{0}$. Let
us now restate our assumptions in the present context. \medskip

Throughout the sequel, we assume that
$X= \{X_{1},...,X_{m},X_{0}\}$ is a set
of smooth vector fields in $\R^n$ satisfying assumptions
(H.2)-(H.3), and we replace (H.1) with the following:
\begin{itemize}
\item[\textbf{(H.1)'}] 
There exists a family of dilations of the form \eqref{intro.dela} such
that $X_{1},...,X_{m}$ are $\delta_{\lambda}$-homogeneous of degree $1$, and
$X_{0}$ is $\delta_{\lambda}$-homogeneous of degree $2$. 
\end{itemize}
As in the previous sections, we set
$q := \txt\sum_{j = 1}^m\sigma_j$. By assumption (H.3), one has
$$q > 2.$$
\begin{example}
(1). In $\mathbb{R}^{2}$, the operator
$$
X_1^{2}+X_0 = \big(\de_{x_{1}}\big)^2 + x_1^{k}\,\de_{x_{2}}\qquad\text{(with
$k\in\mathbb{N}$)},
$$
is homogeneous of degree $2$ with respect to the dilations 
$\dela(x) = (\lambda x_{1},\lambda^{k+2}x_{2})$. \medskip

(2). In $\mathbb{R}^{n}$, the operator
$$
X_1^2+X_0 = \big(\de_{x_1}\big)^2
+ x_{1}\,\de_{x_2}+x_2\,\de_{x_3}+\ldots+x_{n-1}\,\de_{x_n},$$
is homogeneous of degree $2$ with respect to the dilations 
$$\delta_{\lambda}(x)=(\lambda x_{1},\lambda^{3}x_{2}, 
\lambda^{5}x_{3},\cdots,\lambda^{2n-1}x_{n}).$$
\end{example}
The lifting procedure described in Theorem A can be naturally
adapted to the present situation. More precisely,
by arguing essentially as in \cite{BB}, one can prove the following result.
\begin{theorem}
\label{ThmA_drift}
Assume that $X=\{X_0, X_{1},\ldots ,X_{m}\}$ satisfies assumptions
\emph{(H.1)'} and \emph{(H.2)}-\emph{(H.3)},
 of which we inherit the notation. 
 Then the following facts hold:\medskip
 
 \emph{(1)}\,\,There exists a \emph{graded but not stratified}
 homogeneous group 
 $\mathbb{G}=(\mathbb{R}^{N},\ast,D_{\lambda})$
\emph{(}in the sense of \cite{Fo2}\emph{)} of homogeneous dimension
$Q > q$, and there exists a system
$$\widetilde{X} = \{\widetilde{X}_{0},\widetilde{X}_{1},\ldots,\widetilde{X}_{m}\}$$ 
of
Lie-generators for $\mathrm{Lie}(\mathbb{G})$ such that 
\eqref{lifting} in Theorem A holds for every $j = 0,\ldots,m$.
In particular, $\widetilde{X}_1,\ldots,\widetilde{X}_m$ are $D_\lambda$-homogeneous of degree $1$,
while $\widetilde{X}_0$ is $D_\lambda$-homogeneous of degree $2$. \medskip

\emph{(2)}\,\,If $\widetilde{\Gamma}$ is the \emph{(}unique\emph{)} smooth fundamental
solution of $\sum_{i=1}^{m}\widetilde{X}_{i}^{2}+\widetilde{X}_{0}$ vanishing
at infinity con\-struc\-ted in \cite{Fo2}, then $\mathcal{L}$ admits a global
fundamental solution $\Gamma(x;y)$ under the form
$$\Gamma (x;y):=\int_{\mathbb{R}^{p}}\widetilde{\Gamma}
 \big((x,0);(y,\eta )\big)\, \d\eta \qquad (\text{for $x\neq y$
 in $\mathbb{R}^{n}$}).$$
 Furthermore,
if we set $\Gamma_\G := \widetilde{\Gamma}(0;\cdot)$, we have
$$\widetilde{\Gamma}(x,\xi;y,\eta) = \Gamma_\G\big((x,\xi)^{-1}*(y,\eta)\big),$$
and a formula analogous to \eqref{sec.one:mainThm_defGamma22222} holds. \medskip

\emph{(3)} $\Gamma$ is smooth out of the diagonal, it is nonnegative \emph{but not
strictly positive} \emph{(}as in Theorem A\emph{)}; it is locally integrable on
$\mathbb{R}^{n}\times\mathbb{R}^{n}$; it vanishes when $x$ or $y$ go to
infinity; it is jointly $\dela$-homogeneous of degree $2-q<0$, that is,
\eqref{sec.one:mainThm_defGamma3} holds. \medskip

\noindent Finally, $\Gamma$ is not symmetric but
$\Gamma^{\ast}(x;y)  =\Gamma(y;x)$
is the fundamental solution of
$$
\mathcal{L}^{\ast}=\sum_{i=1}^{m}X_{i}^{2}-X_{0},
$$
and enjoys analogous properties of $\Gamma$.
\end{theorem}
\begin{remark} \label{rem.Ladjoint}
 By replacing $X_{0}$ with $Y = -X_{0}$, we see that 
 $\mathcal{L}^{\ast}$ satisfies properties analogous to $\mathcal{L}$. 
 In particular, if $\widetilde{\Gamma}^{\ast}$ is the fundamental solution of 
 $\sum_{i=1}^{m}\widetilde{X}_{i}^{2}-\widetilde{X}_{0}$, then the function 
 $\Gamma^{\ast}$ can be obtained by saturating 
 $\widetilde{\Gamma}^{\ast}$ as in 
 \eqref{sec.one:mainThm_defGamma}. Furthermore, setting
 $\Gamma_\G^* = \Gamma^*(0;\cdot)$, we have
 $$
\Gamma_{\mathbb{G}}^{\ast}\big((x,\xi)^{-1}\ast(y,\eta)\big) =
 \widetilde{\Gamma}^{\ast}(x,\xi;y,\eta).$$
\end{remark}
Let us now check how the results about the geometry of vector fields we
have reviewed and adapted in Section \ref{sec:notations.review} can be extended to the
drift case. 
First of all, the definition of control distance must be adapted to this
situation, as already done in \cite{NSW}:
\begin{equation} \label{eq.defdCC_drift}
d_{X}(x,y):=\inf\left\{  r>0:\,\text{there exists $\gamma\in C(r)$ with
$\gamma(0)=x$ and $\gamma(1)=y$}\right\}  ,%
\end{equation}
where $C(r)$ is the set of the absolutely continuous maps 
$\gamma:[0,1]\rightarrow\mathbb{R}^{n}$ satisfying (a.e.\,on $[0,1]$)
$$
\gamma'(t)=\sum_{j=0}^{m}a_{j}(t)\,X_{j}(\gamma(t)), \qquad
\text{with $|a_{0}(t)|\leq r^{2}$ and $|a_{j}(t)|\leq r$ for all $j=1,\ldots,m$}.
$$
We also have to replace the notion of \emph{length} of a commutator with that
of \emph{weight} of a commutator; this is another fact which is by now standard
after \cite{NSW}. For a multi-index
$$
I=(i_{1},\ldots,i_{k}),\qquad\text{with $i_{1},\ldots,i_k
 \in\{0,1,2,\ldots,m\}$},$$
and a commutator
$$
X_{[I]}:=\left[  \left[  \left[  X_{i_{1}},X_{i_{2}}\right]  ,X_{i_{3}%
}\right]  ,\ldots,X_{i_{k}}\right]
$$
we define the weight of $I$ as
$$
|I| = \sum_{j=1}^{k} p_{i_j}, \qquad \text{where $p_0 = 2$
and $p_{i}=1$ for $i=1,2,...,m$}.
$$
With these modifications, 
all the properties stated in Section
\ref{sec:notations.review} still hold with the same statements and
the same proofs, as they rely on the theory developed in \cite{NSW}
(which covers also the drift case). \medskip

Now, all the properties of $\Gamma$ that we have proved in this paper
throughout Sections 4-to-6 are consequences of the geometric properties
stated in Section \ref{sec:notations.review} and of the properties of $\Gamma$
stated in Theorem A, \emph{with no reference to the explicit form of
the operator} $\mathcal{L}$.
As a consequence, all the pro\-per\-ties proved in Sections 4-to-6 not depending on
the \emph{symmetry} or on the \emph{strict positivity}
 of $\Gamma$ (which are the only basic properties of $\Gamma$ not extending
  to the drift case, see Theorem \ref{ThmA_drift}) still hold under the assumptions of the present
section, with the same proof. \vspace*{0.07cm}

The strict positivity of $\Gamma$ has been used in the proof of the estimates
from below on $\Gamma$ (and only to this aim); thus, these estimates do not
extend to the the drift case.
The symmetry of $\Gamma$ has been used only once, in the proof
of Lemma \ref{Thm 0}, formula \eqref{derixxxx0}. An easy 
variation of such a proof leads
to a slight modification of this identity:
by Remark \ref{rem.Ladjoint}, and using the same notation, one has
$$
\Gamma(x,y)=\Gamma^{\ast}(y,x)  =
\int_{\mathbb{R}^{p}}\Gamma_{\mathbb{G}}^{\ast}\Big((y,0)^{-1}\ast(x,\eta)\Big)d\eta
$$
as a consequence, 
$$
X_{j_{1}}^{x}\cdots X_{j_{t}}^{x}\big(\Gamma(\cdot;y)\big)(x)=\int%
_{\mathbb{R}^{p}}\big(\widetilde{X}_{j_{1}}\cdots\widetilde{X}_{j_{t}}%
\Gamma_{\mathbb{G}}^{\ast}\big)\Big((y,0)^{-1}\ast(x,\eta)\Big)d\eta.
$$
Since the function $\Gamma_{\mathbb{G}}^{\ast}$ has the same homogeneity and
smoothness properties of $\Gamma_{\mathbb{G}}$, the subsequent arguments
proceed unchanged, and we have the following theorem.
\begin{theorem} \label{th.teoremone_drift} 
Let $\mathcal{L}=\sum_{j=1}^{m}X_{j}^{2}+X_{0}$
satisfy assumptions \emph{(H1)'}-\emph{(H2)}-\emph{(H3)}. Moreover, let 
$\Gamma,\,\Gamma_{\mathbb{G}}$ and $\Gamma_{\mathbb{G}}^{\ast}$ be as above. 
Then the following
facts hold.\medskip

\emph{(I).} For any $s,t\geq1$, and any choice of indexes 
$i_{1},\ldots,i_{s},j_{1},\ldots,j_{t}\in\{0,\ldots,m\}$, the following representation formulas
hold true for $x\neq y$ in $\mathbb{R}^{n}$:
\begin{align*}
&  X_{i_{1}}^{y}\cdots X_{i_{s}}^{y}\big(\Gamma(x;\cdot)\big)(y)=\int%
_{\mathbb{R}^{p}}\Big(\widetilde{X}_{i_{1}}\cdots\widetilde{X}_{i_{s}}%
\Gamma_{\mathbb{G}}\Big)\Big((x,0)^{-1}\ast(y,\eta)\Big)d\eta;\\[0.2cm]
&  X_{j_{1}}^{x}\cdots X_{j_{t}}^{x}\big(\Gamma(\cdot;y)\big)(x)=\int%
_{\mathbb{R}^{p}}\Big(\widetilde{X}_{j_{1}}\cdots\widetilde{X}_{j_{t}}%
\Gamma_{\mathbb{G}}^{\ast}\Big)\Big((y,0)^{-1}\ast(x,\eta)\Big)\,d\eta
;\\[0.2cm]
&  X_{j_{1}}^{x}\cdots X_{j_{t}}^{x}X_{i_{1}}^{y}\cdots X_{i_{s}}^{y}%
\Gamma(x;y)\\
&  \qquad\qquad=\int_{\mathbb{R}^{p}}\bigg(\widetilde{X}_{j_{1}}%
\cdots\widetilde{X}_{j_{t}}\Big(\big(\widetilde{X}_{i_{1}}\cdots
\widetilde{X}_{i_{s}}\Gamma_{\mathbb{G}}\big)\circ\iota
\Big)\bigg)\Big((y,0)^{-1}\ast(x,\eta)\Big)d\eta.
\end{align*}
Here $\iota$ denotes the inversion map of the Lie group $\mathbb{G}$.\medskip

\emph{(II).} Given any $Z_1,\ldots,Z_r\in\big\{X_{0}^{x},\ldots,X_{m}^{x},X_{0}^{y}%
,\ldots,X_{m}^{y}\big\}$, we define  \medskip
\begin{align*}
 & \mathcal{Z} := Z_1\cdots Z_r; \\[0.15cm]
 & |\mathcal{Z}| := 
\sum_{k = 1}^r|Z_k|, \qquad\text{where $|Z_k| := 
\begin{cases}
2, & \text{if $Z\in \{X_0^x,X_0^y\}$}, \\
1, & \text{otherwise}
\end{cases}$}
\end{align*}
Then, there exists a constant $C > 0$
$$
\Big\vert \mathcal{Z}\Gamma(x;y)\Big\vert\leq C\,\frac
{d_{X}(x,y)^{2-|\mathcal{Z}|}}{\big\vert B_{X}(x,d_{X}(x,y))\big\vert},
$$
for any $x\neq y\in\mathbb{R}^{n}$.
In par\-ti\-cu\-lar, for every fixed pole $x\in\mathbb{R}^{n}$
we have
$$
\lim_{|y|\rightarrow\infty}\mathcal{Z}\Gamma(x;y)=0.
$$

\emph{(III).} Suppose that $n>2$. Then one has
$$0\leq
\Gamma(x;y)\leq C\,\frac{d_{X}(x,y)^{2}}{\big\vert B_{X}(x,d_{X}%
(x,y))\big\vert},
$$
for any $x,y\in\mathbb{R}^{n}$ \emph{(}with $x\neq y$\emph{)}. Here $C\geq1$
is a structural constant.\medskip

\emph{(IV).} Suppose that $n=2$. For every compact set $K\subseteq
\mathbb{R}^{n}$ there exist a structural constant $c > 0$
and a real number $R>0$ \emph{(}all depending on $K$\emph{)} such
that
$$0\leq
\Gamma(x;y)\leq c\,\frac{d_{X}(x,y)^{2}}{\big|B_{X}(x,d_{X}(x,y))\big|}%
\cdot\log\Big(\frac{R}{d_{X}(x,y)}\Big),
$$
uniformly for $x\neq y$ in $K$. Moreover, for every fixed pole $x\in
\mathbb{R}^{n}$, there exists a constant $\gamma(x)>0$ and $0<\varepsilon
(x)<1$ such that
\[
\Gamma(x;y)\leq\gamma(x)\,F(x,y),
\]
for any $y$ such that $0<d_{X}(x,y)<\varepsilon(x)$.
Here, $F(x,y)$ is as in Theorem \ref{th.teoremone}.
\end{theorem}

\end{document}